\documentclass{article}

\usepackage[english]{babel}\usepackage[latin1]{inputenc}\usepackage[T1]{fontenc}
\usepackage[table,dvipsnames]{xcolor}
\usepackage{mathdots,graphicx,amsfonts,amssymb,amsmath,amsthm,mathrsfs,epstopdf,enumitem,multirow,booktabs,subcaption,mathtools,accents,pifont}
\usepackage[top=2cm,bottom=3cm,left=2cm,right=2cm]{geometry}
\usepackage{sectsty}\sectionfont{\normalsize}\subsectionfont{\normalsize}
\usepackage[labelfont=bf]{caption}

\newtheorem{theorem}{Theorem}[section]
\newtheorem{lemma}{Lemma}[section]
\newtheorem{corollary}{Corollary}[section]
\theoremstyle{definition}
\newtheorem{definition}{Definition}[section]
\newtheorem{remark}{Remark}[section]
\newtheorem{example}{Example}[section]

\numberwithin{equation}{section}
\numberwithin{figure}{section}
\numberwithin{table}{section}

\newcommand{\bu}{{\mathbf 1}}
\newcommand{\bz}{{\mathbf 0}}
\newcommand{\ee}{{\boldsymbol e}}
\newcommand{\ii}{{\boldsymbol i}}
\newcommand{\jj}{{\boldsymbol j}}
\newcommand{\hh}{{\boldsymbol h}}
\newcommand{\kk}{{\boldsymbol k}}
\newcommand{\nn}{{\boldsymbol n}}
\newcommand{\mm}{{\boldsymbol m}}

\newcommand{\xx}{{\boldsymbol x}}
\newcommand{\btheta}{\boldsymbol\theta}
\newcommand{\bnu}{\boldsymbol\nu}
\newcommand{\II}{{\mathcal I}}
\newcommand{\JJ}{{\mathcal J}}
\newcommand{\rie}{RI\hspace{-1.25pt}E}
\renewcommand{\epsilon}{\varepsilon}

\newcommand{\rd}{\color{red}}
\newcommand{\bl}{\color{blue}}

\begin{document}

\title{Analysis of Block Jacobi/Gauss--Seidel and additive/multiplicative Schwarz preconditioning through the theory of GLT sequences,\\ with applications to domain decomposition discretizations}
\author{Carlo Garoni\thanks{Corresponding author}\\
\footnotesize Department of Mathematics, University of Rome Tor Vergata, Rome, Italy (garoni@mat.uniroma2.it)\\[10pt]
Abdessadek Rifqui\\
\footnotesize Vanguard Center, University Mohammed VI Polytechnic, Rabat Morocco (abdessadek.rifqui@um6p.ma)\\[10pt]
Stefano Serra-Capizzano\\
\footnotesize Department of Science and High Technology, University of Insubria, Como, Italy (s.serracapizzano@uninsubria.it)\\[-2pt]
\footnotesize Department of Information Technology, Uppsala University, Uppsala, Sweden (stefano.serra@it.uu.se)}
\date{}

\maketitle

\begin{abstract}
When a linear differential problem is discretized by a linear numerical method characterized by a mesh fineness parameter $n$, the computation of the numerical solution reduces to solving a linear discrete problem identified by a matrix $A_n$ whose size grows with $n$. Regardless of the considered differential problem and numerical method, the sequence of discretization matrices $\{A_n\}_n$ often falls within the class of generalized locally Toeplitz (GLT) sequences. This happens even if the selected numerical method belongs to the family of domain decomposition methods (DDMs), as illustrated herein through examples.
Four widely used preconditioners for DDM discretization matrices are the block Jacobi (BJ), block Gauss--Seidel (BGS), additive Schwarz (AS), and multiplicative Schwarz (MS) preconditioners. In this paper, we provide formal definitions of the BJ/BGS/AS/MS preconditioners for arbitrary multilevel block matrices. These definitions and the associated notations are inspired by the theory of GLT sequences and are proposed as alternatives to those commonly used by the DDM community.
After formulating the necessary definitions, we analyze through the theory of GLT sequences the structure of the BJ/BGS/AS/MS preconditioners when applied to multilevel block matrices $A_n$ belonging to a GLT sequence $\{A_n\}_n$.
To explain our main results, we point out that any GLT sequence $\{A_n\}_n$ is uniquely associated with a special function $\kappa$ called symbol.
In this paper, we prove that, if $\{A_n\}_n$ is a GLT sequence with symbol $\kappa$ and if $P_n^{\hspace{1pt}\mbox{\scriptsize\ding{105}}}(A_n)$ is a $\mbox{\ding{105}}$ preconditioner for $A_n$ with $\mbox{\ding{105}}\in\{BJ,BGS,AS,MS\}$, then: (a)~$\{P_n^{\hspace{1pt}\mbox{\scriptsize\ding{105}}}(A_n)\}_n$ is a GLT sequence with symbol $\kappa$ for $\mbox{\ding{105}}\in\{BJ,BGS,MS\}$; (b)~$\{P_n^{AS}(A_n)\}_n$ is a GLT sequence with symbol $\kappa^{AS}\approx\kappa$, and $\kappa^{AS}=\kappa$ whenever the overlaps in the subdomains used for the construction of $P^{AS}(A_n)$ vanish as $n\to\infty$.
Results (a) and (b) give reasons to believe that the BJ/BGS/AS/MS preconditioners are very efficient for GLT sequences.
A numerical validation of this prediction in the context of isogeometric DDMs is presented.

\smallskip


\smallskip

\end{abstract}

{\footnotesize\tableofcontents}

\section{Introduction}\label{sec:i}

Consider the discretization of a linear differential problem using a numerical method characterized by a mesh fineness parameter $n$.
In this case, the computation of the numerical solution reduces to solving a linear discrete problem (usually a linear system) identified by a matrix $A_n$.
The size of $A_n$ grows as $n$ increases, i.e., as the mesh is progressively refined, and ultimately we are left with a sequence of matrices $A_n$ whose size diverges to $\infty$ as $n\to\infty$.
What is often observed in practice is the following:
\begin{quote}
{\em Regardless of the considered differential problem and the numerical method used for its discretization, the sequence of discretization matrices $\{A_n\}_n$ possesses a sort of (possibly hidden) Toeplitz-like structure and, in particular, it falls within the class of generalized locally Toeplitz (GLT) sequences.}
\end{quote}
This statement alone justifies the interest in the theory of GLT sequences, which nowadays represents a quite extensive research area with numerous applications.
For readers who are new to the subject, we recommend the introduction \cite{GLT-intro} and the six-page conference paper \cite{aip}.
For a comprehensive exposition of the topic, we refer to the books \cite{GLTbookI,GLTbookII} and the book-like papers \cite{rg,bgR,bg,bgd}.

Domain decomposition methods (DDMs) are among the most powerful and widely used parallel numerical methods for the discretization of differential equations.
By combining ideas from variational analysis, numerical linear algebra and iterative methods, they exploit a suitable decomposition of the computational domain into smaller subdomains that can be treated in parallel.
The origin of DDMs can be traced back to Schwarz's rigorous proof of the Dirichlet principle in the context of Riemann's theory of analytic functions~\cite{Sch1870}. This seminal work introduced the classical Schwarz method, thus laying the foundation of modern DDM theory.
A key advancement that followed \cite{Sch1870} was the variational reformulation of Schwarz methods~\cite{Lio88I, Lio88II, Lio88III}, which enabled the development of parallel and optimized Schwarz algorithms.
Since then, DDMs have evolved into a broad family of methods, including additive and multiplicative Schwarz methods, Neumann--Neumann methods, finite element tearing and interconnecting (FETI) methods, balancing domain decomposition (BDD) methods, and optimized Schwarz methods. 
We refer the reader to~\cite{CM94,DDM-book-New,DDM-book} for a comprehensive exposition of DDMs and to \cite{Gan06} for more recent developments regarding the optimization of Schwarz methods.
In accordance with the statement highlighted above in italics, {\em the sequences of matrices $\{A_n\}_n$ arising from DDM discretizations often fall within the class of GLT sequences.} One of the purposes of the present paper is to support this claim through a number of examples; see Section~\ref{sec:DDMs-GLT}.

Four widely used preconditioners for DDM discretization matrices are the block Jacobi (BJ), block Gauss--Seidel (BGS), additive Schwarz (AS), and multiplicative Schwarz (MS) preconditioners. It is no coincidence that the names of the latter two preconditioners coincide with the names of two DDMs mentioned above.
In fact, all DDMs mentioned above can be seen as stationary iterative solvers for linear systems and, more importantly in large-scale computations, as preconditioners for Krylov methods. When applied to DDM discretization matrices, the BJ/BGS/AS/MS preconditioners naturally exploit the block structure induced by the domain decomposition to construct efficient approximations of the inverse operator.
While the BJ/AS preconditioners compute independent local corrections and are therefore well suited for parallel computing, the BGS/MS preconditioners apply these corrections successively, generally leading to stronger preconditioning at the expense of reduced parallelism.
It was soon discovered that the AS and MS preconditioners can be interpreted as, respectively, the BJ and BGS preconditioners with the addition of subdomain overlap. In particular, in the case of non-overlapping subdomains, the AS and MS preconditioners coincide with the BJ and BGS preconditioners, respectively.
Owing to their robustness, scalability and effectiveness in accelerating Krylov methods, the BJ/BGS/AS/MS preconditioners have become essential tools in scientific computing. Their mathematical foundations, convergence theory and practical implementation have been extensively investigated; see \cite{DDM-book-New,Gan08,DDM-book} and the references therein.

The present paper focuses on multilevel block matrices, which are the type of matrices that arise from the discretization of linear differential problems over rectangular domains; see Section~\ref{min} for details.
For this type of matrices, we provide formal definitions of the corresponding BJ/BGS/AS/MS preconditioners.
The definitions provided herein, as well as the associated notations, are formulated on an abstract level and apply to arbitrary multilevel block matrices, not necessarily arising from numerical discretizations. Such definitions and notations, which are inspired by the theory of reduced GLT sequences \cite{rg}, are proposed as alternatives to those commonly used by the DDM community. In order for the reader to become familiar with them, a number of illustrative examples is provided.
After formulating the necessary definitions, we analyze through the theory of GLT sequences the structure of the BJ/BGS/AS/MS preconditioners when applied to multilevel block matrices $A_n$ belonging to an arbitrary GLT sequence $\{A_n\}_n$. To explain our main results, we first point out that any GLT sequence $\{A_n\}_n$ is uniquely associated with a special function $\kappa$ called symbol. We write $\{A_n\}_n\sim_{\rm GLT}\kappa$ to indicate that $\{A_n\}_n$ is a GLT sequence with symbol $\kappa$. The importance of $\kappa$ lies in the fact that it always describes the asymptotic singular value distribution of $\{A_n\}_n$ and, in many cases of interest, it also describes the asymptotic spectral distribution of $\{A_n\}_n$; see Sections~\ref{sec:distr} and~\ref{sec:GLT} for details.
Let us denote by $P^{BJ}(A_n)$, $P^{BGS}(A_n)$, $P^{AS}(A_n)$, $P^{MS}(A_n)$, respectively, the BJ, BGS, AS, MS preconditioner for $A_n$.
We remark that all these preconditioners are constructed from a decomposition of the domain $[0,1]^d$ into a number of not necessarily disjoint subdomains whose union is equal to $[0,1]^d$, where $d$ is the number of levels in the multilevel structure of $A_n$; see Section~\ref{bJ/bGS/AS/MS-def-e} for details.
In our main results (Theorems~\ref{blockdiag(GLT)=GLT_blocktril(GLT)=GLT} and~\ref{AS(GLT)=GLT_MS(GLT)=GLT}), we show that:
\begin{itemize}[nolistsep,leftmargin=*]
	\item If $\{A_n\}_n\sim_{\rm GLT}\kappa$ then $\{P^{\hspace{1pt}\mbox{\scriptsize\ding{105}}}(A_n)\}_n\sim_{\rm GLT}\kappa$ for $\mbox{\ding{105}}\in\{BJ,BGS,MS\}$. In other words, the BJ/BGS/MS preconditioners, when applied to a matrix $A_n$ belonging to a GLT sequence with symbol $\kappa$, preserves the GLT structure of $A_n$ and also the symbol.
	\item If $\{A_n\}_n\sim_{\rm GLT}\kappa$ then $\{P^{AS}(A_n)\}_n\sim_{\rm GLT}\kappa^{AS}$, where $\kappa^{AS}$ is approximately equal to $\kappa$ and $\kappa^{AS}=\kappa$ if and only if the overlaps in the subdomains used for the construction of $P^{AS}(A_n)$ vanish as $n\to\infty$. In other words, the AS preconditioner, when applied to a matrix $A_n$ belonging to a GLT sequence with symbol $\kappa$, preserves the GLT structure of $A_n$ and approximately also the symbol; it exactly preserves the symbol if and only if the subdomain overlaps vanish as $n\to\infty$. 
\end{itemize}
From a practical point of view, Theorems~\ref{blockdiag(GLT)=GLT_blocktril(GLT)=GLT} and~\ref{AS(GLT)=GLT_MS(GLT)=GLT} give reasons to believe that the BJ/BGS/AS/MS preconditioners for matrices $A_n$ belonging to a GLT sequence are very efficient; see Remark~\ref{P-efficiency}. A numerical experiment in support of this claim has been included in Section~\ref{sec:app_discussion}.
We remark that, while proving Theorems~\ref{blockdiag(GLT)=GLT_blocktril(GLT)=GLT} and~\ref{AS(GLT)=GLT_MS(GLT)=GLT}, we also prove several auxiliary results, three of which are significant enough to be considered as further main results of this paper in addition to Theorems~\ref{blockdiag(GLT)=GLT_blocktril(GLT)=GLT} and~\ref{AS(GLT)=GLT_MS(GLT)=GLT}; we expressly refer to Theorems~\ref{rie-lemma}, \ref{S->JGS}, \ref{rie-thm}.

The paper is organized as follows.
In Section~\ref{min}, we introduce the multi-index notation used throughout the paper along with the notion of multilevel block matrices.
In Section~\ref{sec:overview}, we overview the theory of GLT sequences, with a focus on the results that we need in this paper.
In Section~\ref{sec:DDMs-GLT}, we illustrate through examples that sequences of matrices arising from DDM discretizations are often GLT sequences.
In Section~\ref{bJ/bGS/AS/MS-def-e}, we formally define the BJ/BGS/AS/MS preconditioners for multilevel block matrices and we illustrate the definitions through examples.
In Section~\ref{sec:prel}, we collect the necessary preliminaries to prove our main results.
In Section~\ref{sec:main}, we state and prove our main results described above (Theorems~\ref{blockdiag(GLT)=GLT_blocktril(GLT)=GLT} and~\ref{AS(GLT)=GLT_MS(GLT)=GLT}).
In Section~\ref{sec:app_discussion}, we discuss the applications of the main results in the context of GLT preconditioning; we also provide a numerical example from isogeometric DDMs that illustrates the effectiveness of the BJ/BGS/AS/MS preconditioners for GLT sequences.
In Section~\ref{sec:c}, we draw conclusions and suggest future lines of research.

\section{Multi-index notation and multilevel block matrices}\label{min}

A multi-index $\ii$ of length $d$, also called a $d$-index, is a vector in $\mathbb Z^d$; its components are denoted by $i_1,i_2,\ldots,i_d$.
$\bz$,~$\bu$,~$\mathbf2$ are the vectors of all zeros, all ones, all twos, respectively (their size will be clear from the context).
If $\mm$ is a $d$-index, we set $N(\mm)=m_1m_2\cdots m_d$.
A $d$-index $\mm$ is said to be positive if its components are positive, i.e., if $\mm\in\mathbb N^d$. 
If $\{\mm=\mm(n)\}_n$ is a sequence of positive $d$-indices, we say that $\mm\to\infty$ as $n\to\infty$ if $\min(\mm)\to\infty$ as $n\to\infty$.
If $\hh,\kk\in\mathbb R^d$, an inequality such as $\hh\le\kk$ means that $h_i\le k_i$ for all $i=1,\ldots,d$.
If $\hh,\kk$ are $d$-indices such that $\hh\le\kk$, the $d$-index range $\{\hh,\ldots,\kk\}$ is the set $\{\ii\in\mathbb Z^d:\hh\le\ii\le\kk\}$. We assume for this set the standard lexicographic ordering:
\[ \Bigl[\ \ldots\ \bigl[\ [\ (i_1,\ldots,i_d)\ ]_{i_d=h_d,\ldots,k_d}\ \bigr]_{i_{d-1}=h_{d-1},\ldots,k_{d-1}}\ \ldots\ \Bigr]_{i_1=h_1,\ldots,k_1}. \]
For instance, in the case $d=2$, the ordering is
\begin{align*}
&(h_1,h_2),\,(h_1,h_2+1),\,\ldots,\,(h_1,k_2),\,(h_1+1,h_2),\,(h_1+1,h_2+1),\,\ldots,\,(h_1+1,k_2),\\
&\ldots\,\ldots\,\ldots,\,(k_1,h_2),\,(k_1,h_2+1),\,\ldots,\,(k_1,k_2).
\end{align*}
When a $d$-index $\ii$ varies in a finite set $\II\subset\mathbb Z^d$ (this is simply written as $\ii\in\II$), it is understood that $\ii$ follows the lexicographic ordering. For instance, if $\xx=[x_\ii]_{\ii\in\II}$, then $\xx$ is a vector of size $|\II|$ whose components are indexed by a $d$-index $\ii$ varying in $\II$ according to the lexicographic ordering. Similarly, if $X=[x_{\ii\jj}]_{\ii,\jj\in\II}=[x_{\ii\jj}]_{\ii\in\II}^{\jj\in\II}$, then $X$ is a square matrix of size $|\II|$ whose components are indexed by a pair of $d$-indices $\ii,\jj$, both varying in $\II$ according to the lexicographic ordering.
When $\II$ is a $d$-index range $\{\hh,\ldots,\kk\}$, the notation $\ii\in\II$ is often replaced by $\ii=\hh,\ldots,\kk$, and the notations $\xx=[\xx_\ii]_{\ii\in\II}$ and $X=[x_{\ii\jj}]_{\ii,\jj\in\II}=[x_{\ii\jj}]_{\ii\in\II}^{\jj\in\II}$ are often replaced by $\xx=[x_\ii]_{\ii=\hh}^\kk$ and $X=[x_{\ii\jj}]_{\ii,\jj=\hh}^\kk=[x_{\ii\jj}]_{\ii=\hh,\ldots,\kk}^{\jj=\hh,\ldots,\kk}$, respectively.
Operations involving vectors in $\mathbb R^d$ that have no meaning in the vector space $\mathbb R^d$ must always be interpreted in the componentwise sense.
For instance, if $\ii,\jj,\hh\in\mathbb R^d$, then $\jj\hh=(j_1h_1,\ldots,j_dh_d)$, $\ii/\jj=(i_1/j_1,\ldots,i_d/j_d)$, etc.
For all $d$-indices $\ii,\jj$, we define $\delta_{\ii\jj}=1$ if $\ii=\jj$ and $\delta_{\ii\jj}=0$ otherwise.
If $\nn$ is a positive $d$-index, we denote by $\ee_\ii^{(\nn)}$, $\ii=\bu,\ldots,\nn$, the (column) vectors of the canonical basis of $\mathbb C^{N(\nn)}$. In other words, for every $\ii=\bu,\ldots,\nn$, $\ee_\ii^{(\nn)}$ is the vector of $\mathbb C^{N(\nn)}$ having $1$ in position $\ii$ and $0$ elsewhere.
We provide below a few examples to help the reader become familiar with the multi-index notation.

\begin{example}\label{mi.e1}
Let $\hh=\bu=(1,1)$ and $\kk=(4,2)$. The multi-index range $\hh,\ldots,\kk$ consists of $N(\kk-\hh+\bu)=N(\kk)=8$ elements, which are sorted according to the lexicographic ordering as follows:
\[ (1,1),\,(1,2),\;(2,1),\,(2,2),\;(3,1),\,(3,2),\;(4,1),\,(4,2). \]
\end{example}

\begin{example}\label{mi.e2}
Let $a:[0,1]^2\to\mathbb C$ and $\nn\in\mathbb N^2$. Set
$\xx=\bigl[a\bigl(\frac\ii\nn\bigr)\bigr]_{\ii=\bu}^\nn$.
Then, $\xx$ is a vector of size $N(\nn)$ given by
\[ \xx=\left[\begin{array}{c}
\vphantom{\int_{\sum_0^0}}a(\frac1{n_1},\frac1{n_2})\\
\vphantom{\int^{\sum_0^0}}a(\frac1{n_1},\frac2{n_2})\\
\vdots\\
\vphantom{\int_{\sum_0^0}^{\sum_0^0}}a(\frac1{n_1},1)\\
\hline
\vphantom{\int_{\sum_0^0}^{\sum_0^0}}a(\frac2{n_1},\frac1{n_2})\\
\vphantom{\int^{\sum_0^0}}a(\frac2{n_1},\frac2{n_2})\\
\vdots\\
\vphantom{\int_{\sum_0^0}^{\sum_0^0}}a(\frac2{n_1},1)\\
\hline
\vdots\\
\vdots\\
\vdots\\
\hline
\vphantom{\int_{\sum_0^0}^{\sum_0^0}}a(1,\frac1{n_2})\\
\vphantom{\int^{\sum_0^0}}a(1,\frac2{n_2})\\
\vdots\\
\vphantom{\int^{\sum_0^0}}a(1,1)
\end{array}\right]=\left[\begin{array}{c}
\xx_1\\
\hline
\xx_2\\
\hline
\vdots\\
\hline
\xx_{n_1}
\end{array}\right],\qquad\xx_{i_1}=\left[\begin{array}{c}
\vphantom{\int_{\sum_0^0}}a(\frac{i_1\vphantom{i_{\sum}}}{n_1},\frac1{n_2})\\
\vphantom{\int_{\sum_0^0}^{\sum_0^0}}a(\frac{i_1\vphantom{i_{\sum}}}{n_1},\frac2{n_2})\\
\vphantom{\int_{\sum_0^0}}\vdots\\
\vphantom{\int^{\sum_0^0}}a(\frac{i_1\vphantom{i_{\sum}}}{n_1},1)
\end{array}\right],\qquad i_1=1,\ldots,n_1. \]
Note that $\xx$ is partitioned into $n_1$ subvectors $\xx_1,\ldots,\xx_{n_1}$ of size $n_2$. A vector of this kind is referred to as a 2-level vector with level orders $n_1,n_2$.
\end{example}

\begin{example}\label{mi.e3}
Consider the matrix
\begin{equation}\label{A.m}
A=\left[\begin{array}{cc|cc}
4 & 4 & 0 & 0\\
4 & 4 & 0 & 0\\
\hline
0 & 0 & 1 & 1^{\vphantom{\int_0^1}}\\
0 & 0 & 1 & 1
\end{array}\right].
\end{equation}
Instead of indexing the entries of $A$ in the standard way, i.e., by means of two scalar indices $i,j=1,\ldots,4$, we can decide to index the entries of $A$ by means of two 2-indices $\ii,\jj=\bu,\ldots,\mathbf2=(1,1),(1,2),(2,1),(2,2)$. This is possible because both the ranges $1,\ldots,4$ and $\bu,\ldots,\mathbf2$ have 4 elements, and 4 is the size of $A$. The two writings $A=[a_{ij}]_{i,j=1}^4$ and $A=[a_{\ii\jj}]_{\ii,\jj=\bu}^{\mathbf2}$ correspond to the two different indexings. We have
\begin{align*}
a_{(1,1),(1,1)}&=4, &a_{(1,1),(1,2)}&=4, &a_{(1,1),(2,1)}&=0, &a_{(1,1),(2,2)}&=0,\\
a_{(1,2),(1,1)}&=4, &a_{(1,2),(1,2)}&=4, &a_{(1,2),(2,1)}&=0, &a_{(1,2),(2,2)}&=0,\\
a_{(2,1),(1,1)}&=0, &a_{(2,1),(1,2)}&=0, &a_{(2,1),(2,1)}&=1, &a_{(2,1),(2,2)}&=1,\\
a_{(2,2),(1,1)}&=0, &a_{(2,2),(1,2)}&=0, &a_{(2,2),(2,1)}&=1, &a_{(2,2),(2,2)}&=1.
\end{align*}
Indexing the entries of $A$ with the 2-indices $\ii,\jj$ reflects the fact that we are thinking of $A$ as a $2\times 2$ block matrix in which each of the 4 blocks is a $2\times 2$ matrix, as indicated in \eqref{A.m}: for all $\ii,\jj=\bu,\ldots,\mathbf2$, the entry $a_{\ii\jj}$ is the $(i_2,j_2)$ entry of the $(i_1,j_1)$ block of $A$. For example,
\[ a_{(2,1),(2,2)}={\rm entry\ }(1,2){\rm\ in\ the\ }(2,2){\rm\ block\ of\ }A=1. \]
We can therefore write
\[ a_{\ii\jj}=\begin{cases}
4, &\mbox{if $i_1=j_1=1$},\\
0, &\mbox{if $i_1=1$ and $j_1=2$},\\
0, &\mbox{if $i_1=2$ and $j_1=1$},\\
1, &\mbox{if $i_1=j_1=2$}.
\end{cases} \]
Note that $a_{\ii\jj}$ depends only on the first components of the 2-indices $\ii$ and $\jj$.
\end{example}

In the next example, we introduce tensor (Kronecker) products of matrices.
The indexing of tensor products is one of the main motivations for the multi-index notation.

\begin{example}[\textbf{tensor products}]\label{mi.e4}
If $X\in\mathbb C^{m_1\times m_2}$ and $Y\in\mathbb C^{\ell_1\times\ell_2}$, the tensor (Kronecker) product of $X$ and $Y$ is the $m_1\ell_1\times m_2\ell_2$ matrix defined by
\begin{equation}\label{XoY}
X\otimes Y=[x_{ij}Y]_{i=1,\ldots,m_1}^{j=1,\ldots,m_2}=\begin{bmatrix}x_{11}Y & x_{12}Y & \cdots & x_{1m_2}Y\\ x_{21}Y & x_{22}Y & \cdots & x_{2m_2}Y\\ \vdots & \vdots & & \vdots\\ x_{m_11}Y & x_{m_12}Y & \cdots & x_{m_1m_2}Y\end{bmatrix}.
\end{equation}
Using the identities
\[ p=\lfloor p/q\rfloor q+p\,{\rm mod}\,q,\qquad\lfloor p/q\rfloor=\lceil(p+1)/q\rceil-1, \]
which are satisfied for all integers $p\ge0$ and $q\ge1$, for every $i=1,\ldots,m_1\ell_1$ and $j=1,\ldots,m_2\ell_2$ we can write 
\begin{align*}
i&=(\lceil i/\ell_1\rceil-1)\ell_1+((i-1)\,{\rm mod}\,\ell_1)+1,\\
j&=(\lceil j/\ell_2\rceil-1)\ell_2+((j-1)\,{\rm mod}\,\ell_2)+1.
\end{align*}
Thus, for every $i=1,\ldots,m_1\ell_1$ and $j=1,\ldots,m_2\ell_2$, the $(i,j)$ entry of $X\otimes Y$ is given by
\[ (X\otimes Y)_{ij}=x_{\lceil i/\ell_1\rceil,\lceil j/\ell_2\rceil}\,y_{((i-1)\,{\rm mod}\,\ell_1)+1,((j-1)\,{\rm mod}\,\ell_2)+1}. \]
It is clear that this expression is complicated. Now, suppose we decide to index the entries of $X\otimes Y$ by two 2-indices $\ii,\jj$ such that $\ii=\bu,\ldots,\nn$ and $\jj=\bu,\ldots,\kk$, where $\nn=(m_1,\ell_1)$ and $\kk=(m_2,\ell_2)$. This indexing, which is possible because
\begin{align*}
|\{\bu,\ldots,\nn\}|&=N(\nn)=m_1\ell_1={\rm number\ of\ rows\ of\ }X\otimes Y,\\
|\{\bu,\ldots,\kk\}|&=N(\kk)=m_2\ell_2={\rm number\ of\ columns\ of\ }X\otimes Y,
\end{align*}
reflects the fact that we are thinking of $X\otimes Y$ as an $m_1\times m_2$ block matrix in which each of the $m_1m_2$ blocks is a matrix of size $\ell_1\times\ell_2$. Actually, this is the natural way of thinking in view of the structure of $X\otimes Y$; see \eqref{XoY}.
With such an indexing, for every $\ii=\bu,\ldots,\nn$ and $\jj=\bu,\ldots,\kk$, the $(\ii,\jj)$ entry of $X\otimes Y$ is given by
\[ (X\otimes Y)_{\ii\jj}={\rm entry\ }(i_2,j_2){\rm\ in\ the\ }(i_1,j_1){\rm\ block\ of}\ X\otimes Y=x_{i_1j_1}y_{i_2j_2}. \]
We then see that $(X\otimes Y)_{\ii\jj}$ has a much simpler expression than $(X\otimes Y)_{ij}$. In conclusion, by using 2-indices instead of scalar indices in the indexing of $X\otimes Y$, the complicated componentwise expression
\[ X\otimes Y=[(X\otimes Y)_{ij}]_{i=1,\ldots,m_1\ell_1}^{j=1,\ldots,m_2\ell_2}=[x_{\lceil i/\ell_1\rceil,\lceil j/\ell_2\rceil}\,y_{((i-1)\,{\rm mod}\,\ell_1)+1,((j-1)\,{\rm mod}\,\ell_2)+1}]_{i=1,\ldots,m_1\ell_1}^{j=1,\ldots,m_2\ell_2} \]
simplifies to
\[ X\otimes Y=[(X\otimes Y)_{\ii\jj}]_{\ii=\bu,\ldots,\nn}^{\jj=\bu,\ldots,\kk}=[x_{i_1j_1}y_{i_2j_2}]_{\ii=\bu,\ldots,\nn}^{\jj=\bu,\ldots,\kk}. \]
\end{example}

Throughout this paper, we assume in the reader a certain familiarity with tensor products defined in Example~\ref{mi.e4}.
Actually, an in-depth knowledge is not required: all properties of tensor products that we need herein can be found, e.g., in \cite[Section~2.4.5, pp.~40--41]{GLTbookI} or \cite[Section~2.5, pp.~18--20]{GLTbookII}.

In the next example, we introduce multilevel block matrices. The indexing of multilevel block matrices is probably the most important motivation for the multi-index notation.

\begin{example}[\textbf{multilevel block matrices}]\label{mi.e4'}
In many cases, it is convenient to partition matrices into blocks, which are partitioned into smaller blocks, which are partitioned into smaller blocks, and so on until a certain nesting level $d$ is reached. Such matrices are called multilevel block matrices. More precisely, following Tyrtyshnikov \cite[Section~6]{Ty96}, we say that a matrix $A$ is a $d$-level $(s\times t)$-block matrix with level orders $n_1,n_2,\ldots,n_d$ if:
\begin{itemize}[nolistsep,leftmargin=*]
	\item $A$ is thought of as an $(n_1n_2\cdots n_d)\times(n_1n_2\cdots n_d)$ block matrix whose ``entries'' are $s\times t$ blocks. The size of $A$ is therefore $N(\nn)s\times N(\nn)t$ with $\nn=(n_1,n_2,\ldots,n_d)$.
	\item $A$ is partitioned into $n_1^2$ blocks of size $\frac{N(\nn)}{n_1}s\times\frac{N(\nn)}{n_1}t$, each of which is partitioned into $n_2^2$ blocks of size $\frac{N(\nn)}{n_1n_2}s\times\frac{N(\nn)}{n_1n_2}t$, each of which is partitioned into $n_3^2$ blocks of size $\frac{N(\nn)}{n_1n_2n_3}s\times\frac{N(\nn)}{n_1n_2n_3}t$, and so on until the last $n_d^2$ blocks of size $\frac{N(\nn)}{n_1n_2\cdots n_d}s\times\frac{N(\nn)}{n_1n_2\cdots n_d}t=s\times t$, which are the (block) entries of $A$. In formulas,
	\begin{alignat*}{5}
	A&=[A_{i_1j_1}]_{i_1,j_1=1}^{n_1}, &\qquad A_{i_1j_1}&\in\mathbb C^{\tilde n_1s\times\tilde n_1t}, &\qquad\tilde n_1&=\frac{N(\nn)}{n_1};\\[3pt]
	A_{i_1j_1}&=[A_{i_1j_1;\,i_2j_2}]_{i_2,j_2=1}^{n_2}, &\qquad A_{i_1j_1;\,i_2j_2}&\in\mathbb C^{\tilde n_2s\times\tilde n_2t}, &\qquad\tilde n_2&=\frac{N(\nn)}{n_1n_2};\\
	\vdots\\
	A_{i_1j_1;\ldots;i_{d-1}j_{d-1}}&=[A_{i_1j_1;\ldots;i_dj_d}]_{i_d,j_d=1}^{n_d}, &\qquad A_{i_1j_1;\ldots;i_dj_d}&\in\mathbb C^{\tilde n_ds\times\tilde n_dt}, &\qquad\tilde n_d&=\frac{N(\nn)}{n_1n_2\cdots n_d}=1.
	\end{alignat*}
\end{itemize}
Indexing the (block) entries of a $d$-level matrix $A$ by two traditional scalar indices $i,j=1,\ldots,N(\nn)$ is a nightmare. On the contrary, $A$ admits a natural indexing  by means of two $d$-indices $\ii,\jj=\bu,\ldots,\nn$.
Indeed, we have
\[ A=[A_{\ii\jj}]_{\ii,\jj=\bu}^\nn, \]
where $A_{\ii\jj}=A_{i_1j_1;\ldots;i_dj_d}$ for $\ii,\jj=\bu,\ldots,\nn$.
\end{example}

As a byproduct of Example~\ref{mi.e4'}, we provide below the formal definition of multilevel block matrices. 

\begin{definition}[\textbf{multilevel block matrix}]\label{mbms}
A multilevel block matrix is a matrix of the form
\begin{equation}\label{A:mbm}
A=[A_{\ii\jj}]_{\ii,\jj=\bu}^\nn\in\mathbb C^{N(\nn)s\times N(\nn)t},
\end{equation}
where $\nn\in\mathbb N^d$ for some positive integer $d$ and $A_{\ii\jj}\in\mathbb C^{s\times t}$ for some positive integers $s,t$ and for every $\ii,\jj=\bu,\ldots,\nn$.
The components $n_1,\ldots,n_d$ of the multi-index $\nn$ are referred to as the level orders of the matrix $A$ in \eqref{A:mbm}.
Note that each ``entry'' of $A$ is an $s\times t$ block, each ``row'' of $A$ is a block row (consisting of $s$ ``normal'' rows), and each column of $A$ is a block column (consisting of $t$ ``normal'' columns).
The matrix $A$ is also called a $d$-level $(s\times t)$-block matrix if we want to specify the number of levels $d$ and the size $s\times t$ of each block.
In the case where $s=t=1$, i.e., the block entries of $A$ are scalars, the matrix $A$ is simply called a $d$-level matrix.
\end{definition}

One of the most important examples of multilevel block matrices is given by tensor products.

\begin{example}[\textbf{tensor products}]\label{mi.e4-bis}
Let $X_1\in\mathbb C^{n_1\times n_1}$ and $X_2\in\mathbb C^{n_2\times n_2}$. We have seen in Example~\ref{mi.e4} that $X_1\otimes X_2$ admits the following natural indexing by two 2-indices $\ii,\jj=\bu,\ldots,\nn$ with $\nn=(n_1,n_2)$:
\[ X_1\otimes X_2=[(X_1)_{i_1j_1}(X_2)_{i_2j_2}]_{\ii,\jj=\bu}^\nn. \]
Thus, $X_1\otimes X_2$ is a 2-level matrix with level orders $n_1,n_2$.
\end{example}

The importance of multilevel block matrices lies in the following remark.

\begin{remark}\label{mbm-rectangular}
When a linear differential problem defined over a $d$-dimensional rectangular domain is discretized by means of a standard numerical method (such as finite differences, finite elements, isogeometric analysis, etc.), the computation of the numerical solution reduces to solving a linear discrete problem (usually a linear system). If we denote by $A$ the corresponding matrix, then $A$ is normally a $d$-level $(s\times t)$-block matrix with level orders $n_1,\ldots,n_d$, where $n_i$ is deeply connected with the discretization parameter in the $i$th direction and the size $s\times t$ of each block entry depends on the used numerical method (usually $s=t$). For example, in finite difference and isogeometric analysis discretizations we have $s=t=1$ \cite[Chapter~7]{GLTbookII}, while in higher-order finite element discretizations we have $s=t=(p-k)^d$, where $p$ and $k$ are, respectively, the degree and the smoothness of the used polynomial basis functions \cite[Section~6]{bgd}.
\end{remark}

We conclude this section with the definition of multi-indices associated with a set.

\begin{definition}[\textbf{multi-indices associated with a set}]
Let $\nn\in\mathbb N^d$ and let $\Omega\subseteq\mathbb R^d$.
The set of multi-indices $\ii\in\{\bu,\ldots,\nn\}$ associated with $\Omega$ is denoted by $\II_\nn^\Omega$ and is defined as follows:
\[ \II_\nn^\Omega=\biggl\{\ii\in\{\bu,\ldots,\nn\}:\,\frac\ii\nn\in\Omega\biggr\}=\biggl\{\frac\ii\nn:\,\ii=\bu,\ldots,\nn\biggr\}\cap\Omega. \]
The cardinality of $\II_\nn^\Omega$ is denoted by $N_\nn^\Omega$:
\[ N_\nn^\Omega=|\II_\nn^\Omega|. \]
Note that, if $\Omega\supseteq(0,1]^d$, then $\II_\nn^\Omega=\{\bu,\ldots,\nn\}$ and $N_\nn^\Omega=N(\nn)$ for every $\nn\in\mathbb N^d$.
\end{definition}

\begin{example}
let $a,b:[0,1]^2\to\mathbb C$ and let $\nn=(3,2)$. The multi-index range $\bu,\ldots,\nn$ consists of the multi-indices
\[ \bu,\ldots,\nn=(1,1),(1,2),(2,1),(2,2),(3,1),(3,2). \]
Consider the 2-level matrix
\begin{align}\label{sg}
A&=\bigl[a_{\ii\jj}\bigr]_{\ii,\jj=\bu}^\nn=\Bigl[a\Bigl(\frac\ii\nn\Bigr)b\Bigl(\frac\jj\nn\Bigr)\Bigr]_{\ii,\jj=\bu}^\nn\notag\\
&=\left[\begin{array}{cc|cc|cc}
\cellcolor{green!10}a(\frac13,\frac12)b(\frac13,\frac12) & \cellcolor{green!10}a(\frac13,\frac12)b(\frac13,1) & a(\frac13,\frac12)b(\frac23,\frac12) & \cellcolor{green!10}a(\frac13,\frac12)b(\frac23,1) & a(\frac13,\frac12)b(1,\frac12) & \cellcolor{green!10}a(\frac13,\frac12)b(1,1)\vphantom{\Big|}\\
a(\frac13,1)b(\frac13,\frac12) & a(\frac13,1)b(\frac13,1) & a(\frac13,1)b(\frac23,\frac12) & a(\frac13,1)b(\frac23,1) & a(\frac13,1)b(1,\frac12) & a(\frac13,1)b(1,1)\vphantom{\Big|}\\
\hline
a(\frac23,\frac12)b(\frac13,\frac12) & a(\frac23,\frac12)b(\frac13,1) & a(\frac23,\frac12)b(\frac23,\frac12) & a(\frac23,\frac12)b(\frac23,1) & a(\frac23,\frac12)b(1,\frac12) & a(\frac23,\frac12)b(1,1)\vphantom{\Big|}\\
\cellcolor{green!10}a(\frac23,1)b(\frac13,\frac12) & \cellcolor{green!10}a(\frac23,1)b(\frac13,1) & a(\frac23,1)b(\frac23,\frac12) & \cellcolor{green!10}a(\frac23,1)b(\frac23,1) & a(\frac23,1)b(1,\frac12) & \cellcolor{green!10}a(\frac23,1)b(1,1)\vphantom{\Big|}\\
\hline
a(1,\frac12)b(\frac13,\frac12) & a(1,\frac12)b(\frac13,1) & a(1,\frac12)b(\frac23,\frac12) & a(1,\frac12)b(\frac23,1) & a(1,\frac12)b(1,\frac12) & a(1,\frac12)b(1,1)\vphantom{\Big|}\\
\cellcolor{green!10}a(1,1)b(\frac13,\frac12) & \cellcolor{green!10}a(1,1)b(\frac13,1) & a(1,1)b(\frac23,\frac12) & \cellcolor{green!10}a(1,1)b(\frac23,1) & a(1,1)b(1,\frac12) & \cellcolor{green!10}a(1,1)b(1,1)\vphantom{\Big|}\\
\end{array}\right].
\end{align}
Let $\Omega$ and $\Omega'$ be subsets of $\mathbb R^2$ such that $\II_\nn^\Omega=\{(1,1),(2,2),(3,2)\}$ and $\II_\nn^{\Omega'}=\{(1,1),(1,2),(2,2),(3,2)\}$. Then, the submatrix of $A$ given by
\[ B=\bigl[a_{\ii\jj}\bigr]_{\ii\in\II_\nn^\Omega}^{\jj\in\II_\nn^{\Omega'}}=\Bigl[a\Bigl(\frac\ii\nn\Bigr)b\Bigl(\frac\jj\nn\Bigr)\Bigr]_{\ii\in\II_\nn^\Omega}^{\jj\in\II_\nn^{\Omega'}} \]
is the submatrix shaded in green in \eqref{sg}.
\end{example}

\section{Overview of the theory of GLT sequences}\label{sec:overview}

In this section, we overview the theory of GLT sequences.
For conciseness purposes, we only present the results that we need in this paper.
For a comprehensive exposition of the topic, see \cite{rg,bgR,bg,bgd,GLTbookI,GLTbookII}.
For an introduction to the subject, we recommend \cite{GLT-intro,aip}.
Throughout this paper, $\mathbb N=\{1,2,3,\ldots\}$ and the cardinality of a set $E$ is denoted by~$|E|$.

\subsection{Singular value and spectral distribution of a sequence of matrices}\label{sec:distr}

Let $\mu_k$ be the Lebesgue measure in $\mathbb R^k$. Throughout this paper, all terminology from measure theory (such as ``measurable set'', ``measurable function'', ``a.e.'', etc.)\ always refers to the Lebesgue measure. A matrix-valued function $f:D\subseteq\mathbb R^k\to\mathbb C^{s\times t}$ is said to be measurable (respectively, continuous, continuous a.e., in $L^1(D)$, etc.)\ if its components $f_{ij}:D\to\mathbb C$, $i=1,\ldots,s$, $j=1,\ldots,t$, are measurable (respectively, continuous, continuous a.e., in $L^1(D)$, etc.). 
We denote by $C_c(\mathbb C)$ (respectively, $C_c(\mathbb R)$) the space of complex-valued functions defined on $\mathbb C$ (respectively, $\mathbb R$) with bounded support.
For every $x,y\in\mathbb R$, we define $x\wedge y=\min(x,y)$ and $x\vee y=\max(x,y)$.
The singular values of a matrix $A\in\mathbb C^{m\times n}$ are denoted by $\sigma_i(A)$, $i=1,\ldots,m\wedge n$, and the eigenvalues of a matrix $A\in\mathbb C^{m\times m}$ are denoted by $\lambda_i(A)$, $i=1,\ldots,m$. 
A sequence of matrices is, by definition, a sequence of the form $\{A_n\}_n$, where $n$ varies in some infinite subset of $\mathbb N$ and $A_n$ is a matrix of size $d_n\times e_n$ such that both $d_n$ and $e_n$ tend to $\infty$ as $n\to\infty$.

\begin{definition}[\textbf{singular value and spectral distribution of a sequence of matrices}]\label{def-distribution}\hfill
\begin{itemize}[leftmargin=*,nolistsep]
	\item Let $\{A_n\}_n$ be a sequence of matrices with $A_n$ of size $d_n\times e_n$, and let $f:D\subset\mathbb R^k\to\mathbb C^{s\times t}$ be measurable with $0<\mu_k(D)<\infty$.
	We say that $\{A_n\}_n$ has a (asymptotic) singular value distribution described by $f$, and we write $\{A_n\}_n\sim_\sigma f$, if
	\begin{equation*}
	\lim_{n\rightarrow\infty}\frac1{d_n\wedge e_n}\sum_{i=1}^{d_n\wedge e_n}F(\sigma_i(A_n))=\frac1{\mu_k(D)}\int_D\frac{\sum_{i=1}^{s\wedge t}F(\sigma_i(f(\xx)))}{s\wedge t}{\rm d}\xx,\qquad\forall\,F\in C_c(\mathbb R).
	\end{equation*}
	\item Let $\{A_n\}_n$ be a sequence of square matrices with $A_n$ of size $d_n$, and let $f:D\subset\mathbb R^k\to\mathbb C^{s\times s}$ be measurable with $0<\mu_k(D)<\infty$.
	We say that $\{A_n\}_n$ has a (asymptotic) spectral (or eigenvalue) distribution described by $f$, and we write $\{A_n\}_n\sim_\lambda f$, if
	\begin{equation*}
	\lim_{n\rightarrow\infty}\frac1{d_n}\sum_{i=1}^{d_n}F(\lambda_i(A_n))=\frac1{\mu_k(D)}\int_D\frac{\sum_{i=1}^sF(\lambda_i(f(\xx)))}{s}{\rm d}\xx,\qquad\forall\,F\in C_c(\mathbb C).
	\end{equation*}
\end{itemize}
\end{definition}

We remark that the functions $\xx\mapsto\sum_{i=1}^{s\wedge t}F(\sigma_i(f(\xx)))$ and $\xx\mapsto\sum_{i=1}^{s}F(\lambda_i(f(\xx)))$ appearing in Definition~\ref{def-distribution} are well-defined and measurable by \cite[Lemma~2.1]{bg}. We refer the reader to \cite[Remarks~4.1--4.2]{blocking} for the informal meaning behind the singular value and spectral distribution of a sequence of matrices.

\subsection{Special sequences of matrices}\label{sec:ssm}

We introduce in this section a few sequences of matrices that play an important role in the theory of GLT sequences.

\subsubsection{Zero-distributed sequences}
A zero-distributed sequence is a sequence of matrices $\{Z_n\}_n$ such that $\{Z_n\}_n\sim_\sigma0$, i.e.,
\[ \lim_{n\to\infty}\frac1{d_n\wedge e_n}\sum_{i=1}^{d_n\wedge e_n}F(\sigma_i(Z_n))=F(0),\qquad\forall\,F\in C_c(\mathbb R), \]
where $d_n\times e_n$ is the size of $Z_n$.

\subsubsection{Sequences of diagonal sampling matrices}
If $\nn\in\mathbb N^d$ and $a:[0,1]^d\to\mathbb C^{s\times t}$, the $\nn$th (multilevel block) diagonal sampling matrix generated by $a$ is the block diagonal matrix of size $N(\nn)s\times N(\nn)t$ given by
\begin{equation*}
D_\nn(a)=\mathop{\rm diag}_{\ii=\bu,\ldots,\nn}a\Bigl(\frac\ii\nn\Bigr)=\Bigl[a\Bigl(\frac\ii\nn\Bigr)\delta_{\ii\jj}\Bigr]_{\ii,\jj=\bu}^\nn.
\end{equation*}
Any sequence of matrices of the form $\{D_\nn(a)\}_n$, with $a:[0,1]^d\to\mathbb C^{s\times t}$ and $\{\nn=\nn(n)\}_n\subseteq\mathbb N^d$ such that $\nn\to\infty$ as $n\to\infty$, is referred to as a sequence of (multilevel block) diagonal sampling matrices generated by $a$.

\subsubsection{Toeplitz sequences}
If $f:[-\pi,\pi]^d\to\mathbb C^{s\times t}$ is a function in $L^1([-\pi,\pi]^d)$, its Fourier coefficients are denoted by $\{f_\kk\}_{\kk\in\mathbb Z^d}$ and are defined as follows:
\begin{equation*}
f_\kk=\frac1{(2\pi)^d}\int_{[-\pi,\pi]^d}f(\btheta)\hspace{0.75pt}{\rm e}^{-{\rm i}\kk\cdot\btheta}{\rm d}\btheta\,\in\,\mathbb C^{s\times t},\qquad\kk\in\mathbb Z^d,
\end{equation*}
where $\kk\cdot\btheta=k_1\theta_1+\ldots+k_d\theta_d$ is the scalar product and it is understood that the integral of a matrix-valued function is computed componentwise.
If $\nn\in\mathbb N^d$ and $f:[-\pi,\pi]^d\to\mathbb C^{s\times t}$ is a function in $L^1([-\pi,\pi]^d)$, the $\nn$th (multilevel block) Toeplitz matrix generated by $f$ is the $N(\nn)s\times N(\nn)t$ matrix given by
\begin{equation*}
T_\nn(f)=[f_{\ii-\jj}]_{\ii,\jj=\bu}^\nn.
\end{equation*}
Any sequence of matrices of the form $\{T_\nn(f)\}_n$, with $f:[-\pi,\pi]^d\to\mathbb C^{s\times t}$ in $L^1([-\pi,\pi]^d)$ and $\{\nn=\nn(n)\}_n\subseteq\mathbb N^d$ such that $\nn\to\infty$ as $n\to\infty$, is referred to as a (multilevel block) Toeplitz sequence generated by $f$.

\subsection{Approximating classes of sequences}

The notion of approximating classes of sequences (a.c.s.)\ is the cornerstone of an asymptotic approximation theory for sequences of matrices that is fundamental for the theory of GLT sequences; see \cite[Chapter~5]{GLTbookI}. The formal definition of a.c.s.\ is reported in Definition~\ref{a.c.s.}. 
Throughout this paper, we denote by $\|\cdot\|$ the Euclidean norm ($2$-norm) of vectors and the associated induced (operator) norm for matrices.
Recall that, for every matrix $X$, we have $\|X\|=\sigma_{\max}(X)$, where $\sigma_{\max}(X)$ denotes the maximum singular value of $X$.

\begin{definition}[\textbf{approximating class of sequences}]\label{a.c.s.}
Let $\{A_n\}_n$ be a sequence of matrices with $A_n$ of size $d_n\times e_n$, and let $\{\{B_{n,m}\}_n\}_m$ be a sequence of sequences of matrices with $B_{n,m}$ of size $d_n\times e_n$. We say that $\{\{B_{n,m}\}_n\}_m$ is an approximating class of sequences (a.c.s.)\ for $\{A_n\}_n$, and we write $\{B_{n,m}\}_n\xrightarrow{\rm a.c.s.}\{A_n\}_n$, if the following condition is met: for every $m$ there exists $n_m$ such that, for $n\ge n_m$,
\begin{equation*}
A_n=B_{n,m}+R_{n,m}+N_{n,m},\qquad{\rm rank}(R_{n,m})\le c(m)(d_n\wedge e_n),\qquad\|N_{n,m}\|\le\omega(m),
\end{equation*}
where $n_m,\,c(m),\,\omega(m)$ depend only on $m$ and $\lim_{m\to\infty}c(m)=\lim_{m\to\infty}\omega(m)=0$.
\end{definition}

\subsection{GLT sequences}\label{sec:GLT}

The definition of GLT sequences that we propose here can be inferred from \cite[Section~5]{bgR}. It is an equivalent alternative to the usual (complicated) definition.
Throughout this paper, we denote by $O_{s,t}$ the $s\times t$ zero matrix and by $I_s$ the $s\times s$ identity matrix.
If the size is clear from the context, we often write $O$ and $I$ instead of $O_{s,t}$ and~$I_s$.

\begin{definition}[\textbf{generalized locally Toeplitz sequence}]\label{GLT_def}
Let $\{A_n\}_n$ be a sequence of matrices, with $A_n$ of size $N(\nn)s\times N(\nn)t$ for some fixed positive integers $s,t$ and some sequence of positive $d$-indices $\{\nn=\nn(n)\}_n$ tending to $\infty$, and let $\kappa:[0,1]^d\times[-\pi,\pi]^d\to\mathbb C^{s\times t}$ be measurable. We say that $\{A_n\}_n$ is a (multilevel block) GLT sequence with symbol $\kappa$, and we write $\{A_n\}_n\sim_{\rm GLT}\kappa$, if there exist functions $a_{i,m}$, $f_{i,m}$, $i=1,\ldots,N_m$, such that:
\begin{itemize}[nolistsep,leftmargin=*]
	\item $a_{i,m}:[0,1]^d\to\mathbb C$ is continuous a.e.\ on $[0,1]^d$ and $f_{i,m}:[-\pi,\pi]^d\to\mathbb C^{s\times t}$ belongs to $L^1([-\pi,\pi]^d)$;
	\item $\kappa_m(\xx,\btheta)=\sum_{i=1}^{N_m}a_{i,m}(\xx)f_{i,m}(\btheta)\to\kappa(\xx,\btheta)$ a.e.\ on $[0,1]^d\times[-\pi,\pi]^d$; \vspace{2pt}
	\item $\{A_{n,m}\}_n=\bigl\{\sum_{i=1}^{N_m}D_\nn(a_{i,m}I_s)T_\nn(f_{i,m})\bigr\}_n\xrightarrow{\rm a.c.s.}\{A_n\}_n$.
\end{itemize}
\end{definition}

The symbol of a GLT sequence is unique, in the following sense: if $\{A_n\}_n\sim_{\rm GLT}\kappa$, with $\{A_n\}_n$ and $\kappa$ as in Definition~\ref{GLT_def}, and if $\xi:[0,1]^d\times[-\pi,\pi]^d\to\mathbb C^{s\times t}$ is another measurable function, then we have $\{A_n\}_n\sim_{\rm GLT}\xi$ if and only if $\xi=\kappa$ a.e.\ on $[0,1]^d\times[-\pi,\pi]^d$; see \cite[Theorem~4.3]{bgR}.
The properties of GLT sequences that we need in this paper are listed below. The corresponding proofs can be found in \cite{bgR,bg}.
Throughout this paper, if $X$ is a matrix, we denote by $X^\dag$ the Moore--Penrose pseudoinverse of $X$; see \cite{Bini,GV} for details. What is relevant for our purposes is that $X^\dag=X^{-1}$ whenever $X$ is invertible.

\begin{enumerate}[nolistsep,leftmargin=37.5pt]
	\item[\textbf{GLT1.}] Let $\{A_n\}_n$ be a sequence of matrices, with $A_n$ of size $N(\nn)s\times N(\nn)t$ for some fixed positive integers $s,t$ and some sequence of positive $d$-indices $\{\nn=\nn(n)\}_n$ tending to $\infty$, and let $\kappa:[0,1]^d\times[-\pi,\pi]^d\to\mathbb C^{s\times t}$ be measurable.
	\begin{itemize}[nolistsep,leftmargin=*]
		\item If $\{A_n\}_n\sim_{\rm GLT}\kappa$, then $\{A_n\}_n\sim_\sigma\kappa$.
		\item If $\{A_n\}_n\sim_{\rm GLT}\kappa$ and the matrices $A_n$ are Hermitian, then $s=t$, $\kappa$ is Hermitian a.e.\ on $[0,1]^d\times[-\pi,\pi]^d$, and $\{A_n\}_n\sim_\lambda\kappa$.
	\end{itemize}
	\item[\textbf{GLT2.}] Let $s,t$ be positive integers and let $\{\nn=\nn(n)\}_n$ be a sequence of positive $d$-indices tending to $\infty$. Then,
	\begin{itemize}[nolistsep,leftmargin=*]
		\item $\{T_\nn(f)\}_n\sim_{\rm GLT}\kappa(\xx,\btheta)=f(\btheta)$ if $f:[-\pi,\pi]^d\to\mathbb C^{s\times t}$ belongs to $L^1([-\pi,\pi]^d)$;
		\item $\{D_\nn(a)\}_n\sim_{\rm GLT}\kappa(\xx,\btheta)=a(\xx)$ if $a:[0,1]^d\to\mathbb C^{s\times t}$ is continuous a.e.\ on $[0,1]^d$;
		\item for every sequence of matrices $\{Z_n\}_n$ with $Z_n$ of size $N(\nn)s\times N(\nn)t$, we have
		$\{Z_n\}_n\sim_{\rm GLT}\kappa(\xx,\btheta)=O_{s,t}$ if and only if $\{Z_n\}_n\sim_\sigma0$.
	\end{itemize}
	\item[\textbf{GLT3.}] Let $\{A_n\}_n$, $\{B_n\}_n$ be sequences of matrices, with $A_n$ of size $N(\nn)s\times N(\nn)t$ and $B_n$ of size $N(\nn)u\times N(\nn)v$ for some fixed positive integers $s,t,u,v$ and some sequence of positive $d$-indices $\{\nn=\nn(n)\}_n$ tending to $\infty$, and let $\kappa:[0,1]^d\times[-\pi,\pi]^d\to\mathbb C^{s\times t}$ and $\xi:[0,1]^d\times[-\pi,\pi]^d\to\mathbb C^{u\times v}$ be measurable. Suppose that $\{A_n\}_n\sim_{\rm GLT}\kappa$ and $\{B_n\}_n\sim_{\rm GLT}\xi$. Then,
	\begin{itemize}[nolistsep,leftmargin=*]
		\item $\{\alpha A_n+\beta B_n\}_n\sim_{\rm GLT}\alpha\kappa+\beta\xi$ for every $\alpha,\beta\in\mathbb C$ if $\kappa$ and $\xi$ are summable (i.e., $s=u$ and $t=v$);
		\item $\{A_nB_n\}_n\sim_{\rm GLT}\kappa\xi$ if $\kappa$ and $\xi$ are multipliable (i.e., $t=u$);
		\item $\{A_n^\dag\}_n\sim_{\rm GLT}\kappa^{-1}$ if $\kappa$ is invertible a.e.~on $[0,1]^d\times[-\pi,\pi]^d$ (which implies in particular that $s=t$).
	\end{itemize}
\end{enumerate}

\section{DDM discretization matrices and GLT sequences}\label{sec:DDMs-GLT}

In this section, we illustrate through examples that the matrices arising from a DDM discretization often give rise to a GLT sequence.
This means that, if $A_n$ is a DDM discretization matrix corresponding to the mesh fineness parameter~$n$ and if mesh refinement occurs when $n$ tends to $\infty$, then $\{A_n\}_n$ is often a GLT sequence.
Throughout this paper, if $\Omega\subseteq\mathbb R^d$, we denote by $\accentset{\circ}\Omega$ and $\overline\Omega$, respectively, the interior and the closure of $\Omega$.
If $\{\alpha_n\}_n$ and $\{\beta_n\}_n$ are any two numerical sequences with $\beta_n\ne0$ eventually as $n\to\infty$, the notation $\alpha_n\sim\beta_n$ means that $\alpha_n/\beta_n\to1$ as $n\to\infty$, and the notation $\alpha_n=o(\beta_n)$ means that $\alpha_n/\beta_n\to0$ as $n\to\infty$.

\begin{example}[\textbf{decomposition of $(0,1)$ into a unique subdomain -- FD discretization}]\label{exa1}
Consider the second-order boundary value problem
\begin{equation}\label{classic-p}
\left\{\begin{aligned}
-u''(x)&=f(x),\qquad x\in(0,1),\\
u(0)&=u(1)=0.
\end{aligned}\right.
\end{equation}
Let $n\in\mathbb N$, set $h=(n+1)^{-1}$ and $x_i=ih$ for $i=0,\ldots,n+1$.
We decompose the domain $(0,1)$ into a unique subdomain equal to itself. Then, we look for an approximation $\tilde u$ of the solution $u$ over the unique subdomain $(0,1)$.
To this end, we use the classical 3-point central finite difference (FD) approximation of the one-dimensional Laplace operator:
\begin{equation*}
-u''(x_i)\approx\frac{-u(x_{i-1})+2u(x_i)-u(x_{i+1})}{h^2},\qquad i=1,\ldots,n.
\end{equation*}
This means that the values of the solution $u$ at the points $x_i$, $i=1,\ldots,n$, satisfy (approximately) the following linear system:
\begin{equation*}
-u(x_{i-1})+2u(x_i)-u(x_{i+1})=h^2f(x_i),\qquad i=1,\ldots,n.
\end{equation*}
We then approximate the nodal value $u(x_i)$ with the value $u_i$ for $i=0,\ldots,n+1$, where $u_0=u_{n+1}=0$ due to the Dirichlet boundary conditions and the vector $(u_1,\ldots,u_n)$ solves the linear system
\begin{equation}\label{i-form}
-u_{i-1}+2u_i-u_{i+1}=h^2f(x_i),\qquad i=1,\ldots,n.
\end{equation}
The approximation $\tilde u$ of the solution $u$ is defined as any function on $[0,1]$ such that $\tilde u(x_i)=u_i$ for every $i=0,\ldots,n+1$.
For example, we can take $\tilde u$ as the piecewise linear continuous function that interpolates the values $u_0,\ldots,u_{n+1}$ over the nodes $x_0,\ldots,x_{n+1}$.
The matrix of the linear system \eqref{i-form}, which is the DDM discretization matrix of this example, is given by
\begin{equation}\label{Tn(2-2cost)}
A_n=\begin{bmatrix}
2 & -1 & & & \\
-1 & 2 & -1 & & \\
& \ddots & \ddots & \ddots & \\
& & -1 & 2 & -1\\
& & & -1 & 2
\end{bmatrix}=T_n(2-2\cos\theta).
\end{equation}
By {\bf GLT2}, we have $\{A_n\}_n\sim_{\rm GLT}2-2\cos\theta$.
\end{example}

\begin{example}[\textbf{decomposition of $(0,1)$ into two subdomains -- FD discretization}]\label{exa2}
Consider again problem \eqref{classic-p}, let $n\in\mathbb N$, set $h=(n+1)^{-1}$ and $x_i=ih$ for $i=0,\ldots,n+1$. We decompose the domain $(0,1)$ into two subdomains $\Omega_1=\Omega_{1,n}=(0,1/2+vh)$ and $\Omega_2=\Omega_{2,n}=(1/2-vh,1)$, where $v\le(n+1)/2$ is a non-negative integer that determines the size of the overlap; see Figure~\ref{DD1d}.
In practice, the overlap $(1/2-vh,1/2+vh)$ consists of $2v$ mesh cells. Depending on $v$, the overlap may or may not vanish in the limit of mesh refinement $n\to\infty$.
For example, the overlap vanishes as $n\to\infty$ if $v$ is a fixed number independent of $n$. On the contrary, the overlap does not vanish as $n\to\infty$ if $v=v_n=\lfloor(n+1)\epsilon\rfloor$ for some $\epsilon\in(0,1/2)$, in which case $\Omega_1$ ``converges'' to $(0,1/2+\epsilon)$ as $n\to\infty$, $\Omega_2$ ``converges'' to $(1/2-\epsilon,1)$ as $n\to\infty$, and the overlap ``converges'' to $(1/2-\epsilon,1/2+\epsilon)$ as $n\to\infty$. 
Note that, in the case $v=0$, there is no overlap between $\Omega_1$ and $\Omega_2$: this is the so-called non-overlapping case.

\begin{figure}
\centering
\includegraphics[width=0.75\textwidth]{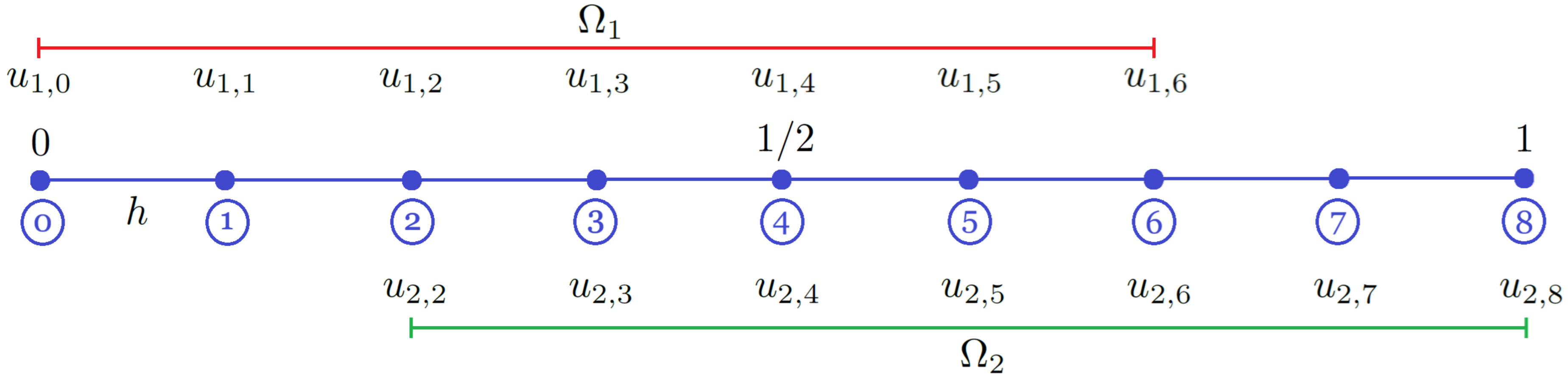}
\caption{Illustration for Example~\ref{exa2} in the case $n=7$ and $v=2$.}
\label{DD1d}
\end{figure}

Let us assume that $n$ is odd as in Figure~\ref{DD1d}, so that $x=x_{(n+1)/2}=1/2$ coincides with a grid point as well as its neighborhoods
\begin{align*}
x_{(n+1)/2-v}&=1/2-vh\\
\ \ \vdots\\
x_{(n+1)/2-1}&=1/2-h\\
x_{(n+1)/2}&=1/2\\
x_{(n+1)/2+1}&=1/2+h\\
\ \ \vdots\\
x_{(n+1)/2+v}&=1/2+vh.
\end{align*}
Ten, we look for two approximations of the solution $u$:
\begin{itemize}[nolistsep,leftmargin=*]
	\item an approximation $\tilde u_1$ of $u_1=u\big|_{\Omega_1}$ defined over the subdomain $\Omega_1$;
	\item an approximation $\tilde u_2$ of $u_2=u\big|_{\Omega_2}$ defined over the subdomain $\Omega_2$.
\end{itemize}
Note that the set of grid points in $\overline\Omega_1$ is given by $\{x_i:i\in\widehat\Omega_1\}$ with $\widehat\Omega_1=\{0,\ldots,(n+1)/2+v\}$ and the set of grid points in $\overline\Omega_2$ is given by $\{x_i:i\in\widehat\Omega_2\}$ with $\widehat\Omega_2=\{(n+1)/2-v,\ldots,n+1\}$, and define
\begin{alignat}{3}
\tilde u_1(x_i)&=u_{1,i}, &\qquad i&\in\widehat\Omega_1=\left\{0,\ldots,\dfrac{n+1}2+v\right\},\label{uk1}\\
\tilde u_2(x_i)&=u_{2,i}, &\qquad i&\in\widehat\Omega_2=\left\{\dfrac{n+1}2-v,\ldots,n+1\right\}.\label{uk2}
\end{alignat}
Determining $\tilde u_1$ and $\tilde u_2$ is equivalent to determining the $n+2v+3$ unknowns \eqref{uk1}--\eqref{uk2}. Once this is done, $\tilde u_1$ is defined as any function on $\overline\Omega_1=[0,1/2+vh]$ satisfying \eqref{uk1}, $\tilde u_2$ is defined as any function on $\overline\Omega_2=[1/2-vh,1]$ satisfying \eqref{uk2}, and the global approximation $\tilde u$ of the solution $u$ over the whole domain $(0,1)$ is defined as
\begin{equation}\label{utilde}
\tilde u(x)=\begin{cases}
\tilde u_1(x), &\quad x\in[0,1/2-vh),\\[3pt]
\tilde u_2(x), &\quad x\in(1/2+vh,1],\\[3pt]
\alpha_1\tilde u_1(x)+\alpha_2\tilde u_2(x), &\quad x\in[1/2-vh,1/2+vh],
\end{cases}
\end{equation}
where $\alpha_1\tilde u_1(x)+\alpha_2\tilde u_2(x)$ is a convex combination of $\tilde u_1(x)$ and $\tilde u_2(x)$, i.e., $\alpha_1,\alpha_2\ge0$ and $\alpha_1+\alpha_2=1$.

There are several ways to determine $\tilde u_1$ and $\tilde u_2$, i.e., the $n+2v+3$ unknowns \eqref{uk1}--\eqref{uk2}.
In this example, we propose a way that works in the overlapping case $v\ne0$.
Specifically, we use the classical 3-point central FD approximation of the one-dimensional Laplace operator at the $(n+1)/2+v-2$ points $x_1,\ldots,x_{(n+1)/2+v-2}$ lying inside $\Omega_1$ and at the $(n+1)/2+v-2$ points $x_{(n+1)/2-v+2},\ldots,x_n$ lying inside $\Omega_2$. In addition, we impose the two Dirichlet boundary conditions and two additional conditions at each of the two interface points $x=1/2\pm vh$ in order to make the number of equations equal to the number $n+2v+3$ of unknowns. The resulting linear system is the following:
\begin{equation}\label{i-form'}
\begin{cases}
u_{1,0}=0 &\quad\mbox{\footnotesize(Dirichlet boundary condition at $x=0$)}\\[5pt]
-u_{1,i-1}+2u_{1,i}-u_{1,i+1}=h^2f(x_i) &\quad i=1,\ldots,\dfrac{n+1}2+v-2\ \ \mbox{\footnotesize(FD approx.~in $\Omega_1$)}\\[5pt]
\rd u_{1,(n+1)/2+v}=u_{2,(n+1)/2+v}&\rd\quad\mbox{\footnotesize(interface condition: $C^0$ continuity at $x=1/2+vh$)}\\[5pt]
\bl u_{1,(n+1)/2+v}-u_{1,(n-1)/2+v}=u_{2,(n+3)/2+v}-u_{2,(n+1)/2+v}&\bl\quad\mbox{\footnotesize(interface condition: $C^1$ continuity at $x=1/2+vh$)}\\[5pt]
\hline\\[-10pt]
\rd u_{2,(n+1)/2-v}=u_{1,(n+1)/2-v}&\rd\quad\mbox{\footnotesize(interface condition: $C^0$ continuity at $x=1/2-vh$)}\\[5pt]
\bl u_{2,(n+3)/2-v}-u_{2,(n+1)/2-v}=u_{1,(n+1)/2-v}-u_{1,(n-1)/2-v}&\bl\quad\mbox{\footnotesize(interface condition: $C^1$ continuity at $x=1/2-vh$)}\\[5pt]
-u_{2,i-1}+2u_{2,i}-u_{2,i+1}=h^2f(x_i) &\quad i=\dfrac{n+1}2-v+2,\ldots,n\ \ \mbox{\footnotesize(FD approx.~in $\Omega_2$)}\\[5pt]
u_{2,n+1}=0 &\quad\mbox{\footnotesize(Dirichlet boundary condition at $x=1$)}
\end{cases}
\end{equation}
It must now be said that, in the standard DDM approach, the boundary values $u_{1,0}$, $u_{2,n+1}$ and the values ${\rd u_{1,(n+1)/2+v}}$, ${\rd u_{2,(n+1)/2-v}}$ are not considered as unknowns, because the former are equal to $0$ due to the Dirichlet boundary conditions and the latter are equal to ${\rd u_{2,(n+1)/2+v}}$, ${\rd u_{1,(n+1)/2-v}}$, respectively, due to the red interface conditions. Not considering $u_{1,0}$, $u_{2,n+1}$, ${\rd u_{1,(n+1)/2+v}}$, ${\rd u_{2,(n+1)/2-v}}$ as unknowns is equivalent to removing the first, last, red equations from the linear system \eqref{i-form'}, after replacing everywhere 
$u_{1,0}$, $u_{2,n+1}$ with $0$ and ${\rd u_{1,(n+1)/2+v}}$, ${\rd u_{2,(n+1)/2-v}}$ with ${\rd u_{2,(n+1)/2+v}}$, ${\rd u_{1,(n+1)/2-v}}$, respectively. The resulting linear system is the following:
\begin{equation}\label{i-form'-reduced}
\begin{cases}
-u_{1,i-1}+2u_{1,i}-u_{1,i+1}=h^2f(x_i) &\quad i=1,\ldots,\dfrac{n+1}2+v-2\ \ \mbox{\footnotesize(FD approx.~in $\Omega_1$)}\\[5pt]
\bl -u_{1,(n-1)/2+v}+2u_{2,(n+1)/2+v}-u_{2,(n+3)/2+v}=0&\bl\quad\mbox{\footnotesize(interface condition: $C^1$ continuity at $x=1/2+vh$)}\\[7.5pt]
\hline\\[-10pt]
\bl -u_{1,(n-1)/2-v}+2u_{1,(n+1)/2-v}-u_{2,(n+3)/2-v}=0&\bl\quad\mbox{\footnotesize(interface condition: $C^1$ continuity at $x=1/2-vh$)}\\[5pt]
-u_{2,i-1}+2u_{2,i}-u_{2,i+1}=h^2f(x_i) &\quad i=\dfrac{n+1}2-v+2,\ldots,n\ \ \mbox{\footnotesize(FD approx.~in $\Omega_2$)}
\end{cases}
\end{equation}
The matrix of the linear system \eqref{i-form'-reduced}, which is the DDM discretization matrix of this example, is given by
\begin{equation}\label{e2:A_n}
A_n=\left[\begin{array}{rrrrrr|rrrrrr}
2 & -1 & & & & & & & & & & \\
-1 & 2 & -1 & & & & & & & & & \\
& \ddots & \ddots & \ddots & & & & & & & & \\
& & \ddots & \ddots & \ddots & & & & & & & \\
& & & -1 & 2 & -1 & & & & & & \\
& & & & & \bl-1 & \bl0 & \bl\cdots & \bl0 & \bl2 & \bl-1 & \\
\hline
& \bl-1 & \bl2 & \bl0 & \bl\cdots & \bl0^{\vphantom{\int}} & \bl-1 & & & & & \\
& & & & & & -1 & 2 & -1 & & & \\
& & & & & & & \ddots & \ddots & \ddots & & \\
& & & & & & & & \ddots & \ddots & \ddots & \\
& & & & & & & & & -1 & 2 & -1 \\
& & & & & & & & & & -1 & 2\\
\end{array}\right]=T_{n+2v-1}(2-2\cos\theta)+R_n,
\end{equation}
where the number of zeros between $-1$ and $2$ in the above blue rows is equal to $2v-1$ and $R_n$ is a low-rank matrix with ${\rm rank}(R_n)\le2$ for every $n$.
Since $\{R_n\}_n\sim_\sigma0$, {\bf GLT2}--{\bf GLT3} yield $\{A_n\}_n\sim_{\rm GLT}2-2\cos\theta$.

A remark is in order before moving on to the next example.
While the red interface conditions in \eqref{i-form'} are essentially mandatory (in order to avoid that the approximate solution $\tilde u$ in \eqref{utilde} is discontinuous at $x=1/2+vh$ or $x=1/2-vh$), the blue interface conditions can be replaced with other suitable conditions. For example, the first blue condition can be replaced with a higher-order $C^1$ continuity condition at $x=1/2+vh$, such as
	\[ \bl \frac{3u_{1,(n+1)/2+v}-4u_{1,(n-1)/2+v}+u_{1,(n-3)/2+v}}2=\frac{-3u_{2,(n+1)/2+v}+4u_{2,(n+3)/2+v}-u_{2,(n+5)/2+v}}2, \]
	or with a condition that forces the FD approximation also at $x=x_i$ for $i=(n+1)/2+v-1$, such as
	\[ \bl -u_{1,i-1}+2u_{1,i}-u_{1,i+1}=h^2f(x_i),\qquad i=\frac{n+1}2+v-1. \]
	Similar replacements can also be considered for the second blue condition.
	It is clear, however, that any change in the blue interface conditions only produces a change in the blue rows of the matrix \eqref{e2:A_n}. Thus, if we denote by $\tilde A_n$ the new matrix, then $\tilde A_n=T_{n+2v+3}(2+2\cos\theta)+\tilde R_n$ with ${\rm rank}(\tilde R_n)\le2$, and {\bf GLT2}--{\bf GLT3} yield again $\{\tilde A_n\}_n\sim_{\rm GLT}2-2\cos\theta$.
\end{example}

\begin{remark}\label{always-low-rank-perturbation}
The difference between the classical FD discretization matrix \eqref{Tn(2-2cost)} and the FD-DDM discretization matrix \eqref{e2:A_n} amounts to a low-rank perturbation caused by the interface conditions. Even if we decompose the domain $(0,1)$ into more than two subdomains, the difference between the classical FD discretization matrix \eqref{Tn(2-2cost)} and the resulting FD-DDM discretization matrix remains a low-rank perturbation caused by the interface conditions. This shows why a DDM discretization often preserves the GLT structure and the symbol of the ``corresponding'' classical discretization (which in Examples~\ref{exa1}--\ref{exa2} is represented by the classical 3-point central FD approximation).
\end{remark}

In Examples~\ref{exa5}--\ref{exa6}, we assume that the reader has prior knowledge of B-splines, whose properties will be used systematically.
For readers who are not familiar with B-splines, we recommend \cite{LMS}.

\begin{example}[\textbf{decomposition of $(0,1)$ into a unique subdomain -- IgA discretization}]\label{exa5}
Consider again problem~\eqref{classic-p}, whose weak form reads as follows: find $u\in H_0^1(0,1)$ such that
\begin{equation}\label{weak_solution}
\int_0^1u'(x)v'(x){\rm d}x=\int_0^1f(x)v(x){\rm d}x,\qquad\forall\,v\in H^1_0(0,1).
\end{equation}
Under mild assumptions on the function $f$, problem \eqref{weak_solution} has a unique solution $u$, which is precisely the solution of \eqref{classic-p}; see \cite[Chapter~8]{Brezis}.
In the framework of Galerkin B-spline isogeometric analysis (IgA) based on B-splines of degree $p\ge1$ and maximal smoothness $C^{p-1}$ defined over a uniform grid with $n+1$ points, we fix $p,n\ge1$ and we look for an approximation $\tilde u$ of $u$ by solving \eqref{weak_solution} in the polynomial spline space
\begin{align*}
\mathcal W_{n,p}=\left\{s\in\mathcal V_{n,p}:\,s(0)=s(1)=0\right\}\subset H_0^1(0,1),
\end{align*}
where
\[ \mathcal V_{n,p}=\left\{s\in C^{p-1}[0,1]:\,s\big|_{\left[\frac{k}{n},\frac{k+1}{n}\right)}\in\mathbb P_p\ \textup{for}\ k=0,\ldots,n-1\right\}. \]
More precisely, the approximate solution $\tilde u$ is defined as the (unique) solution of the following (Galerkin) problem, which is the same as problem \eqref{weak_solution} with $H^1_0(0,1)$ replaced by the subspace $\mathcal W_{n,p}$: find $\tilde u\in\mathcal W_{n,p}$ such that
\[ \int_0^1\tilde u'(x)v'(x){\rm d}x=\int_0^1f(x)v(x){\rm d}x,\qquad\forall\,v\in\mathcal W_{n,p}. \]
Denoting by $\{B_{1,p},\ldots,B_{n+p,p}\}$ the B-spline basis of $\mathcal V_{n,p}$, the B-spline basis of $\mathcal W_{n,p}$ is given by $\{B_{2,p},\ldots,B_{n+p-1,p}\}$.
Thus, we can write $\tilde u=\sum_{j=2}^{n+p-1}u_jB_{j,p}$ for a unique vector $\boldsymbol u=(u_2,\ldots,u_{n+p-1})^T$. By linearity, the computation of $\tilde u$ (i.e., of $\boldsymbol u$) reduces to solving the linear system
\begin{equation}\label{IgA-system}
A_{n,p}\boldsymbol u=\boldsymbol f,\qquad A_{n,p}=\left[\int_0^1B_{j,p}'(x)B_{i,p}'(x){\rm d}x\right]_{i,j=2}^{n+p-1},\qquad\boldsymbol f=\left[\int_0^1f(x)B_{i,p}(x){\rm d}x\right]_{i=2}^{n+p-1}.
\end{equation}
For a detailed description of the Galerkin B-spline IgA discretization of \eqref{classic-p} briefly outlined here, we refer the reader to \cite[Section~10.7.2]{GLTbookI}.
The matrix of the linear system \eqref{IgA-system}, which is the DDM discretization matrix of this example, is given by
\begin{equation}\label{Anp-dec}
A_{n,p}=n\,T_{n+p-2}(f_p)+R_{n,p},\qquad{\rm rank}(R_{n,p})\le4(p-1)=o(n),
\end{equation}
where
\begin{equation}\label{fp-symbol}
f_p(\theta)=-\phi''_{2p+1}(p+1)-2\sum_{k=1}^p\phi''_{2p+1}(p+1-k)\cos(k\theta)
\end{equation}
and $\phi_{2p+1}$ denotes the cardinal B-spline of degree $2p+1$ over the uniform knot sequence $0,1,\ldots,2p+2$; see \cite[Section~10.7.2]{GLTbookI} for details.
Since $\{n^{-1}R_{n,p}\}_n$ is a sequence of low-rank matrices (and hence it is zero-distributed), it follows from {\bf GLT2}--{\bf GLT3} that
\[ \{n^{-1}A_{n,p}\}_n=\{T_{n+p-2}(f_p)+n^{-1}R_{n,p}\}_n\sim_{\rm GLT}f_p(\theta). \]
\end{example}

\begin{example}[\textbf{decomposition of $(0,1)$ into two overlapping subdomains -- IgA discretization}]\label{exa6}
Consider again problem~\eqref{classic-p}. We decompose the domain $(0,1)$ into two overlapping subdomains
\[ \Omega_1=(0,\beta), \qquad \Omega_2=(\alpha,1), \qquad 0<\alpha<\beta<1. \]
In this example, it is assumed that $\Omega_1$ and $\Omega_2$ are fixed subomains independent of the considered discretization.
In particular, the overlap between $\Omega_1$ and $\Omega_2$ is the fixed interval $(\alpha,\beta)$.

We discretize \eqref{classic-p} on each subdomain by a Galerkin B-spline IgA discretization based on B-splines of degree $p\ge1$ and maximal smoothness $C^{p-1}$ defined over a uniform grid with $n+1$ points.
More precisely, let $u$ be the solution of \eqref{classic-p}, let $u_1=u\big|_{\Omega_1}$, $u_2=u\big|_{\Omega_2}$, and note that $u_1$, $u_2$ satisfy the following problems:
\begin{align}
\label{pOmega1}\textup{(problem on $\Omega_1$)}\quad
\left\{\begin{aligned}
-u_1''(x)&=f(x),\qquad x\in\Omega_1,\\
u_1(0)&=0,\\
u_1(\beta)&=u_2(\beta),
\end{aligned}\right.\\[5pt]
\label{pOmega2}\textup{(problem on $\Omega_2$)}\quad
\left\{\begin{aligned}
-u_2''(x)&=f(x),\qquad x\in\Omega_2,\\
u_2(\alpha)&=u_1(\alpha),\\
u_2(1)&=0,
\end{aligned}\right.
\end{align}
where $u_2(\beta)$ is considered as a known parameter in problem \eqref{pOmega1} and $u_1(\alpha)$ is considered as a known parameter in problem \eqref{pOmega2}.
For $A,B\in\mathbb R$, define
\begin{align*}
H^1_{A,B}(\Omega_1)&=\{v\in H^1(\Omega_1):\,v(0)=A,\ v(\beta)=B\},\\
H^1_{A,B}(\Omega_2)&=\{v\in H^1(\Omega_2):\,v(\alpha)=A,\ v(1)=B\}.
\end{align*}
Note that $H^1_{0,0}(\Omega_1)=H^1_0(\Omega_1)$ and $H^1_{0,0}(\Omega_2)=H^1_0(\Omega_2)$.
The weak forms of problems \eqref{pOmega1}--\eqref{pOmega2} read as follows: find $u_1\in H^1_{0,u_2(\beta)}(\Omega_1)$ such that
\begin{equation}\label{Galpi1}
\int_{\Omega_1}u_1'(x)v'(x){\rm d}x=\int_{\Omega_1}f(x)v(x){\rm d}x,\qquad\forall\,v\in H^1_0(\Omega_1);
\end{equation}
find $u_2\in H^1_{u_1(\alpha),0}(\Omega_2)$ such that
\begin{equation}\label{Galpi2}
\int_{\Omega_2}u_2'(x)v'(x){\rm d}x=\int_{\Omega_2}f(x)v(x){\rm d}x,\qquad\forall\,v\in H^1_0(\Omega_2).
\end{equation}
For $p,n\ge1$ and $A,B\in\mathbb R$, define
\begin{equation*}
\begin{aligned}
\mathcal W_{n,p}^{A,B}(\Omega_1)&=\left\{s\in\mathcal V_{n,p}(\Omega_1):\,s(0)=A,\ s(\beta)=B\right\},\\
\mathcal W_{n,p}^{A,B}(\Omega_2)&=\left\{s\in\mathcal V_{n,p}(\Omega_2):\,s(\alpha)=A,\ s(1)=B\right\},
\end{aligned}
\end{equation*}
where
\begin{equation*}
\begin{aligned}
\mathcal V_{n,p}(\Omega_1)&=\left\{s\in C^{p-1}(\overline\Omega_1):\,s\big|_{\left[\frac{k\beta}{n},\frac{(k+1)\beta}{n}\right)}\in\mathbb P_p\ \textup{for}\ k=0,\ldots,n-1\right\},\\
\mathcal V_{n,p}(\Omega_2)&=\left\{s\in C^{p-1}(\overline\Omega_2):\,s\big|_{\left[\alpha+\frac{k(1-\alpha)}{n},\,\alpha+\frac{(k+1)(1-\alpha)}{n}\right)}\in\mathbb P_p\ \textup{for}\ k=0,\ldots,n-1\right\}.
\end{aligned}
\end{equation*}
The approximations $\tilde u_1$, $\tilde u_2$ of $u_1$, $u_2$ are defined as the solutions of the following (Galerkin) problems, which are the same as problems \eqref{Galpi1}--\eqref{Galpi2} with $H^1_{0,u_2(\beta)}(\Omega_1)$, $H^1_0(\Omega_1)$ and $H^1_{u_1(\alpha),0}(\Omega_2)$, $H^1_0(\Omega_2)$ replaced by $\mathcal W_{n,p}^{0,\tilde u_2(\beta)}(\Omega_1)$, $\mathcal W_{n,p}^{0,0}(\Omega_1)$ and $\mathcal W_{n,p}^{\tilde u_1(\alpha),0}$, $\mathcal W_{n,p}^{0,0}(\Omega_2)$, respectively: find $\tilde u_1\in\mathcal W_{n,p}^{0,\tilde u_2(\beta)}(\Omega_1)$ such that
\begin{equation}\label{Galpi1'}
\int_{\Omega_1}\tilde u_1'(x)v'(x){\rm d}x=\int_{\Omega_1}f(x)v(x){\rm d}x,\qquad\forall\,v\in\mathcal W_{n,p}^{0,0}(\Omega_1);
\end{equation}
find $\tilde u_2\in\mathcal W_{n,p}^{\tilde u_1(\alpha),0}(\Omega_2)$ such that
\begin{equation}\label{Galpi2'}
\int_{\Omega_2}\tilde u_2'(x)v'(x){\rm d}x=\int_{\Omega_2}f(x)v(x){\rm d}x,\qquad\forall\,v\in\mathcal W_{n,p}^{0,0}(\Omega_2).
\end{equation}
Once $\tilde u_1$, $\tilde u_2$ have been found by solving \eqref{Galpi1'}--\eqref{Galpi2'}, the global approximation $\tilde u$ of the solution $u$ over the whole domain $(0,1)$ is defined as
\begin{equation*}
\tilde u(x)=\begin{cases}
\tilde u_1(x), &\quad x\in[0,\alpha),\\[3pt]
\tilde u_2(x), &\quad x\in(\beta,1],\\[3pt]
\alpha_1\tilde u_1(x)+\alpha_2\tilde u_2(x), &\quad x\in[\alpha,\beta],
\end{cases}
\end{equation*}
where $\alpha_1\tilde u_1(x)+\alpha_2\tilde u_2(x)$ is a convex combination of $\tilde u_1(x)$ and $\tilde u_2(x)$, i.e., $\alpha_1,\alpha_2\ge0$ and $\alpha_1+\alpha_2=1$.

Denoting by $\{B_{1,p}^{(1)},\ldots,B_{n+p,p}^{(1)}\}$ and $\{B_{1,p}^{(2)},\ldots,B_{n+p,p}^{(2)}\}$ the B-spline bases of $\mathcal V_{n,p}(\Omega_1)$ and $\mathcal V_{n,p}(\Omega_2)$, respectively, the B-spline bases of $\mathcal W_{n,p}^{0,0}(\Omega_1)$ and $\mathcal W_{n,p}^{0,0}(\Omega_2)$ are given by $\{B_{2,p}^{(1)},\ldots,B_{n+p-1,p}^{(1)}\}$ and $\{B_{2,p}^{(2)},\ldots,B_{n+p-1,p}^{(2)}\}$, respectively.
Write
\[ \tilde u_1=\sum_{j=1}^{n+p}u_{1,j}B_{j,p}^{(1)},\qquad\tilde u_2=\sum_{j=1}^{n+p}u_{2,j}B_{j,p}^{(2)} \]
for unique vectors $\boldsymbol u_1=(u_{1,1},\ldots,u_{1,n+p})^T$ and $\boldsymbol u_2=(u_{2,1},\ldots,u_{2,n+p})^T$. 
By linearity, conditions \eqref{Galpi1'}--\eqref{Galpi2'} are equivalent to $2(n+p-2)$ linear equations on the $2(n+p)$ unknowns given by the components of the vectors $\boldsymbol u_1$ and $\boldsymbol u_2$:
\begin{alignat}{3}
\sum_{j=1}^{n+p}u_{1,j}\int_{\Omega_1}B_{j,p}^{(1)\mbox{\normalsize$'$}}(x)B_{i,p}^{(1)\mbox{\normalsize$'$}}(x){\rm d}x&=\int_{\Omega_1}f(x)B_{i,p}^{(1)}(x){\rm d}x,&\qquad &i=2,\ldots,n+p-1,\label{Galpi1''}\\[3pt]
\sum_{j=1}^{n+p}u_{2,j}\int_{\Omega_2}B_{j,p}^{(2)\mbox{\normalsize$'$}}(x)B_{i,p}^{(2)\mbox{\normalsize$'$}}(x){\rm d}x&=\int_{\Omega_2}f(x)B_{i,p}^{(2)}(x){\rm d}x,&\qquad &i=2,\ldots,n+p-1.\label{Galpi2''}
\end{alignat}
To these $2(n+p-2)$ equations, we must add four additional equations obtained by imposing the Dirichlet boundary conditions and the interface conditions at $x=\alpha,\beta$ that force $\tilde u_1$, $\tilde u_2$ to belong to $\mathcal W_{n,p}^{0,\tilde u_2(\beta)}$, $\mathcal W_{n,p}^{\tilde u_1(\alpha),0}$, respectively:
\begin{alignat}{5}
\tilde u_1(0)&=0&&\quad\iff\quad & u_{1,1}&=0,\label{Dir0}\\
\rd\tilde u_1(\beta)&\rd=\tilde u_2(\beta)&&\rd\quad\iff\quad &\rd u_{1,n+p}&\rd=\sum_{j=1}^{n+p}u_{2,j}B_{j,p}^{(2)}(\beta),\label{Intbeta}\\[3pt]
\rd\tilde u_2(\alpha)&\rd=\tilde u_1(\alpha)&&\rd\quad\iff\quad &\rd u_{2,1}&\rd=\sum_{j=1}^{n+p}u_{1,j}B_{j,p}^{(1)}(\alpha),\label{Intalpha}\\
\tilde u_2(1)&=0&&\quad\iff\quad & u_{2,n+p}&=0.\label{Dir1}
\end{alignat}
Putting together \eqref{Galpi1''}--\eqref{Dir1}, we get the following linear system of $2(n+p)$ equations in the $2(n+p)$ unknowns given by the components of $\boldsymbol u_1$ and $\boldsymbol u_2$:
\begin{equation}\label{forma_suprema}
\begin{cases}
u_{1,1}=0 &\quad\mbox{\footnotesize(Dirichlet boundary condition at $x=0$)}\\[5pt]
\displaystyle\sum_{j=1}^{n+p}u_{1,j}\int_{\Omega_1}B_{j,p}^{(1)\mbox{\normalsize$'$}}(x)B_{i,p}^{(1)\mbox{\normalsize$'$}}(x){\rm d}x=\int_{\Omega_1}f(x)B_{i,p}^{(1)}(x){\rm d}x, &\quad i=2,\ldots,n+p-1\ \ \mbox{\footnotesize(IgA approx.~in $\Omega_1$)}\\[5pt]
\rd\displaystyle u_{1,n+p}=\sum_{j=1}^{n+p}u_{2,j}B_{j,p}^{(2)}(\beta)&\rd\quad\mbox{\footnotesize(interface condition: $C^0$ continuity at $x=\beta$)}\\[12pt]
\hline\\[-12pt]
\rd\displaystyle u_{2,1}=\sum_{j=1}^{n+p}u_{1,j}B_{j,p}^{(1)}(\alpha)&\rd\quad\mbox{\footnotesize(interface condition: $C^0$ continuity at $x=\alpha$)}\\[5pt]
\displaystyle\sum_{j=1}^{n+p}u_{2,j}\int_{\Omega_2}B_{j,p}^{(2)\mbox{\normalsize$'$}}(x)B_{i,p}^{(2)\mbox{\normalsize$'$}}(x){\rm d}x=\int_{\Omega_2}f(x)B_{i,p}^{(2)}(x){\rm d}x,&\quad i=2,\ldots,n+p-1\ \ \mbox{\footnotesize(IgA approx.~in $\Omega_2$)}\\[13.5pt]
u_{2,n+p}=0 &\quad\mbox{\footnotesize(Dirichlet boundary condition at $x=1$)}
\end{cases}
\end{equation}
Actually, in the standard DDM approach, the boundary values $u_{1,1}$, $u_{2,n+p}$ are not considered as unknowns, because they are equal to $0$ due to the Dirichlet boundary conditions, and the interface values $\rd u_{1,n+p}$, $\rd u_{2,1}$ are not considered as unknowns either, because they are directly expressed in terms of the other variables through the red interface conditions as long as $B_{1,p}^{(2)}(\beta)=B_{n+p,p}^{(1)}(\alpha)=0$. From now on we assume that $n\ge\max(1-\alpha,\beta)/(\beta-\alpha)$ so that the supports of the B-splines $B_{1,p}^{(2)}$, $B_{n+p,p}^{(1)}$ are contained in $[\alpha,\beta]$ and, consequently, $B_{1,p}^{(2)}(\beta)=B_{n+p,p}^{(1)}(\alpha)=0$. Under this assumption, not considering $u_{1,1}$, $u_{2,n+p}$, $\rd u_{1,n+p}$, $\rd u_{2,1}$ as unknowns is equivalent to removing the first, last, red equations from the linear system~\eqref{forma_suprema} after replacing everywhere $u_{1,1}$, $u_{2,n+p}$ with $0$ and $\rd u_{1,n+p}$, $\rd u_{2,1}$ with $\rd\sum_{j=2}^{n+p-1}u_{2,j}B_{j,p}^{(2)}(\beta)$, $\rd\sum_{j=2}^{n+p-1}u_{1,j}B_{j,p}^{(1)}(\alpha)$, respectively.\,\footnote{\,Note that $\rd u_{1,n+p}$ is replaced everywhere with $\rd\sum_{j=2}^{n+p-1}u_{2,j}B_{j,p}^{(2)}(\beta)$ and not with $\rd\sum_{j=1}^{n+p}u_{2,j}B_{j,p}^{(2)}(\beta)$, because the Dirichlet boundary condition $u_{2,n+p}=0$ and the assumption $B_{1,p}^{(2)}(\beta)=0$ ``kill'' the first and last terms of the summation. A similar observation applies to $\rd u_{2,1}$.}
The resulting linear system is the following:
\begin{equation}\label{forma_suprema_ridotta}
\begin{cases}
\displaystyle\sum_{j=2}^{n+p-1}u_{1,j}\int_{\Omega_1}B_{j,p}^{(1)\mbox{\normalsize$'$}}(x)B_{i,p}^{(1)\mbox{\normalsize$'$}}(x){\rm d}x+\sum_{j=2}^{n+p-1}u_{2,j}B_{j,p}^{(2)}(\beta)\int_{\Omega_1}B_{n+p,p}^{(1)\mbox{\normalsize$'$}}(x)B_{i,p}^{(1)\mbox{\normalsize$'$}}(x){\rm d}x=\int_{\Omega_1}f(x)B_{i,p}^{(1)}(x){\rm d}x,\\[15pt]
i=2,\ldots,n+p-1\ \ \mbox{\footnotesize(IgA approx.~in $\Omega_1$)}\\[7.5pt]
\hline\\[-7.5pt]
\displaystyle\sum_{j=2}^{n+p-1}u_{1,j}B_{j,p}^{(1)}(\alpha)\int_{\Omega_2}B_{1,p}^{(2)\mbox{\normalsize$'$}}(x)B_{i,p}^{(2)\mbox{\normalsize$'$}}(x){\rm d}x+\sum_{j=2}^{n+p-1}u_{2,j}\int_{\Omega_2}B_{j,p}^{(2)\mbox{\normalsize$'$}}(x)B_{i,p}^{(2)\mbox{\normalsize$'$}}(x){\rm d}x=\int_{\Omega_2}f(x)B_{i,p}^{(2)}(x){\rm d}x,\\[15pt]
i=2,\ldots,n+p-1\ \ \mbox{\footnotesize(IgA approx.~in $\Omega_2$)}
\end{cases}
\end{equation}
For later convenience, we multiply the first $n+p-2$ equations in \eqref{forma_suprema_ridotta} by $\beta$ and the second $n+p-2$ equations in \eqref{forma_suprema_ridotta} by $1-\alpha$. In this way, we obtain the equivalent linear system
\begin{equation}\label{the_sys}
B_{n,p,\alpha,\beta}\left[\begin{array}{c}u_{1,2}\\\vdots\\ u_{1,n+p-1}\vphantom{\Big|}\\\hline u_{2,2}\vphantom{\Big|}\\\vdots\\ u_{2,n+p-1}\end{array}\right]=\boldsymbol f_{n,p,\alpha,\beta},
\end{equation}
where the DDM discretization matrix $B_{n,p,\alpha,\beta}$ and the right-hand side $\boldsymbol f_{n,p,\alpha,\beta}$ are given by
\begin{equation}\label{mat_rhs}
B_{n,p,\alpha,\beta}=\left[\begin{array}{c|c}\beta A_{n,p,\alpha,\beta}^{(1)} & \beta R_{n,p,\alpha,\beta}^{(1)} \vphantom{\Big|}\\ \hline (1-\alpha)R_{n,p,\alpha,\beta}^{(2)} & (1-\alpha)A_{n,p,\alpha,\beta}^{(2)}\vphantom{\Big|}\end{array}\right],\qquad\boldsymbol f_{n,p,\alpha,\beta}=\left[\begin{array}{c}
\beta\left[\displaystyle\int_{\Omega_1}f(x)B_{i,p}^{(1)}(x){\rm d}x\right]_{i=2}^{n+p-1}\vphantom{\Bigg|_{\int_0^1}}\\
\hline
(1-\alpha)\left[\displaystyle\int_{\Omega_2}f(x)B_{i,p}^{(2)}(x){\rm d}x\right]_{i=2}^{n+p-1}\vphantom{\Bigg|^{\int_0^1}}\end{array}\right],
\end{equation}
with
\begin{alignat*}{3}
A_{n,p,\alpha,\beta}^{(1)}&=\left[\int_{\Omega_1}B_{j,p}^{(1)\mbox{\normalsize$'$}}(x)B_{i,p}^{(1)\mbox{\normalsize$'$}}(x){\rm d}x\right]_{i,j=2}^{n+p-1},&\qquad R_{n,p,\alpha,\beta}^{(1)}&=\left[B_{j,p}^{(2)}(\beta)\int_{\Omega_1}B_{n+p,p}^{(1)\mbox{\normalsize$'$}}(x)B_{i,p}^{(1)\mbox{\normalsize$'$}}(x){\rm d}x\right]_{i,j=2}^{n+p-1},\\
R_{n,p,\alpha,\beta}^{(2)}&=\left[B_{j,p}^{(1)}(\alpha)\int_{\Omega_2}B_{1,p}^{(2)\mbox{\normalsize$'$}}(x)B_{i,p}^{(2)\mbox{\normalsize$'$}}(x){\rm d}x\right]_{i,j=2}^{n+p-1},&\qquad A_{n,p,\alpha,\beta}^{(2)}&=\left[\int_{\Omega_2}B_{j,p}^{(2)\mbox{\normalsize$'$}}(x)B_{i,p}^{(2)\mbox{\normalsize$'$}}(x){\rm d}x\right]_{i,j=2}^{n+p-1}.
\end{alignat*}
By the properties of B-splines, we have $B_{j,p}^{(1)}(x)=B_{j,p}(x/\beta)$ and $B_{j,p}^{(2)}(x)=B_{j,p}((x-\alpha)/(1-\alpha))$ for every $j=1,\ldots,n+p$, where $\{B_{1,p},\ldots,B_{n+p,p}\}$ is the B-spline basis of $\mathcal V_{n,p}$ considered in Example~\ref{exa5}. Moreover, $B_{j,p}^{(2)}(\beta)$ and $B_{j,p}^{(1)}(\alpha)$ are nonzero for at most $p+1$ indices $j$, hence the nonzero columns of $R_{n,p,\alpha,\beta}^{(1)}$ and $R_{n,p,\alpha,\beta}^{(2)}$ are at most $p+1$. Thus, denoting again by $A_{n,p}$ the DDM matrix of Example~\ref{exa5} appearing in \eqref{IgA-system}--\eqref{Anp-dec} and by $f_p$ the function \eqref{fp-symbol}, we have
\begin{alignat*}{3}
\beta A_{n,p,\alpha,\beta}^{(1)}&=A_{n,p}=n\,T_{n+p-2}(f_p)+R_{n,p},&\qquad{\rm rank}(R_{n,p})&\le4(p-1)=o(n),\\[3pt]
{\rm rank}(R_{n,p,\alpha,\beta}^{(1)})&\le p+1=o(n),\\[3pt]
{\rm rank}(R_{n,p,\alpha,\beta}^{(2)})&\le p+1=o(n),\\[3pt]
(1-\alpha)A_{n,p,\alpha,\beta}^{(2)}&=A_{n,p}=n\,T_{n+p-2}(f_p)+R_{n,p},&\qquad{\rm rank}(R_{n,p})&\le4(p-1)=o(n).
\end{alignat*}
It follows that
\begin{align*}
n^{-1}B_{n,p,\alpha,\beta}&=\left[\begin{array}{c|c}T_{n+p-2}(f_p) & O \vphantom{\Big|}\\ \hline O & T_{n+p-2}(f_p)\vphantom{\Big|}\end{array}\right]+\left[\begin{array}{c|c}n^{-1}R_{n,p} & \beta n^{-1}R_{n,p,\alpha,\beta}^{(1)} \vphantom{\Big|}\\ \hline (1-\alpha)n^{-1}R_{n,p,\alpha,\beta}^{(2)} & n^{-1}R_{n,p}\vphantom{\Big|}\end{array}\right]\\
&=T_{2(n+p-2)}(f_p)+S_{n,p,\alpha,\beta},
\end{align*}
where
\begin{align*}
S_{n,p,\alpha,\beta}&=\underbrace{\left[\begin{array}{c|c}T_{n+p-2}(f_p) & O \vphantom{\Big|}\\ \hline O & T_{n+p-2}(f_p)\vphantom{\Big|}\end{array}\right]-T_{2(n+p-2)}(f_p)}_{{\rm rank}\,\le\,2p}\,+\,\underbrace{\left[\begin{array}{c|c}n^{-1}R_{n,p} & \beta n^{-1}R_{n,p,\alpha,\beta}^{(1)} \vphantom{\Big|}\\ \hline (1-\alpha)n^{-1}R_{n,p,\alpha,\beta}^{(2)} & n^{-1}R_{n,p}\vphantom{\Big|}\end{array}\right]}_{{\rm rank}\,\le\,8(p-1)+2(p+1)},\\
{\rm rank}(S_{n,p,\alpha,\beta})&\le12p-6=o(n).
\end{align*}
Since $\{n^{-1}S_{n,p,\alpha,\beta}\}_n$ is a sequence of low-rank matrices (and hence it is zero-distributed), it follows from {\bf GLT2}--{\bf GLT3} that
\[ \{n^{-1}B_{n,p,\alpha,\beta}\}_n=\{T_{2(n+p-2)}(f_p)+S_{n,p,\alpha,\beta}\}_n\sim_{\rm GLT}f_p(\theta). \]
Note that the symbol $f_p(\theta)$ is the same as the symbol of the GLT sequence $\{n^{-1}A_{n,p}\}_n$ in Example~\ref{exa5} and, in particular, it is independent of $\alpha$ and $\beta$.
\end{example}

In Examples~\ref{exa3}--\ref{exa4}, we denote by $\ee_1,\ldots,\ee_d$ the vectors of the canonical basis of $\mathbb R^d$.

\begin{example}[\textbf{decomposition of $(0,1)^d$ into a unique subdomain -- FD discretization}]\label{exa3}
Let $d$ be a positive integer and consider the second-order boundary value problem
\begin{equation}\label{classic-p2}
\begin{cases}
-\Delta u(\xx)=f(\xx), &\quad\xx\in(0,1)^d,\\
u(\xx)=0,&\quad\xx\in\partial((0,1)^d).
\end{cases}
\end{equation}
Let $\nn\in\mathbb N^d$, set $\hh=(\nn+\bu)^{-1}$ and $\xx_\ii=\ii\hh$ for $\ii=\bz,\ldots,\nn+\bu$.
We decompose the domain $(0,1)^d$ into a unique subdomain equal to itself. Then, we look for an approximation $\tilde u$ of the solution $u$ over the unique subdomain $(0,1)^d$. To this end, we use the classical $(2d+1)$-point central FD approximation of the $d$-dimensional Laplace operator:
\begin{align*}
-\Delta u(\xx_\ii)=\sum_{k=1}^d-\frac{\partial^2u}{\partial x_k^2}(\xx_\ii)&\approx\sum_{k=1}^d\frac{-u(\xx_\ii-h_k\ee_k)+2u(\xx_\ii)-u(\xx_\ii+h_k\ee_k)}{h_k^2}\\
&=\sum_{k=1}^d\frac{-u(\xx_{\ii-\ee_k})+2u(\xx_\ii)-u(\xx_{\ii+\ee_k})}{h_k^2},\qquad\ii=\bu,\ldots,\nn.
\end{align*}
This means that the values of the solution $u$ at the points $\xx_\ii$, $\ii=\bu,\ldots,\nn$, satisfy (approximately) the following linear system:
\begin{equation*}
\sum_{k=1}^d\frac{-u(\xx_{\ii-\ee_k})+2u(\xx_\ii)-u(\xx_{\ii+\ee_k})}{h_k^2}=f(\xx_\ii),\qquad\ii=\bu,\ldots,\nn.
\end{equation*}
We then approximate the nodal value $u(\xx_\ii)$ with the value $u_\ii$ for $\ii=\bz,\ldots,\nn+\bu$, where $u_\ii=0$ for $\ii\not\in\{\bu,\ldots,\nn\}$ due to the Dirichlet boundary conditions and the vector $(u_\bu,\ldots,u_\nn)$ solves the linear system
\begin{equation}\label{i-form2}
\sum_{k=1}^d\frac{-u_{\ii-\ee_k}+2u_\ii-u_{\ii+\ee_k}}{h_k^2}=f(\xx_\ii),\qquad\ii=\bu,\ldots,\nn.
\end{equation}
The approximation $\tilde u$ of the solution $u$ is defined as any function on $[0,1]$ such that $\tilde u(\xx_\ii)=u_\ii$ for every $\ii=\bz,\ldots,\nn+\bu$.
The matrix of the linear system \eqref{i-form2}, which is the DDM discretization matrix of this example, is a $d$-level matrix given by
\begin{equation*}
A_\nn=\sum_{k=1}^d\frac1{h_k^2}T_\nn(2-2\cos\theta_k),
\end{equation*}
where $2-2\cos\theta_k$ is a $d$-variate function in the variables $\btheta=(\theta_1,\ldots,\theta_d)\in[-\pi,\pi]^d$ that depends only on the $k$th variable $\theta_k$; see \cite[Section~7.3]{GLTbookII} for details.

Suppose now that $\nn+\bu=\bnu n$ for some positive integer (mesh fineness parameter) $n$ and for a fixed vector $\bnu$ independent of $n$.
It is understood that $\bnu\in\mathbb Q^d$, $\bnu>\bz$, and $n$ belongs to the set of positive integers such that $\nn+\bu=\bnu n\in\mathbb N^d$.
This assumption, which is usually satisfied in practice, says that each mesh size parameter $h_k=(n_k+1)^{-1}=(\nu_kn)^{-1}$ tends to $0$ as $n\to\infty$ with the same asymptotic speed as the others. In a common scenario, one normally takes $\bnu=\bu$, i.e., $\nn+\bu=(n,\ldots,n)$, which means that the mesh size parameter $h_k=n^{-1}$ is the same in each direction $k=1,\ldots,d$.
By {\bf GLT2}--{\bf GLT3}, we have
\[ \{n^{-2}A_\nn\}_n=\left\{\sum_{k=1}^d\nu_k^2\hspace{1pt}T_\nn(2-2\cos\theta_k)\right\}_n\sim_{\rm GLT}\sum_{k=1}^d\nu_k^2(2-2\cos\theta_k). \]
\end{example}

\begin{example}[\textbf{decomposition of $(0,1)^d$ into two subdomains -- FD discretization}]\label{exa4}

\begin{figure}
\centering
\includegraphics[width=0.5\textwidth]{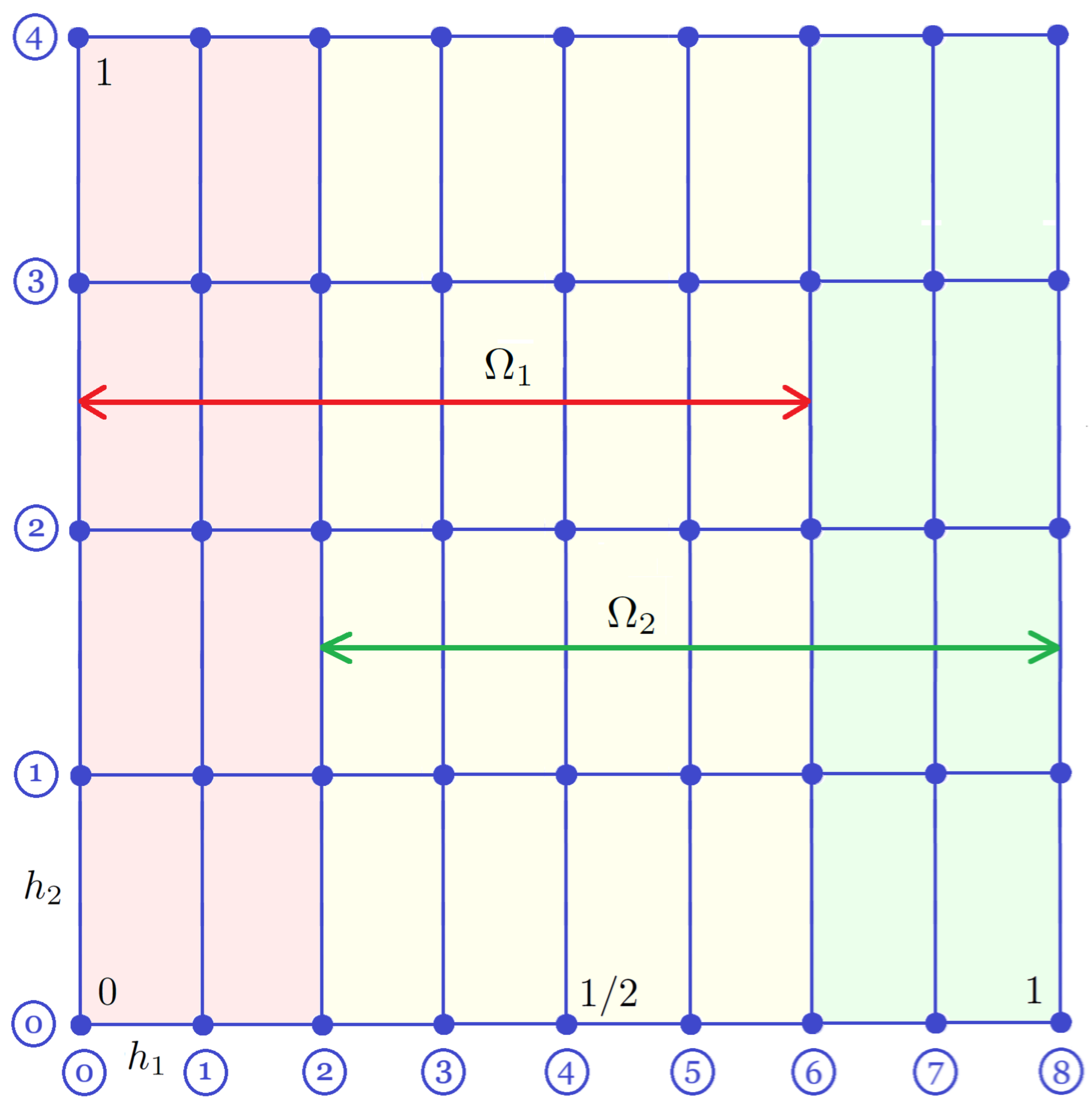}
\caption{Illustration for Example~\ref{exa4} in the case $d=2$, $\nn=(n_1,n_2)=(7,3)$ and $v=2$.}
\label{DD2d}
\end{figure}

Consider again problem \eqref{classic-p2}, let $\nn\in\mathbb N^d$, set $\hh=(\nn+\bu)^{-1}$ and $\xx_\ii=\ii\hh$ for $\ii=\bz,\ldots,\nn+\bu$. We decompose the domain $(0,1)^d$ into two subdomains $\Omega_1=\Omega_{1,\nn}=(0,1/2+vh_1)\times(0,1)^{d-1}$ and $\Omega_2=\Omega_{2,\nn}=(1/2-vh_1,1)\times(0,1)^{d-1}$, where $v\le(n_1+1)/2$ is a non-negative integer that determines the size of the overlap; see Figure~\ref{DD2d}.
In practice, the overlap $(1/2-vh_1,1/2+vh_1)\times(0,1)^{d-1}$ consists of $2v(n_2+1)\cdots(n_d+1)$ mesh cells. As in Example~\ref{exa2}, depending on $v$, the overlap may or may not vanish in the limit of mesh refinement $\nn\to\infty$.

Let us assume that $n_1$ is odd as in Figure~\ref{DD2d}, so that the $(d-1)$-dimensional hyperplane $\{\xx\in\mathbb R^d:x_1=1/2\}$ contains a side of the mesh.
Then, we look for two approximations of the solution $u$:
\begin{itemize}[nolistsep,leftmargin=*]
	\item an approximation $\tilde u_1$ of $u_1=u\big|_{\Omega_1}$ defined over the subdomain $\Omega_1$;
	\item an approximation $\tilde u_2$ of $u_2=u\big|_{\Omega_2}$ defined over the subdomain $\Omega_2$.
\end{itemize}
Note that the set of grid points in $\overline\Omega_1$ is given by $\{\xx_\ii:\ii\in\widehat\Omega_1\}$ with $\widehat\Omega_1=\{\bz,\ldots,((n_1+1)/2+v,n_2+1,\ldots,n_d+1)\}$ and the set of grid points in $\overline\Omega_2$ is given by $\{\xx_\ii:\ii\in\widehat\Omega_2\}$ with $\widehat\Omega_2=\{((n_1+1)/2-v,0,\ldots,0),\ldots,\nn+\bu\}$, and define
\begin{alignat}{3}
\tilde u_1(\xx_\ii)&=u_{1,\ii}, &\qquad\ii&\in\widehat\Omega_1=\left\{\bz,\ldots,\left(\dfrac{n_1+1}2+v,n_2+1,\ldots,n_d+1\right)\right\},\label{uk1'}\\
\tilde u_2(\xx_\ii)&=u_{2,\ii}, &\qquad\ii&\in\widehat\Omega_2=\left\{\left(\dfrac{n_1+1}2-v,0,\ldots,0\right),\ldots,\nn+\bu\right\}.\label{uk2'}
\end{alignat}
Determining $\tilde u_1$ and $\tilde u_2$ is equivalent to determining the $(n_1+2v+3)(n_2+1)\cdots(n_d+1)$ unknowns \eqref{uk1'}--\eqref{uk2'}. Once this is done, $\tilde u_1$ is defined as any function on $\overline\Omega_1=[0,1/2+vh_1]\times[0,1]^{d-1}$ satisfying \eqref{uk1'}, $\tilde u_2$ is defined as any function on $\overline\Omega_2=[1/2-vh_1,1]\times[0,1]^{d-1}$ satisfying \eqref{uk2'}, and the global approximation $\tilde u$ of the solution $u$ over the whole domain $(0,1)^d$ is defined as
\begin{equation*}
\tilde u(\xx)=\begin{cases}
\tilde u_1(\xx), &\quad\xx\in[0,1/2-vh_1)\times[0,1]^{d-1},\\[3pt]
\tilde u_2(\xx), &\quad\xx\in(1/2+vh_1,1]\times[0,1]^{d-1},\\[3pt]
\alpha_1\tilde u_1(\xx)+\alpha_2\tilde u_2(\xx), &\quad\xx\in[1/2-vh_1,1/2+vh_1]\times[0,1]^{d-1},
\end{cases}
\end{equation*}
where $\alpha_1\tilde u_1(\xx)+\alpha_2\tilde u_2(\xx)$ is a convex combination of $\tilde u_1(\xx)$ and $\tilde u_2(\xx)$, i.e., $\alpha_1,\alpha_2\ge0$ and $\alpha_1+\alpha_2=1$.

There are several ways to determine $\tilde u_1$ and $\tilde u_2$, i.e., the $(n_1+2v+3)(n_2+2)\cdots(n_d+2)$ unknowns \eqref{uk1'}--\eqref{uk2'}.
In this example, we propose a way that works in the overlapping case $v\ne0$.
Specifically, we use the classical $(2d+1)$-point central FD approximation of the $d$-dimensional Laplace operator at the $((n_1+1)/2+v-2)n_2\cdots n_d$ points $x_\ii$, $\ii=\bu,\ldots,((n_1+1)/2+v-2,n_2,\ldots,n_d)$, lying inside $\Omega_1$ and at the $((n_1+1)/2+v-2)n_2\cdots n_d$ points $x_\ii$, $\ii=((n_1+1)/2-v+2,1,\ldots,1),\ldots,\nn$, lying inside $\Omega_2$. In addition, we impose the Dirichlet boundary conditions 
and two additional conditions at each of the $2n_2\cdots n_d$ interface points $\xx_\ii$, $\ii\in\{\bu,\ldots,\nn\}\cap\{\xx\in\mathbb R^d:x_1=1/2\pm vh_1\}=\{(i_1,1,\ldots,1),\ldots,(i_1,n_2,\ldots,n_d):i_1=(n_1+1)/2\pm v\}$, in order to make the number of equations equal to the number $(n_1+2v+3)(n_2+2)\cdots(n_d+2)$ of unknowns.
The resulting linear system is the following:
\begin{equation}\label{i-form''}
\begin{cases}
u_{1,\ii}=0 &\quad\ii\in(\{\bz,\ldots,\nn+\bu\}\setminus\{\bu,\ldots,\nn\})\cap\widehat\Omega_1\\[7.5pt]
\displaystyle\sum_{k=1}^d\frac{-u_{1,\ii-\ee_k}+2u_{1,\ii}-u_{1,\ii+\ee_k}}{h_k^2}=f(\xx_\ii) &\quad\ii=\bu,\ldots,\left(\dfrac{n_1+1}2+v-2,n_2,\ldots,n_d\right)\\[7.5pt]
\rd u_{1,\ii}=u_{2,\ii}&\rd\quad\ii=\left(\dfrac{n_1+1}2+v,1,\ldots,1\right),\ldots\left(\dfrac{n_1+1}2+v,n_2,\ldots,n_d\right)\\[7.5pt]
\bl u_{1,\ii}-u_{1,\ii-\ee_1}=u_{2,\ii+\ee_1}-u_{2,\ii}&\bl\quad\ii=\left(\dfrac{n_1+1}2+v,1,\ldots,1\right),\ldots\left(\dfrac{n_1+1}2+v,n_2,\ldots,n_d\right)\\[10pt]
\hline\\[-10pt]
\rd u_{2,\ii}=u_{1,\ii}&\rd\quad\ii=\left(\dfrac{n_1+1}2-v,1,\ldots,1\right),\ldots\left(\dfrac{n_1+1}2-v,n_2,\ldots,n_d\right)\\[7.5pt]
\bl u_{2,\ii+\ee_1}-u_{2,\ii}=u_{1,\ii}-u_{1,\ii-\ee_1}&\bl\quad\ii=\left(\dfrac{n_1+1}2-v,1,\ldots,1\right),\ldots\left(\dfrac{n_1+1}2-v,n_2,\ldots,n_d\right)\\[7.5pt]
\displaystyle\sum_{k=1}^d\frac{-u_{2,\ii-\ee_k}+2u_{2,\ii}-u_{2,\ii+\ee_k}}{h_k^2}=f(\xx_\ii) &\quad\ii=\left(\dfrac{n_1+1}2-v+2,1,\ldots,1\right),\ldots,\nn\\[15pt]
u_{2,\ii}=0 &\quad\ii\in(\{\bz,\ldots,\nn+\bu\}\setminus\{\bu,\ldots,\nn\})\cap\widehat\Omega_2
\end{cases}
\end{equation}
In the standard DDM approach, the values
\begin{alignat}{3}
&u_{1,\ii}, &\qquad\ii&\in(\{\bz,\ldots,\nn+\bu\}\setminus\{\bu,\ldots,\nn\})\cap\widehat\Omega_1,\label{b1}\\
&\rd u_{1,\ii}, &\rd\qquad\ii&\rd=\left(\dfrac{n_1+1}2+v,1,\ldots,1\right),\ldots\left(\dfrac{n_1+1}2+v,n_2,\ldots,n_d\right),\label{r1}\\
&\rd u_{2,\ii}, &\rd\qquad\ii&\rd=\left(\dfrac{n_1+1}2-v,1,\ldots,1\right),\ldots\left(\dfrac{n_1+1}2-v,n_2,\ldots,n_d\right),\label{r2}\\
&u_{2,\ii}, &\qquad\ii&\in(\{\bz,\ldots,\nn+\bu\}\setminus\{\bu,\ldots,\nn\})\cap\widehat\Omega_2,\label{b2}
\end{alignat}
are not considered as unknowns, since they are specified by the Dirichlet boundary conditions and the red interface conditions. Not considering \eqref{b1}--\eqref{b2} as unknowns is equivalent to removing the first, last, red equations from the linear system \eqref{i-form''} after replacing everywhere \eqref{b1}--\eqref{b2} with their values specified by the first, last, red equations. The resulting linear system is the following:
\begin{equation}\label{i-form'-reduced'}
\begin{cases}
\displaystyle\sum_{k=1}^d\frac{-u_{1,\ii-\ee_k}+2u_{1,\ii}-u_{1,\ii+\ee_k}}{h_k^2}=f(\xx_\ii) &\quad\ii=\bu,\ldots,\left(\dfrac{n_1+1}2+v-2,n_2,\ldots,n_d\right)\\[11pt]
\bl-u_{1,\ii-\ee_1}+2u_{2,\ii}-u_{2,\ii+\ee_1}=0&\bl\quad\ii=\left(\dfrac{n_1+1}2+v,1,\ldots,1\right),\ldots\left(\dfrac{n_1+1}2+v,n_2,\ldots,n_d\right)\\[10pt]
\hline\\[-10pt]
\bl-u_{1,\ii-\ee_1}+2u_{1,\ii}-u_{2,\ii+\ee_1}=0&\bl\quad\ii=\left(\dfrac{n_1+1}2-v,1,\ldots,1\right),\ldots\left(\dfrac{n_1+1}2-v,n_2,\ldots,n_d\right)\\[7.5pt]
\displaystyle\sum_{k=1}^d\frac{-u_{2,\ii-\ee_k}+2u_{2,\ii}-u_{2,\ii+\ee_k}}{h_k^2}=f(\xx_\ii) &\quad\ii=\left(\dfrac{n_1+1}2-v+2,1,\ldots,1\right),\ldots,\nn
\end{cases}
\end{equation}
The matrix of the linear system \eqref{i-form'-reduced'}, which is the DDM discretization matrix of this example, is given by
\begin{equation*}
A_\nn=\sum_{k=1}^d\frac1{h_k^2}T_{\nn+(2v-1)\ee_1}(2-2\cos\theta_k)+R_\nn,\qquad{\rm rank}(R_\nn)\le2n_2\cdots n_d,
\end{equation*}
where $2-2\cos\theta_k$ is a $d$-variate function in the variables $\btheta=(\theta_1,\ldots,\theta_d)\in[-\pi,\pi]^d$ that depends only on the $k$th variable $\theta_k$, as in Example~\ref{exa3}.

Suppose now that $\nn+\bu=\bnu n$ as in Example~\ref{exa3}. Then,
\begin{align*}
{\rm size}(R_\nn)&={\rm size}(A_\nn)=(n_1+2v-1)n_2\cdots n_d\ge(n_1-1)n_2\cdots n_d\sim\nu_1\nu_2\cdots\nu_dn^d,\\
{\rm rank}(R_\nn)&\le2n_2\cdots n_d\sim2\nu_2\cdots\nu_d n^{d-1}=o(n^d).
\end{align*}
It follows that $\{n^{-2}R_\nn\}_n$ is a sequence of low-rank matrices and hence it is zero-distributed. Thus, by {\bf GLT2}--{\bf GLT3},
\[ \{n^{-2}A_\nn\}_n=\left\{\sum_{k=1}^d\nu_k^2\hspace{1pt}T_\nn(2-2\cos\theta_k)+n^{-2}R_\nn\right\}_n\sim_{\rm GLT}\sum_{k=1}^d\nu_k^2(2-2\cos\theta_k). \]
\end{example}


\begin{remark}\label{overlap-cases}
In the examples of this section, we have considered DDM discretizations where the subdomain overlap may or may not vanish in the limit of mesh refinement $n\to\infty$.
In most cases, the subdomain overlap consists of a few mesh cells---as it happens in Examples~\ref{exa2} and~\ref{exa4} for a fixed $v$ independent of $n$---and therefore it vanishes in the limit of mesh refinement. The choice of a subdomain overlap that vanishes as $n\to\infty$ is usually made in order to keep to a minimum the computational cost of computing twice the solution on the overlap region.
However, there are important applications where one or more subdomain overlaps do not vanish in the limit of mesh refinement.
This happens, for instance, when dealing with domain truncation via perfectly matched layers~\cite{Ber94} or absorbing boundary conditions~\cite{EM77}; see, e.g., the model problem in \cite[Section~2]{gander_jakabcin_outrata}.
\end{remark}

\section{BJ/BGS/AS/MS preconditioners}\label{bJ/bGS/AS/MS-def-e}

In this section, we formally define the BJ/BGS/AS/MS preconditioners for multilevel block matrices.
The definitions provided herein, as well as the associated notations, are formulated on an abstract level and apply to arbitrary multilevel block matrices, not necessarily arising from numerical discretizations. Such definitions and notations, which are inspired by the theory of reduced GLT sequences \cite{rg}, are proposed as alternatives to those commonly used by the DDM community. In order for the reader to become familiar with them, a number of illustrative examples is provided.

We begin by introducing zeroing, restriction and expansion operators, which are completely analogous to those appearing within the theory of reduced GLT sequences \cite[Section~4]{rg}. As in the case of multigrid methods, one may use ``cutting'' or ``projection'' as alternatives to ``restriction'', and ``extension'' or ``prolongation'' as alternatives to ``expansion''. Throughout this paper, however, we only use ``restriction'' and ``expansion''.
In what follows, the characteristic (indicator) function of a set $\Omega$ is denoted by $\chi_\Omega$.

\begin{definition}[\textbf{restriction matrix}]\label{R-matrix}
Let $d,s$ be positive integers, let $\nn\in\mathbb N^d$, and let $\Omega\subseteq[0,1]^d$. The restriction matrix $R_\Omega^{\nn,s}$ is defined as the $N_\nn^\Omega s\times N(\nn)s$ matrix obtained from $D_\nn(\chi_\Omega I_s)$ by removing the zero (block) rows, i.e., the (block) rows corresponding to multi-indices $\ii\not\in\II_\nn^\Omega$:
\[ R_\Omega^{\nn,s}=[(\ee_\ii^{(\nn)})^T\otimes I_s]_{\ii\in\II_\nn^\Omega}=[(\ee_\ii^{(\nn)})^T]_{\ii\in\II_\nn^\Omega}\otimes I_s=R_\Omega^{\nn,1}\otimes I_s. \]
\end{definition}

A few basic properties of restriction matrices are collected in Lemma~\ref{r-lemma}.

\begin{lemma}\label{r-lemma}
Let $d,s$ be positive integers, let $\nn\in\mathbb N^d$, and let $\Omega\subseteq[0,1]^d$. Then,
\begin{equation}\label{basicReq}
R_\Omega^{\nn,s}(R_\Omega^{\nn,s})^T=I_{N_\nn^\Omega s},\qquad(R_\Omega^{\nn,s})^TR_\Omega^{\nn,s}=D_\nn(\chi_\Omega I_s),
\end{equation}
and
\begin{equation}\label{basicReq'}
R_\Omega^{\nn,s}D_\nn(\chi_\Omega I_s)=R_\Omega^{\nn,s},\qquad D_\nn(\chi_\Omega I_s)(R_\Omega^{\nn,s})^T=(R_\Omega^{\nn,s})^T.
\end{equation}
\end{lemma}
\begin{proof}
The first equation in \eqref{basicReq} follows from the definition of $R_\Omega^{\nn,s}$ and from standard block matrix multiplications.
Alternatively, we can prove it by using the properties of tensor products:
\begin{align*}
R_\Omega^{\nn,s}(R_\Omega^{\nn,s})^T&=(R_\Omega^{\nn,1}\otimes I_s)(R_\Omega^{\nn,1}\otimes I_s)^T=(R_\Omega^{\nn,1}\otimes I_s)((R_\Omega^{\nn,1})^T\otimes I_s)=R_\Omega^{\nn,1}(R_\Omega^{\nn,1})^T\otimes I_s\\
&=[(\ee_\ii^{(\nn)})^T]_{\ii\in\II_\nn^\Omega}[\ee_\jj^{(\nn)}]^{\jj\in\II_\nn^\Omega}\otimes I_s=[\delta_{\ii\jj}]_{\ii,\jj\in\II_\nn^\Omega}\otimes I_s=I_{N_\nn^\Omega}\otimes I_s=I_{N_\nn^\Omega s}.
\end{align*}
Similarly, we can prove the second equation in \eqref{basicReq}: 
\begin{align*}
(R_\Omega^{\nn,s})^TR_\Omega^{\nn,s}&=(R_\Omega^{\nn,1}\otimes I_s)^T(R_\Omega^{\nn,1}\otimes I_s)=((R_\Omega^{\nn,1})^T\otimes I_s)(R_\Omega^{\nn,1}\otimes I_s)=(R_\Omega^{\nn,1})^TR_\Omega^{\nn,1}\otimes I_s\\
&=[(\ee_\jj^{(\nn)})]^{\jj\in\II_\nn^\Omega}[(\ee_\ii^{(\nn)})^T]_{\ii\in\II_\nn^\Omega}\otimes I_s=\Biggl(\,\sum_{\ii\in\II_\nn^\Omega}\underbrace{\ee_\ii^{(\nn)}(\ee_\ii^{(\nn)})^T}_{\substack{\textup{matrix with $1$}\\\textup{in position $(\ii,\ii)$}\\\textup{and $0$ elsewhere}}}\Biggr)\otimes I_s=D_\nn(\chi_\Omega)\otimes I_s=D_\nn(\chi_\Omega I_s).
\end{align*}
The first equation in \eqref{basicReq'} is obtained by right-multiplying the first equation in \eqref{basicReq} by $R_\Omega^{\nn,s}$ and by using the second equation in \eqref{basicReq}.
The second equation in \eqref{basicReq'} is obtained by right-multiplying the second equation in \eqref{basicReq} by $(R_\Omega^{\nn,s})^T$ and by using the first equation in \eqref{basicReq}.
\end{proof}

We are now ready to define zeroing, restriction and expansion operators.

\begin{definition}[\textbf{zeroing, restriction and expansion operators}]\label{ZRE-operators}
Let $d,s,t$ be positive integers, let $\nn\in\mathbb N^d$, and let $\Omega_L,\Omega_R\subseteq[0,1]^d$.
\begin{itemize}[nolistsep,leftmargin=*]
	\item We define the zeroing operator $Z^{\nn,s,t}_{\Omega_L,\Omega_R}$ as follows:
	\[ Z^{\nn,s,t}_{\Omega_L,\Omega_R}:\mathbb C^{N(\nn)s\times N(\nn)t}\to\mathbb C^{N(\nn)s\times N(\nn)t},\qquad Z_{\Omega_L,\Omega_R}^{\nn,s,t}(A)=D_\nn(\chi_{\Omega_L}I_s)\,A\,D_\nn(\chi_{\Omega_R}I_t). \]
	In other words, if $A=[A_{\ii\jj}]_{\ii,\jj=\bu}^\nn\in\mathbb C^{N(\nn)s\times N(\nn)t}$ is a multilevel block matrix with blocks $A_{\ii\jj}\in\mathbb C^{s\times t}$ for every $\ii,\jj=\bu,\ldots,\nn$, then the $s\times t$ block in position $(\ii,\jj)$ of the multilevel block matrix $Z_{\Omega_L,\Omega_R}^{\nn,s,t}(A)$ is given by
	\begin{align*}	(Z_{\Omega_L,\Omega_R}^{\nn,s,t}(A))_{\ii\jj}&=(D_\nn(\chi_{\Omega_L}I_s)\,A\,D_\nn(\chi_{\Omega_R}I_t))_{\ii\jj}=(D_\nn(\chi_{\Omega_L}I_s))_{\ii\ii}\,A_{\ii\jj}\,(D_\nn(\chi_{\Omega_R}I_t))_{\jj\jj}\\	&=\chi_{\Omega_L}\!\Bigl(\frac\ii\nn\Bigr)I_s\,A_{\ii\jj}\,\chi_{\Omega_R}\!\Bigl(\frac\jj\nn\Bigr)I_t=\chi_{\Omega_L}\!\Bigl(\frac\ii\nn\Bigr)\chi_{\Omega_R}\!\Bigl(\frac\jj\nn\Bigr)A_{\ii\jj}=\begin{cases}
	A_{\ii\jj}, &\mbox{if $\ii\in\II_\nn^{\Omega_L}$ and $\jj\in\II_\nn^{\Omega_R}$},\\
	O_{s,t}, &\mbox{otherwise}.
	\end{cases}
	\end{align*}
	In practice, the action of the operator $Z_{\Omega_L,\Omega_R}^{\nn,s,t}$ on a multilevel block matrix $A$ of size $N(\nn)s\times N(\nn)t$ consists in zeroing out all (block) rows of $A$ corresponding to multi-indices $\ii\not\in\II_\nn^{\Omega_L}$ and all (block) columns of $A$ corresponding to multi-indices $\jj\not\in\II_\nn^{\Omega_R}$.
	\item We define the restriction operator $R^{\nn,s,t}_{\Omega_L,\Omega_R}$ as follows:
	\[ R^{\nn,s,t}_{\Omega_L,\Omega_R}:\mathbb C^{N(\nn)s\times N(\nn)t}\to\mathbb C^{N_\nn^{\Omega_L}s\times N_\nn^{\Omega_R}t},\qquad R_{\Omega_L,\Omega_R}^{\nn,s,t}(A)=R^{\nn,s}_{\Omega_L}\,A\,(R^{\nn,t}_{\Omega_R})^T. \]
	In other words, if $A=[A_{\ii\jj}]_{\ii,\jj=\bu}^\nn\in\mathbb C^{N(\nn)s\times N(\nn)t}$ is a multilevel block matrix with blocks $A_{\ii\jj}\in\mathbb C^{s\times t}$ for every $\ii,\jj=\bu,\ldots,\nn$, then
	\begin{align*}
	R_{\Omega_L,\Omega_R}^{\nn,s,t}(A)&=R^{\nn,s}_{\Omega_L}\,A\,(R^{\nn,t}_{\Omega_R})^T=[(\ee_\ii^{(\nn)})^T\otimes I_s]_{\ii\in\II_\nn^{\Omega_L}}[A_{\ii\jj}]_{\ii,\jj=\bu}^\nn[\ee_\jj^{(\nn)}\otimes I_t]^{\jj\in\II_\nn^{\Omega_R}}=[A_{\ii\jj}]_{\ii\in\II_\nn^{\Omega_L}}^{\jj\in\II_\nn^{\Omega_R}}.
	\end{align*}
	In practice, the action of the operator $R_{\Omega_L,\Omega_R}^{\nn,s,t}$ on a multilevel block matrix $A$ of size $N(\nn)s\times N(\nn)t$ consists in removing all (block) rows of $A$ corresponding to multi-indices $\ii\not\in\II_\nn^{\Omega_L}$ and all (block) columns of $A$ corresponding to multi-indices $\jj\not\in\II_\nn^{\Omega_R}$.
	\item We define the expansion operator $E^{\nn,s,t}_{\Omega_L,\Omega_R}$ as follows:
	\[ E^{\nn,s,t}_{\Omega_L,\Omega_R}:\mathbb C^{N_\nn^{\Omega_L}s\times N_\nn^{\Omega_R}t}\to\mathbb C^{N(\nn)s\times N(\nn)t},\qquad E_{\Omega_L,\Omega_R}^{\nn,s,t}(B)=(R^{\nn,s}_{\Omega_L})^T\,B\,R^{\nn,t}_{\Omega_R}. \]
	In other words, if
	\[ B=[B_{\ii\jj}]_{\ii\in\II_\nn^{\Omega_L}}^{\jj\in\II_\nn^{\Omega_R}}\in\mathbb C^{N_\nn^{\Omega_L}s\times N_\nn^{\Omega_R}t} \]
	is a block matrix whose (block) entries $B_{\ii\jj}\in\mathbb C^{s\times t}$ are indexed by two multi-indices $\ii\in\II_\nn^{\Omega_L}$ and $\jj\in\II_\nn^{\Omega_R}$, then
	\begin{align*}
	E_{\Omega_L,\Omega_R}^{\nn,s,t}(B)&=(R^{\nn,s}_{\Omega_L})^T\,B\,R^{\nn,t}_{\Omega_R}=[\ee_\ii^{(\nn)}\otimes I_s]^{\ii\in\II_\nn^{\Omega_L}}[B_{\ii\jj}]_{\ii\in\II_\nn^{\Omega_L}}^{\jj\in\II_\nn^{\Omega_R}}[(\ee_\jj^{(\nn)})^T\otimes I_t]_{\jj\in\II_\nn^{\Omega_R}}\\
	&=\sum_{\substack{\ii\in\II_\nn^{\Omega_L}\\[1pt]\jj\in\II_\nn^{\Omega_R}}}(\ee_\ii^{(\nn)}\otimes I_s)B_{\ii\jj}((\ee_\jj^{(\nn)})^T\otimes I_t)=\sum_{\substack{\ii\in\II_\nn^{\Omega_L}\\[1pt]\jj\in\II_\nn^{\Omega_R}}}(\ee_\ii^{(\nn)}\otimes I_s)(1\otimes B_{\ii\jj})((\ee_\jj^{(\nn)})^T\otimes I_t)\\
	&=\sum_{\substack{\ii\in\II_\nn^{\Omega_L}\\[1pt]\jj\in\II_\nn^{\Omega_R}}}\underbrace{\ee_\ii^{(\nn)}(\ee_\jj^{(\nn)})^T}_{\substack{\textup{matrix with $1$}\\\textup{in position $(\ii,\jj)$}\\\textup{and $0$ elsewhere}}}{}\otimes B_{\ii\jj},
	\end{align*}
	which means that $E_{\Omega_L,\Omega_R}^{\nn,s,t}(B)=[\underbrace{(E_{\Omega_L,\Omega_R}^{\nn,s,t}(B))_{\ii\jj}}_{\substack{\textup{$s\times t$ block in}\\\textup{position $(\ii,\jj$)}}}]_{\ii,\jj=\bu}^\nn$, where, for every $\ii,\jj=\bu,\ldots,\nn$,
	\[ (E_{\Omega_L,\Omega_R}^{\nn,s,t}(B))_{\ii\jj}=\begin{cases}
	B_{\ii\jj}, &\mbox{if $\ii\in\II_\nn^{\Omega_L}$ and $\jj\in\II_\nn^{\Omega_R}$},\\
	O_{s,t}, &\mbox{if $\ii\not\in\II_\nn^{\Omega_L}$ or $\jj\not\in\II_\nn^{\Omega_R}$}.\end{cases} \]
	In practice, the action of the operator $E_{\Omega_L,\Omega_R}^{\nn,s,t}$ on a block matrix $B$ of size $N_\nn^{\Omega_L}s\times N_\nn^{\Omega_R}t$ consists in expanding $B$ to a $N(\nn)s\times N(\nn)t$ matrix by adding zero (block) rows in correspondence with multi-indices $\ii\not\in\II_\nn^{\Omega_L}$ and zero (block) columns in correspondence with multi-indices $\jj\not\in\II_\nn^{\Omega_R}$.
\end{itemize}
\end{definition}
 
Example~\ref{illustration} provides an illustration of Definition~\ref{ZRE-operators} in the 1-level scalar case $d=s=t=1$.

\begin{example}\label{illustration}

\begin{figure}
\centering
\includegraphics[width=0.6\textwidth]{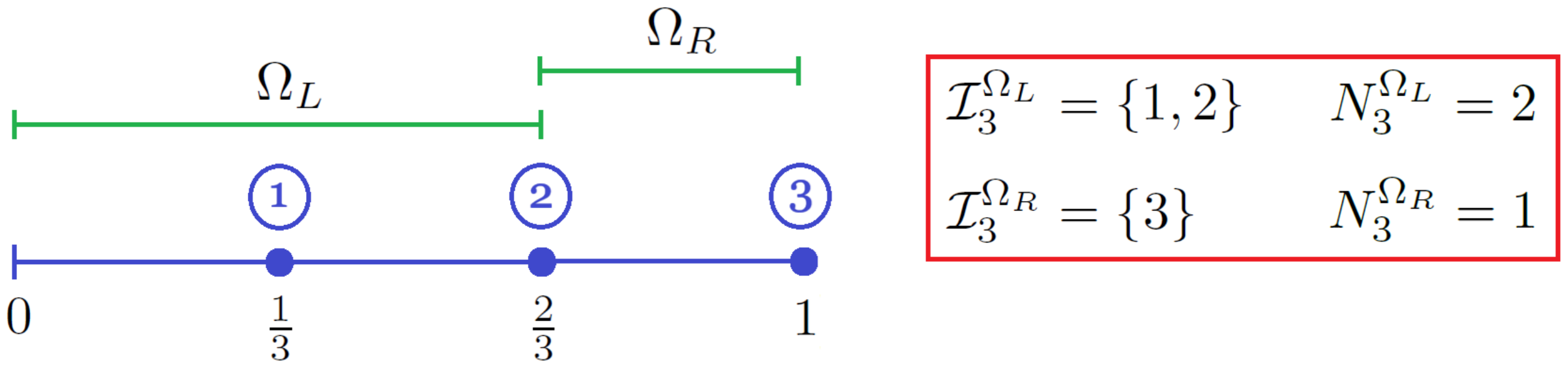}
\caption{Illustration for Example~\ref{illustration}.}
\label{DD-1d}
\end{figure}

Let $d=s=t=1$, let $n=3$, and let $\Omega_L=[0,\frac23]$, $\Omega_R=(\frac23,1]$. Note that $\II_3^{\Omega_L}=\{1,2\}$, $\II_3^{\Omega_R}=\{3\}$ and, consequently, $N_3^{\Omega_L}=2$, $N_3^{\Omega_R}=1$; see also Figure~\ref{DD-1d}.
Let
\[ A=\begin{bmatrix}a & b & c\\ d & e & f\\ g & h & i\end{bmatrix}\in\mathbb C^{3\times 3}. \]
Then, we have
\begin{alignat*}{5}
Z_{\Omega_L,\Omega_R}^{3,1,1}(A)&=\begin{bmatrix}0 & 0 & c\\ 0 & 0 & f\\ 0 & 0 & 0\end{bmatrix},&\qquad R_{\Omega_L,\Omega_R}^{3,1,1}(A)&=\begin{bmatrix}c\\ f\end{bmatrix},&\qquad E_{\Omega_L,\Omega_R}^{3,1,1}\biggl(\begin{bmatrix}c\\ f\end{bmatrix}\biggr)&=\begin{bmatrix}0 & 0 & c\\ 0 & 0 & f\\ 0 & 0 & 0\end{bmatrix}=Z_{\Omega_L,\Omega_R}^{3,1,1}(A).
\end{alignat*}
Note in particular that $E_{\Omega_L,\Omega_R}^{3,1,1}(R_{\Omega_L,\Omega_R}^{3,1,1}(A))=Z_{\Omega_L,\Omega_R}^{3,1,1}(A)$.
\end{example}

A few basic equations involving zeroing, restriction and expansion operators are collected in Lemma~\ref{zre-eqs}.

\begin{lemma}\label{zre-eqs}
Let $d,s,t$ be positive integers, let $\nn\in\mathbb N^d$, and let $\Omega_L,\Omega_R\subseteq[0,1]^d$. Then,
\begin{alignat}{3}
R_{\Omega_L,\Omega_R}^{\nn,s,t}(Z_{\Omega_L,\Omega_R}^{\nn,s,t}(A))&=R^{\nn,s,t}_{\Omega_L,\Omega_R}(A),&\qquad A&\in\mathbb C^{N(\nn)s\times N(\nn)t},\label{RZ=R}\\
Z_{\Omega_L,\Omega_R}^{\nn,s,t}(E^{\nn,s,t}_{\Omega_L,\Omega_R}(B))&=E^{\nn,s,t}_{\Omega_L,\Omega_R}(B),&\qquad B&\in\mathbb C^{N_\nn^{\Omega_L}s\times N_\nn^{\Omega_R}t},\label{ZE=E}\\
R_{\Omega_L,\Omega_R}^{\nn,s,t}(E^{\nn,s,t}_{\Omega_L,\Omega_R}(B))&=B,&\qquad B&\in\mathbb C^{N_\nn^{\Omega_L}s\times N_\nn^{\Omega_R}t},\label{RE=I}\\
E^{\nn,s,t}_{\Omega_L,\Omega_R}(R_{\Omega_L,\Omega_R}^{\nn,s,t}(A))&=Z_{\Omega_L,\Omega_R}^{\nn,s,t}(A),&\qquad A&\in\mathbb C^{N(\nn)s\times N(\nn)t}.\label{ER=Z}
\end{alignat}
\end{lemma}
\begin{proof}
Equations~\eqref{RZ=R}--\eqref{ER=Z} are intuitively clear: it is enough to think about the actions of zeroing, restriction and expansion operators on an arbitrary matrix.
Alternatively, we can easily prove \eqref{RZ=R}--\eqref{ER=Z} through Lemma~\ref{r-lemma}: 
\begin{align*}
R_{\Omega_L,\Omega_R}^{\nn,s,t}(Z_{\Omega_L,\Omega_R}^{\nn,s,t}(A))&=R^{\nn,s}_{\Omega_L}D_\nn(\chi_{\Omega_L}I_s)\,A\,D_\nn(\chi_{\Omega_R}I_t)(R^{\nn,t}_{\Omega_R})^T=R^{\nn,s}_{\Omega_L}\,A\,(R^{\nn,t}_{\Omega_R})^T=R^{\nn,s,t}_{\Omega_L,\Omega_R}(A),\\
Z_{\Omega_L,\Omega_R}^{\nn,s,t}(E^{\nn,s,t}_{\Omega_L,\Omega_R}(B))&=D_\nn(\chi_{\Omega_L}I_s)(R^{\nn,s}_{\Omega_L})^T\,B\,R^{\nn,t}_{\Omega_R}D_\nn(\chi_{\Omega_R}I_t)=(R^{\nn,s}_{\Omega_L})^T\,B\,R^{\nn,t}_{\Omega_R}=E^{\nn,s,t}_{\Omega_L,\Omega_R}(B),\\
R_{\Omega_L,\Omega_R}^{\nn,s,t}(E^{\nn,s,t}_{\Omega_L,\Omega_R}(B))&=R^{\nn,s}_{\Omega_L}(R^{\nn,s}_{\Omega_L})^T\,B\,R^{\nn,t}_{\Omega_R}(R^{\nn,t}_{\Omega_R})^T=I_{N_\nn^{\Omega_L}s}\,B\,I_{N_\nn^{\Omega_R}t}=B,\\
E^{\nn,s,t}_{\Omega_L,\Omega_R}(R_{\Omega_L,\Omega_R}^{\nn,s,t}(A))&=(R^{\nn,s}_{\Omega_L})^TR^{\nn,s}_{\Omega_L}\,A\,(R^{\nn,t}_{\Omega_R})^TR^{\nn,t}_{\Omega_R}=D_\nn(\chi_{\Omega_L}I_s)\,A\,D_\nn(\chi_{\Omega_R}I_t)=Z_{\Omega_L,\Omega_R}^{\nn,s,t}(A). \qedhere
\end{align*}
\end{proof}

With the definitions of zeroing, restriction and expansion operators at hand, we can now define the BJ/BGS preconditioners for multilevel block matrices. These definitions are based on the notions of blockdiag and blocktril operators, which are introduced here as generalizations of the analogous notions presented in \cite{GLH}.

\begin{definition}[\textbf{block Jacobi/Gauss--Seidel preconditioners}]\label{bJGSp}
Let $d,s,t,\nu$ be positive integers, let $\nn\in\mathbb N^d$, and let $\Omega_1,\ldots,\Omega_\nu\subseteq[0,1]^d$.
\begin{itemize}[nolistsep,leftmargin=*]
	\item We define the operator $\mathop{\rm blockdiag}_{\Omega_1,\ldots,\Omega_\nu}^{\nn,s,t}$ as follows:
	\begin{equation}\label{blockdiag-op}
	\mathop{\rm blockdiag}^{\nn,s,t}_{\Omega_1,\ldots,\Omega_\nu}:\mathbb C^{N(\nn)s\times N(\nn)t}\to\mathbb C^{N(\nn)s\times N(\nn)t},\qquad\mathop{\rm blockdiag}_{\Omega_1,\ldots,\Omega_\nu}^{\nn,s,t}(A)=\sum_{i=1}^\nu Z_{\Omega_i,\Omega_i}^{\nn,s,t}(A).
	\end{equation}
	If $A\in\mathbb C^{N(\nn)s\times N(\nn)t}$ and $\{\Omega_1,\ldots,\Omega_\nu\}$ is a partition of $[0,1]^d$, then
	\begin{equation}\label{bJ-pop}
	P_{\Omega_1,\ldots,\Omega_\nu}^{BJ,\nn,s,t}(A)=\mathop{\rm blockdiag}_{\Omega_1,\ldots,\Omega_\nu}^{\nn,s,t}(A)
	\end{equation}
	is referred to as the block Jacobi (BJ) preconditioner of $A$ generated by the partition $\{\Omega_1,\ldots,\Omega_\nu\}$. 
	\item We define the operator $\mathop{\rm blocktril}_{\Omega_1,\ldots,\Omega_\nu}^{\nn,s,t}$ as follows:
	\begin{equation}\label{blocktril-op}
	\mathop{\rm blocktril}^{\nn,s,t}_{\Omega_1,\ldots,\Omega_\nu}:\mathbb C^{N(\nn)s\times N(\nn)t}\to\mathbb C^{N(\nn)s\times N(\nn)t},\qquad\mathop{\rm blocktril}_{\Omega_1,\ldots,\Omega_\nu}^{\nn,s,t}(A)=\sum_{1\le j\le i\le\nu}Z_{\Omega_i,\Omega_j}^{\nn,s,t}(A).
	\end{equation}
	If $A\in\mathbb C^{N(\nn)s\times N(\nn)t}$ and $\{\Omega_1,\ldots,\Omega_\nu\}$ is a partition of $[0,1]^d$, then
	\begin{equation}\label{bgs-pop}
	P_{\Omega_1,\ldots,\Omega_\nu}^{BGS,\nn,s,t}(A)=\mathop{\rm blocktril}_{\Omega_1,\ldots,\Omega_\nu}^{\nn,s,t}(A)
	\end{equation}
	is referred to as block Gauss--Seidel (BGS) preconditioner of $A$ generated by the partition $\{\Omega_1,\ldots,\Omega_\nu\}$. 
\end{itemize}
\end{definition}

Example~\ref{illustration'} provides an illustration of Definition~\ref{bJGSp} in the $1$-level scalar case $d=s=t=1$ with $\nu=2$.

\begin{example}\label{illustration'}
Let $d=s=t=1$, let $\nu=2$, let $n=3$, and let $\Omega_1=[0,\frac23]$, $\Omega_2=(\frac23,1]$. Note that $\Omega_1,\Omega_2$ are just the sets $\Omega_L,\Omega_R$ of Example~\ref{illustration} and $\{\Omega_1,\Omega_2\}$ is a partition of $[0,1]$.
We have $\II_3^{\Omega_1}=\{1,2\}$, $\II_3^{\Omega_2}=\{3\}$ and $N_3^{\Omega_1}=2$, $N_3^{\Omega_2}=1$; see also Figure~\ref{DD-1d}, keeping in mind that $\Omega_1=\Omega_L$ and $\Omega_2=\Omega_R$.
Let
\[ A=\begin{bmatrix}a & b & c\\ d & e & f\\ g & h & i\end{bmatrix}\in\mathbb C^{3\times 3}. \]
Then, we have
{\allowdisplaybreaks\begin{alignat*}{5}
Z_{\Omega_1,\Omega_1}^{3,1,1}(A)&=\begin{bmatrix}a & b & 0\\ d & e & 0\\ 0 & 0 & 0\end{bmatrix},&\qquad R_{\Omega_1,\Omega_1}^{3,1,1}(A)&=\begin{bmatrix}a & b\\ d & e\end{bmatrix},&\qquad E_{\Omega_1,\Omega_1}^{3,1,1}\biggl(\begin{bmatrix}a & b\\ d & e\end{bmatrix}\biggr)&=\begin{bmatrix}a & b & 0\\ d & e & 0\\ 0 & 0 & 0\end{bmatrix}=Z_{\Omega_1,\Omega_1}^{3,1,1}(A),\\*
Z_{\Omega_1,\Omega_2}^{3,1,1}(A)&=\begin{bmatrix}0 & 0 & c\\ 0 & 0 & f\\ 0 & 0 & 0\end{bmatrix},&\qquad R_{\Omega_1,\Omega_2}^{3,1,1}(A)&=\begin{bmatrix}c\\ f\end{bmatrix},&\qquad E_{\Omega_1,\Omega_2}^{3,1,1}\biggl(\begin{bmatrix}c\\ f\end{bmatrix}\biggr)&=\begin{bmatrix}0 & 0 & c\\ 0 & 0 & f\\ 0 & 0 & 0\end{bmatrix}=Z_{\Omega_1,\Omega_2}^{3,1,1}(A),\\
Z_{\Omega_2,\Omega_1}^{3,1,1}(A)&=\begin{bmatrix}0 & 0 & 0\\ 0 & 0 & 0\\ g & h & 0\end{bmatrix},&\qquad R_{\Omega_2,\Omega_1}^{3,1,1}(A)&=\begin{bmatrix}g & h\end{bmatrix},&\qquad E_{\Omega_2,\Omega_1}^{3,1,1}\bigl(\begin{bmatrix}g & h\end{bmatrix}\bigr)&=\begin{bmatrix}0 & 0 & 0\\ 0 & 0 & 0\\ g & h & 0\end{bmatrix}=Z_{\Omega_2,\Omega_1}^{3,1,1}(A),\\*
Z_{\Omega_2,\Omega_2}^{3,1,1}(A)&=\begin{bmatrix}0 & 0 & 0\\ 0 & 0 & 0\\ 0 & 0 & i\end{bmatrix},&\qquad R_{\Omega_2,\Omega_2}^{3,1,1}(A)&=\begin{bmatrix}i\end{bmatrix},&\qquad E_{\Omega_2,\Omega_2}^{3,1,1}\bigl(\begin{bmatrix}i\end{bmatrix}\bigr)&=\begin{bmatrix}0 & 0 & 0\\ 0 & 0 & 0\\ 0 & 0 & i\end{bmatrix}=Z_{\Omega_2,\Omega_2}^{3,1,1}(A),
\end{alignat*}}%
and
\begin{align*}
P_{\Omega_1,\Omega_2}^{BJ,3,1,1}(A)&=\mathop{\rm blockdiag}_{\Omega_1,\Omega_2}^{3,1,1}(A)=\sum_{i=1}^2Z_{\Omega_1,\Omega_2}^{3,1,1}(A)=\begin{bmatrix}a & b & 0\\ d & e & 0\\ 0 & 0 & i\end{bmatrix}=\begin{bmatrix}R_{\Omega_1,\Omega_1}^{3,1,1}(A) & \\[7.5pt] & R_{\Omega_2,\Omega_2}^{3,1,1}(A)\end{bmatrix},\\
P_{\Omega_1,\Omega_2}^{BGS,3,1,1}(A)&=\sum_{1\le j\le i\le 2}Z_{\Omega_i,\Omega_j}^{3,1,1}(A)=\begin{bmatrix}a & b & 0\\ d & e & 0\\ g & h & i\end{bmatrix}=\begin{bmatrix}R_{\Omega_1,\Omega_1}^{3,1,1}(A) & \\[7.5pt] R_{\Omega_2,\Omega_1}^{3,1,1}(A) & R_{\Omega_2,\Omega_2}^{3,1,1}(A)\end{bmatrix}.
\end{align*}
Note that
\[ A=\sum_{i,j=1}^2Z_{\Omega_i,\Omega_j}^{3,1,1}(A)=\begin{bmatrix}R_{\Omega_1,\Omega_1}^{3,1,1}(A) & R_{\Omega_1,\Omega_2}^{3,1,1}(A)\\[7.5pt] R_{\Omega_2,\Omega_1}^{3,1,1}(A) & R_{\Omega_2,\Omega_2}^{3,1,1}(A)\end{bmatrix}. \]
\end{example}

In Example~\ref{recupero}, we show that the BJ/BGS preconditioners considered in \cite{GLH} are special instances of those defined in Definition~\ref{bJGSp}.
Throughout this paper, a partition of a positive integer $n$ is a vector of positive integers $(n_1,\ldots,n_\nu)$ such that $n_1+\ldots+n_\nu=n$.

\begin{example}\label{recupero}
Let $d=1$, let $s,t,\nu,n$ be positive integers, let $(n_1,\ldots,n_\nu)$ be a partition of $n$, and let $\{\Omega_1,\ldots,\Omega_\nu\}$ be the partition of $[0,1]$ associated with $(n_1,\ldots,n_\nu)$, which is defined as follows:
\begin{align*}
\Omega_1&=\Bigl[0,\frac{n_1}{n}\Bigr],\\
\Omega_2&=\Bigl(\frac{n_1}{n},\frac{n_1+n_2}{n}\Bigr],\\
\Omega_3&=\Bigl(\frac{n_1+n_2}n,\frac{n_1+n_2+n_3}n\Bigr],\\
\vdots & \\
\Omega_\nu&=\Bigl(\frac{n_1+\ldots+n_{\nu-1}}n,1\Bigr].
\end{align*}
By construction, for every $k=1,\ldots,\nu$, the set of indices $i\in\{1,\ldots,n\}$ associated with $\Omega_k$ is given by
\[ \II_n^{\Omega_k}=\{n_1+\ldots+n_{k-1}+1,\ldots,n_1+\ldots+n_k\}. \]
If $A=[A_{ij}]_{i,j=1}^n\in\mathbb C^{ns\times nt}$ is a block matrix with blocks $A_{ij}\in\mathbb C^{s\times t}$ for every $i,j=1,\ldots,n$, then, by direct computation,
\begin{align}
P_{n_1,\ldots,n_\nu}^{BJ,n,s,t}(A)&:=P_{\Omega_1,\ldots,\Omega_\nu}^{BJ,n,s,t}(A)=\mathop{\rm blockdiag}_{\Omega_1,\ldots,\Omega_\nu}^{n,s,t}(A)=\sum_{i=1}^\nu Z_{\Omega_i,\Omega_i}^{n,s,t}(A)=\begin{bmatrix}A^{(11)} & & \\ & \ddots & \\ & & A^{(\nu\nu)}\end{bmatrix},\label{block-J}\\
P_{n_1,\ldots,n_\nu}^{BGS,n,s,t}(A)&:=P_{\Omega_1,\ldots,\Omega_\nu}^{BGS,n,s,t}(A)=\mathop{\rm blocktril}_{\Omega_1,\ldots,\Omega_\nu}^{n,s,t}(A)=\sum_{1\le j\le i\le \nu}Z_{\Omega_i,\Omega_j}^{n,s,t}(A)=\begin{bmatrix}A^{(11)} & & \\ \vdots & \ddots & \\[3pt] A^{(\nu1)} & \cdots & A^{(\nu\nu)}\end{bmatrix},\label{block-GS}
\end{align}
where $A^{(pq)}$ is the block of $A$ of size $n_ps\times n_qt$ in position $(p,q)$, i.e., $A^{(pq)}=[A_{ij}]_{i=n_1+\ldots+n_{p-1}+1,\ldots,n_1+\ldots+n_p}^{j=n_1+\ldots+n_{q-1}+1,\ldots,n_1+\ldots+n_q}$ $(1\le p,q\le\nu)$.
The matrices \eqref{block-J}--\eqref{block-GS} are just the BJ/BGS preconditioners for $A$ generated by the partition $(n_1,\ldots,n_\nu)$, as defined in \cite[Definition~4.1]{GLH}.
In the case where $s=t=1$, the matrices \eqref{block-J}--\eqref{block-GS} coincide with the classical BJ/BGS preconditioners for $A$ generated by the partition $(n_1,\ldots,n_\nu)$, as defined in standard numerical analysis courses.
\end{example}

\begin{remark}\label{A-BJ-BGS}
Let $d,s,t,\nu$ be positive integers, let $\nn\in\mathbb N^d$, and let $\{\Omega_1,\ldots,\Omega_\nu\}$ be a partition of $[0,1]^d$. By Definitions~\ref{ZRE-operators}--\ref{bJGSp}, for every $A\in\mathbb C^{N(\nn)s\times N(\nn)t}$, we have
\[ A=\sum_{i,j=1}^\nu Z_{\Omega_i,\Omega_j}^{\nn,s,t}(A),\qquad P_{\Omega_1,\ldots,\Omega_\nu}^{BJ,\nn,s,t}(A)=\sum_{i=1}^\nu Z_{\Omega_i,\Omega_i}^{\nn,s,t}(A),\qquad P_{\Omega_1,\ldots,\Omega_\nu}^{BGS,\nn,s,t}(A)=\sum_{1\le j\le i\le\nu}Z_{\Omega_i,\Omega_j}^{\nn,s,t}(A). \]
We can therefore ``see'' the connection between $A$ and the BJ/BGS preconditioners $P^{BJ,\nn,s,t}_{\Omega_1,\ldots,\Omega_\nu}(A)$ and $P^{BGS,\nn,s,t}_{\Omega_1,\ldots,\Omega_\nu}(A)$.
\end{remark}

\begin{remark}
Let $d,s,t,\nu$ be positive integers, let $\nn\in\mathbb N^d$, and let $\{\Omega_1,\ldots,\Omega_\nu\}$ be a partition of $[0,1]^d$. Then, there exists a permutation matrix $\Pi^\nn_{\Omega_1,\ldots,\Omega_\nu}$ of size $N(\nn)$, depending only on $\nn$ and $\Omega_1,\ldots,\Omega_\nu$, such that, for every $A\in\mathbb C^{N(\nn)s\times N(\nn)t}$,
\begin{align*}
(\Pi_{\Omega_1,\ldots,\Omega_\nu}^\nn\otimes I_s)P_{\Omega_1,\ldots,\Omega_\nu}^{BJ,\nn,s,t}(A)(\Pi_{\Omega_1,\ldots,\Omega_\nu}^\nn\otimes I_t)^T&=\begin{bmatrix}
R_{\Omega_1,\Omega_1}^{\nn,s,t}(A) & & & \\[10pt]
& R_{\Omega_2,\Omega_2}^{\nn,s,t}(A) & & \\[10pt]
& & \ddots & \\[10pt]
& & & R_{\Omega_\nu,\Omega_\nu}^{\nn,s,t}(A)
\end{bmatrix},\\
(\Pi_{\Omega_1,\ldots,\Omega_\nu}^\nn\otimes I_s)P_{\Omega_1,\ldots,\Omega_\nu}^{BGS,\nn,s,t}(A)(\Pi_{\Omega_1,\ldots,\Omega_\nu}^\nn\otimes I_t)^T&=\begin{bmatrix}
R_{\Omega_1,\Omega_1}^{\nn,s,t}(A) & & & \\[15pt]
R_{\Omega_2,\Omega_1}^{\nn,s,t}(A) & R_{\Omega_2,\Omega_2}^{\nn,s,t}(A) & & \\[15pt]
\vdots & \ddots & \ddots & \\[15pt]
R_{\Omega_\nu,\Omega_1}^{\nn,s,t}(A) & \cdots & R_{\Omega_\nu,\Omega_{\nu-1}}^{\nn,s,t}(A) & R_{\Omega_\nu,\Omega_\nu}^{\nn,s,t}(A)
\end{bmatrix}.
\end{align*}
The permutation matrix $\Pi_{\Omega_1,\ldots,\Omega_\nu}^\nn$ is defined in the following way from the vectors of the canonical basis of $\mathbb C^{N(\nn)}$:
\begin{itemize}[nolistsep,leftmargin=*]
	\item The 1st $N_\nn^{\Omega_1}$ rows of $\Pi^\nn_{\Omega_1,\ldots,\Omega_\nu}$ are the vectors $(\ee_\ii^{(\nn)})^T$ corresponding to the $d$-indices $\ii\in\II_\nn^{\Omega_1}$,
	\item The 2nd $N_\nn^{\Omega_2}$ rows of $\Pi^\nn_{\Omega_1,\ldots,\Omega_\nu}$ are the vectors $(\ee_\ii^{(\nn)})^T$ corresponding to the $d$-indices $\ii\in\II_\nn^{\Omega_2}$,
	
	$\vdots$ $\vphantom{N_\nn^{\Omega_2}\Pi^\nn_{\Omega_1,\ldots,\Omega_\nu}(\ee_\ii^{(\nn)})^T\ii\in\II_\nn^{\Omega_2}}$
	\item The $\nu$th $N_\nn^{\Omega_\nu}$ rows of $\Pi^\nn_{\Omega_1,\ldots,\Omega_\nu}$ are the vectors $(\ee_\ii^{(\nn)})^T$ corresponding to the $d$-indices $\ii\in\II_\nn^{\Omega_\nu}$.
\end{itemize}
In formulas,
\[ \Pi^\nn_{\Omega_1,\ldots,\Omega_\nu}=\left[\begin{array}{c}[(\ee_\ii^{(\nn)})^T]_{\ii\in\II_{\nn\vphantom{\int}}^{\Omega_1}}\\ \hline [(\ee_\ii^{(\nn)\vphantom{\int^1}})^T]_{\ii\in\II_{\nn\vphantom{\int}}^{\Omega_2}}\\ 
\hline \vdots\vphantom{\Big|} \\ \hline [(\ee_\ii^{(\nn)\vphantom{\int^1}})^T]_{\ii\in\II_\nn^{\Omega_\nu}}\end{array}\right]. \]
\end{remark}

Before moving on to the definition of the AS/MS preconditioners, we introduce the restriction-inversion-expansion (RIE) operator for multilevel block matrices.

\begin{definition}[\textbf{restriction-inversion-expansion operator}]\label{rie}
Let $d,s$ be positive integers, let $\nn\in\mathbb N^d$, and let $\Omega\subseteq[0,1]^d$.
We define the restriction-inversion-expansion (RIE) operator $\rie_\Omega^{\nn,s}$ as follows:
\begin{equation}\label{rie-op}
\rie_\Omega^{\nn,s}:D_\Omega^{\rie,\nn,s}\subseteq\mathbb C^{N(\nn)s\times N(\nn)s}\to\mathbb C^{N(\nn)s\times N(\nn)s},\qquad\rie_\Omega^{\nn,s}(A)=E_{\Omega,\Omega}^{\nn,s,s}({R_{\Omega,\Omega}^{\nn,s,s}(A)}^{-1}),
\end{equation}
where $D^{\rie,\nn,s}_\Omega$ is the domain of $\rie_\Omega^{\nn,s}$ given by
\[ D^{\rie,\nn,s}_\Omega=\bigl\{A\in\mathbb C^{N(\nn)s\times N(\nn)s}:\,R_{\Omega,\Omega}^{\nn,s,s}(A)\textup{ is invertible}\bigr\}. \]
\end{definition}

Using the RIE operator, we now define the AS/MS preconditioners for multilevel block matrices.
In what follows, a cover of $[0,1]^d$ is a family of sets $\{\Omega_1,\ldots,\Omega_\nu\}$ such that $\Omega_1,\ldots,\Omega_\nu\subseteq[0,1]^d$ and $\Omega_1\cup\cdots\cup\Omega_\nu=[0,1]^d$.

\begin{definition}[\textbf{additive/multiplicative Schwarz preconditioners}]\label{amSp}
Let $d,s,\nu$ be positive integers, let $\nn\in\mathbb N^d$, and let $\{\Omega_1,\ldots,\Omega_\nu\}$ be a cover of $[0,1]^d$.
\begin{itemize}[nolistsep,leftmargin=*]
	\item We define the additive Schwarz (AS) operator as follows:
	\begin{equation}\label{as-op}
	\begin{aligned}
	&P^{AS,\nn,s}_{\Omega_1,\ldots,\Omega_\nu}:D^{AS,\nn,s}_{\Omega_1,\ldots,\Omega_\nu}\subseteq\mathbb C^{N(\nn)s\times N(\nn)s}\to\mathbb C^{N(\nn)s\times N(\nn)s},\\
	&P_{\Omega_1,\ldots,\Omega_\nu}^{AS,\nn,s}(A)=\Biggl[\sum_{i=1}^\nu\rie_{\Omega_i}^{\nn,s}(A)\Biggr]^{-1},
	\end{aligned}
	\end{equation}
	where $D^{AS,\nn,s}_{\Omega_1,\ldots,\Omega_\nu}$ is the domain of $P^{AS,\nn,s}_{\Omega_1,\ldots,\Omega_\nu}$ given by
	\begin{equation*}
	\begin{aligned}
	D^{AS,\nn,s}_{\Omega_1,\ldots,\Omega_\nu}&=\Biggl\{A\in D_{\Omega_1}^{\rie,\nn,s}\cap\cdots\cap D_{\Omega_\nu}^{\rie,\nn,s}:\,\sum_{i=1}^\nu\rie_{\Omega_i}^{\nn,s}(A)\textup{ is invertible}\Biggr\}\\
	&=\Biggl\{A\in\mathbb C^{N(\nn)s\times N(\nn)s}:\,R_{\Omega_1,\Omega_1}^{\nn,s,s}(A),\ldots,R_{\Omega_\nu,\Omega_\nu}^{\nn,s,s}(A),\,\sum_{i=1}^\nu\rie_{\Omega_i}^{\nn,s}(A)\textup{ are invertible}\Biggr\}.
	\end{aligned}
	\end{equation*}
	If $A\in D^{AS,\nn,s}_{\Omega_1,\ldots,\Omega_\nu}$, then $P_{\Omega_1,\ldots,\Omega_\nu}^{AS,\nn,s}(A)$ is referred to as the additive Schwarz (AS) preconditioner of $A$ generated by the cover $\{\Omega_1,\ldots,\Omega_\nu\}^{\vphantom{1}}$.
	\item We define the multiplicative Schwarz (MS) operator as follows:
	\begin{equation}\label{ms-op}
	\begin{aligned}
	&P^{MS,\nn,s}_{\Omega_1,\ldots,\Omega_\nu}:D^{MS,\nn,s}_{\Omega_1,\ldots,\Omega_\nu}\subseteq\mathbb C^{N(\nn)s\times N(\nn)s}\to\mathbb C^{N(\nn)s\times N(\nn)s},\\
	&P_{\Omega_1,\ldots,\Omega_\nu}^{MS,\nn,s}(A)=A\Biggl[I-\prod_{i=\nu}^1\Bigl(I-\rie_{\Omega_i}^{\nn,s}(A)A\Bigr)\Biggr]^{-1},
	\end{aligned}
	\end{equation}
	where $D^{MS,\nn,s}_{\Omega_1,\ldots,\Omega_\nu}$ is the domain of $P^{MS,\nn,s}_{\Omega_1,\ldots,\Omega_\nu}$ given by
	\begin{equation*}
	\begin{aligned}
	D^{MS,\nn,s}_{\Omega_1,\ldots,\Omega_\nu}&=\Biggl\{A\in D_{\Omega_1}^{\rie,\nn,s}\cap\cdots\cap D_{\Omega_\nu}^{\rie,\nn,s}:\,I-\prod_{i=\nu}^1\Bigl(I-\rie_{\Omega_i}^{\nn,s}(A)A\Bigr)\textup{ is invertible}\Biggr\}\\
	&=\Biggl\{A\in\mathbb C^{N(\nn)s\times N(\nn)s}:\,R_{\Omega_1,\Omega_1}^{\nn,s,s}(A),\ldots,R_{\Omega_\nu,\Omega_\nu}^{\nn,s,s}(A),\,I-\prod_{i=\nu}^1\Bigl(I-\rie_{\Omega_i}^{\nn,s}(A)A\Bigr)\textup{ are invertible}\Biggr\}.
	\end{aligned}
	\end{equation*}
	If $A\in D^{MS,\nn,s}_{\Omega_1,\ldots,\Omega_\nu}$, then $P_{\Omega_1,\ldots,\Omega_\nu}^{MS,\nn,s}(A)$ is referred to as the multiplicative Schwarz (MS) preconditioner of $A$ generated by the cover $\{\Omega_1,\ldots,\Omega_\nu\}^{\vphantom{1}}$.
\end{itemize}
\end{definition}

Example~\ref{illustration''} provides an illustration of Definition~\ref{amSp} in the $1$-level scalar case $d=s=1$ with $\nu=2$.

\begin{example}\label{illustration''}
Let $d=s=1$, let $\nu=2$, let $n=3$, and let $\Omega_1=[0,\frac23]$, $\Omega_2=(\frac23,1]$. Note that $\Omega_1,\Omega_2$ are the same as in Example~\ref{illustration'} and $\{\Omega_1,\Omega_2\}$ is a partition of $[0,1]$.
We have $\II_3^{\Omega_1}=\{1,2\}$, $\II_3^{\Omega_2}=\{3\}$ and $N_3^{\Omega_1}=2$, $N_3^{\Omega_2}=1$.
Let
\[ A=\begin{bmatrix}a & b & c\\ d & e & f\\ g & h & i\end{bmatrix}\in\mathbb C^{3\times 3}. \]
To define the AS/MS preconditioners for $A$, it is necessary that $A\in D_{\Omega_1}^{\rie,3,1}\cap D_{\Omega_2}^{\rie,3,1}$, i.e., the matrices
\begin{equation}\label{Rcommon}
R_{\Omega_1,\Omega_1}^{3,1,1}(A)=\begin{bmatrix}a & b\\ d & e\end{bmatrix},\qquad R_{\Omega_2,\Omega_2}^{3,1,1}(A)=\begin{bmatrix}i\end{bmatrix}
\end{equation}
must be invertible. This allows us to compute the ``RIE matrices''
\begin{align*}
\rie_{\Omega_1}^{3,1}(A)&=E_{\Omega_1,\Omega_1}^{3,1,1}({R_{\Omega_1,\Omega_1}^{3,1,1}(A)}^{-1})=\begin{bmatrix}{R_{\Omega_1,\Omega_1}^{3,1,1}(A)}^{-1} & \begin{array}{c}0\\0\end{array}\\ \begin{array}{cc}0 &\quad 0 \end{array} & 0\end{bmatrix}=\begin{bmatrix}\hspace{-7.5pt}\ \ \;\begin{bmatrix}a & b\\ d & e\end{bmatrix}^{-1} & \hspace{-10pt}\begin{array}{c}0\\0\end{array}\\[10pt] \hspace{-7.5pt}\begin{array}{cc}0 & 0 \end{array} & \hspace{-10pt}0\end{bmatrix},\\
\rie_{\Omega_2}^{3,1}(A)&=E_{\Omega_2,\Omega_2}^{3,1,1}({R_{\Omega_2,\Omega_2}^{3,1,1}(A)}^{-1})=\begin{bmatrix}\,\,0 & 0 & 0\\ \,\,0 & 0 & 0\\ \,\,0 & 0 & {R_{\Omega_2,\Omega_2}^{3,1,1}(A)}^{-1}\end{bmatrix}=\begin{bmatrix}\,\,0 & 0 & \hspace{-12pt}0\\ \,\,0 & 0 & \hspace{-12pt}0\\ \,\,0 & 0 & \begin{bmatrix}i\end{bmatrix}^{-1}\end{bmatrix}.
\end{align*}
Suppose now that $A\in D^{AS,3,1}_{\Omega_1,\Omega_2}$. This means that the matrices \eqref{Rcommon} are invertible as well as the matrix
\[ \sum_{i=1}^2\rie_{\Omega_i}^{3,1}(A)=\begin{bmatrix}{R_{\Omega_1,\Omega_1}^{3,1,1}(A)}^{-1} & \\ & {R_{\Omega_2,\Omega_2}^{3,1,1}(A)}^{-1}\end{bmatrix}. \]
Then, we can compute the AS preconditioner for $A$:
\[ P^{AS,3,1}_{\Omega_1,\Omega_2}(A)=\Biggl[\sum_{i=1}^2\rie_{\Omega_i}^{3,1}(A)\Biggr]^{-1}=\begin{bmatrix}R_{\Omega_1,\Omega_1}^{3,1,1}(A) & \\[7.5pt] & R_{\Omega_2,\Omega_2}^{3,1,1}(A)\end{bmatrix}=\begin{bmatrix}a & b & 0\\ d & e & 0\\ 0 & 0 & i\end{bmatrix}. \]
Suppose now that $A\in D^{MS,3,1}_{\Omega_1,\Omega_2}$. This means that the matrices \eqref{Rcommon} are invertible as well as the matrix
\begin{align*}
I-\prod_{i=2}^1\Bigl(I-\rie_{\Omega_i}^{3,1}(A)A\Bigr)&=I-\prod_{i=2}^1\left(I-\rie_{\Omega_i}^{3,1}(A)\begin{bmatrix}R_{\Omega_1,\Omega_1}^{3,1,1}(A) & R_{\Omega_1,\Omega_2}^{3,1,1}(A)\\[7.5pt] R_{\Omega_2,\Omega_1}^{3,1,1}(A) & R_{\Omega_2,\Omega_2}^{3,1,1}(A)\end{bmatrix}\right)\\
&=I-\begin{bmatrix}\begin{bmatrix}1 & 0\\ 0 & 1\end{bmatrix} & \begin{array}{c}0\\0\end{array}\\[10pt] -{R_{\Omega_2,\Omega_2}^{3,1,1}(A)}^{-1}R_{\Omega_2,\Omega_1}^{3,1,1}(A) & 0\end{bmatrix}\begin{bmatrix}\begin{bmatrix}0 & 0\\ 0 & 0\end{bmatrix} & -{R_{\Omega_1,\Omega_1}^{3,1,1}(A)}^{-1}R_{\Omega_1,\Omega_2}^{3,1,1}(A)\\[10pt] \begin{array}{cc}0 & 0\end{array} & 1\end{bmatrix}\\
&=\begin{bmatrix}\begin{bmatrix}1 & 0\\ 0 & 1\end{bmatrix} & {R_{\Omega_1,\Omega_1}^{3,1,1}(A)}^{-1}R_{\Omega_1,\Omega_2}^{3,1,1}(A)\\[10pt] \begin{array}{cc}0 & 0\end{array} & 1-{R_{\Omega_2,\Omega_2}^{3,1,1}(A)}^{-1}R_{\Omega_2,\Omega_1}^{3,1,1}(A){R_{\Omega_1,\Omega_1}^{3,1,1}(A)}^{-1}R_{\Omega_1,\Omega_2}^{3,1,1}(A)\end{bmatrix}.
\end{align*}
Then, we can compute the MS preconditioner for $A$:\,\footnote{\,The second equality holds because, by standard block matrix multiplication, one can show that
\[  \begin{bmatrix}R_{\Omega_1,\Omega_1}^{3,1,1}(A) & \\[7.5pt] R_{\Omega_2,\Omega_1}^{3,1,1}(A) & R_{\Omega_2,\Omega_2}^{3,1,1}(A)\end{bmatrix}\Biggl[I-\prod_{i=2}^1\Bigl(I-\rie_{\Omega_i}^{3,1}(A)A\Bigr)\Biggr]=\begin{bmatrix}R_{\Omega_1,\Omega_1}^{3,1,1}(A) & R_{\Omega_1,\Omega_2}^{3,1,1}(A)\\[7.5pt] R_{\Omega_2,\Omega_1}^{3,1,1}(A) & R_{\Omega_2,\Omega_2}^{3,1,1}(A)\end{bmatrix}=A. \]}
\[ P^{MS,3,1}_{\Omega_1,\Omega_2}(A)=A\Biggl[I-\prod_{i=2}^1\Bigl(I-\rie_{\Omega_i}^{3,1}(A)A\Bigr)\Biggr]^{-1}=\begin{bmatrix}R_{\Omega_1,\Omega_1}^{3,1,1}(A) & \\[7.5pt] R_{\Omega_2,\Omega_1}^{3,1,1}(A) & R_{\Omega_2,\Omega_2}^{3,1,1}(A)\end{bmatrix}=\begin{bmatrix}a & b & 0\\ d & e & 0\\ g & h & i\end{bmatrix}. \]
By comparing this example with Example~\ref{illustration'}, we note that
\[ P^{AS,3,1}_{\Omega_1,\Omega_2}(A)=P^{BJ,3,1,1}_{\Omega_1,\Omega_2}(A),\qquad P^{MS,3,1}_{\Omega_1,\Omega_2}(A)=P^{BGS,3,1,1}_{\Omega_1,\Omega_2}(A). \]
We will see in Theorem~\ref{S->JGS} that this is due to the fact that $\{\Omega_1,\Omega_2\}$ is a partition of $[0,1]$.
\end{example}

\begin{remark}\label{r:coverCN}
The reader may wonder why the AS/MS operators have been defined under the additional assumption that $\{\Omega_1,\ldots,\Omega_\nu\}$ is a cover of $[0,1]^d$, whereas this assumption is not present in the definition of the blockdiag/blocktril operators (cf.~Definitions~\ref{bJGSp} and~\ref{amSp}).
The reason is the following.
Suppose for the moment that we define the AS/MS operators exactly as in Definition~\ref{amSp} for any choice of sets $\Omega_1,\ldots,\Omega_\nu\subseteq[0,1]^d$.
In principle, this is possible.
Now, assume that there exists a multi-index $\ii\in\{\bu,\ldots,\nn\}$ such that $\ii/\nn\not\in\Omega_1\cup\cdots\cup\Omega_\nu$.
Then, by Lemma~\ref{coverCN} below, we have
\[ D^{AS,\nn,s}_{\Omega_1,\ldots,\Omega_\nu}=D^{MS,\nn,s}_{\Omega_1,\ldots,\Omega_\nu}=\emptyset, \]
which means that the AS/MS operators \eqref{as-op}--\eqref{ms-op} are trivial (because their domains are empty).
Thus, in order for the AS/MS operators \eqref{as-op}--\eqref{ms-op} to be non-trivial, it is necessary that the union $\Omega_1\cup\cdots\cup\Omega_\nu$ contains all grid points $\{\ii/\nn:\ii=\bu,\ldots,\nn\}$. Hence, we can assume from the start that $\{\Omega_1,\ldots,\Omega_\nu\}$ is a cover of $[0,1]^d$. Indeed, if $\{\Omega_1,\ldots,\Omega_\nu\}$ was not a cover of $[0,1]^d$ but contained all grid points $\{\ii/\nn:\ii=\bu,\ldots,\nn\}$, then we could arbitrarily extend it to a cover of $[0,1]^d$ without altering the definitions of the AS/MS operators \eqref{as-op}--\eqref{ms-op}, because such definitions depend only on $\II_\nn^{\Omega_1},\ldots,\II_\nn^{\Omega_\nu}$, i.e., on the grid points taken from $\{\ii/\nn:\ii=\bu,\ldots,\nn\}$ that are contained in each set $\Omega_1,\ldots,\Omega_\nu$. In other words, if $\Omega_1',\ldots,\Omega_\nu'\subseteq[0,1]^d$ are such that
\[ \II_\nn^{\Omega_i'}=\II_\nn^{\Omega_i},\qquad i=1,\ldots,\nu, \]
then we have $D^{AS,\nn,s}_{\Omega_1',\ldots,\Omega_\nu'}=D^{AS,\nn,s}_{\Omega_1,\ldots,\Omega_\nu}$, $D^{MS,\nn,s}_{\Omega_1',\ldots,\Omega_\nu'}=D^{MS,\nn,s}_{\Omega_1,\ldots,\Omega_\nu}$, and
\begin{alignat*}{3}
P^{AS,\nn,s}_{\Omega_1',\ldots,\Omega_\nu'}(A)&=P^{AS,\nn,s}_{\Omega_1,\ldots,\Omega_\nu}(A),&&\qquad A\in D^{AS,\nn,s}_{\Omega_1,\ldots,\Omega_\nu},\\
P^{MS,\nn,s}_{\Omega_1',\ldots,\Omega_\nu'}(A)&=P^{MS,\nn,s}_{\Omega_1,\ldots,\Omega_\nu}(A),&&\qquad A\in D^{MS,\nn,s}_{\Omega_1,\ldots,\Omega_\nu}.
\end{alignat*}
\end{remark}

\begin{lemma}\label{coverCN}
Let $d,s,\nu$ be positive integers, let $\nn\in\mathbb N^d$, and let $\Omega_1,\ldots,\Omega_\nu\subseteq[0,1]^d$.
Suppose that there exists a multi-index $\ii\in\{\bu,\ldots,\nn\}$ such that $\ii/\nn\not\in\Omega_1\cup\cdots\cup\Omega_\nu$. Then,
\begin{equation}\label{D=D=0}
D^{AS,\nn,s}_{\Omega_1,\ldots,\Omega_\nu}=D^{MS,\nn,s}_{\Omega_1,\ldots,\Omega_\nu}=\emptyset.
\end{equation}
\end{lemma}
\begin{proof}
Fix any matrix $A\in\mathbb C^{N(\nn)s\times N(\nn)s}$ such that $R_{\Omega_1,\Omega_1}^{\nn,s,s}(A),\ldots,R_{\Omega_\nu,\Omega_\nu}^{\nn,s,s}(A)$ are invertible.
By definition of the RIE operator, since $\ii/\nn\not\in\Omega_1\cup\cdots\cup\Omega_\nu$ by assumption, the (block) row of the matrix $\rie_{\Omega_i}^{\nn,s}(A)$ corresponding to the multi-index $\ii$ is zero for every $i=1,\ldots,\nu$. As a consequence, the (block) row corresponding to the multi-index $\ii$ is zero also for the matrices
$\sum_{i=1}^\nu\rie_{\Omega_i}^{\nn,s}(A)$ and $I-\prod_{i=\nu}^1(I-\rie_{\Omega_i}^{\nn,s}(A)A)$.
It follows that these matrices 
are not invertible, hence $A\not\in D^{AS,\nn,s}_{\Omega_1,\ldots,\Omega_\nu}$ and $A\not\in D^{MS,\nn,s}_{\Omega_1,\ldots,\Omega_\nu}$. Since this is true regardless of the considered matrix $A$, we infer that \eqref{D=D=0} holds.
\end{proof}

\section{Preliminaries}\label{sec:prel}

In this section, we collect the necessary preliminaries for proving the main results of this paper (Theorems~\ref{blockdiag(GLT)=GLT_blocktril(GLT)=GLT} and~\ref{AS(GLT)=GLT_MS(GLT)=GLT}).
We first introduce in Section~\ref{sec:conv_set} the appropriate notion of convergence for sequences of sets.
Then, in Section~\ref{sec:connect}, we establish a fundamental connection between the blockdiag and RIE operators.
Finally, in Section~\ref{sec:Tconnect}, we prove that the AS/MS preconditioners coincide with the BJ/BGS preconditioners in the case where the sets $\Omega_1,\ldots,\Omega_\nu$ appearing in Definition~\ref{amSp} form a partition of $[0,1]^d$.

\subsection{Convergence of sets}\label{sec:conv_set}

The notion of convergence for sequences of sets that we use in the main results of this paper is reported in Definition~\ref{c-set}.
This notion draws inspiration from the theory of reduced GLT sequences~\cite[Section~3]{rg}.
Throughout this paper, if $\Omega,\Omega'\subseteq\mathbb R^d$, we denote by $\Omega\triangle\Omega'$ their symmetric difference: $\Omega\triangle\Omega'=(\Omega\setminus\Omega')\cup(\Omega'\setminus\Omega)$.

\begin{definition}[\textbf{convergence of a sequence of sets with respect to a sequence of multi-indices}]\label{c-set}
Let $\{\nn=\nn(n)\}_n$ be a sequence of positive $d$-indices, let $\{\Omega_n\}_n$ be a sequence of subsets of $\mathbb R^d$, and let $\Omega\subseteq\mathbb R^d$. We say that $\Omega_n$ converges to $\Omega$ with respect to the sequence $\{\nn=\nn(n)\}_n$, and we write $\Omega_n\xrightarrow{{\rm w.r.t.}\,\nn}\Omega$, if $N_\nn^{\Omega_n\triangle\Omega}=o(N(\nn))$ as $n\to\infty$, i.e.,
\[ \lim_{n\to\infty}\frac{N_\nn^{\Omega_n\triangle\Omega}}{N(\nn)}=0. \]
If the previous limit relation holds for a subsequence of indices $n\in\II$, i.e.,
\[ \lim_{\substack{\vphantom{\int}n\to\infty\\n\in\II}}\frac{N_\nn^{\Omega_n\triangle\Omega}}{N(\nn)}=0, \]
then we say that $\Omega_n\xrightarrow{{\rm w.r.t.}\,\nn}\Omega$ as $n\to\infty$ in $\II$. Note that writing ``$\Omega_n\xrightarrow{{\rm w.r.t.}\,\nn}\Omega$'' is equivalent to writing ``$\Omega_n\xrightarrow{{\rm w.r.t.}\,\nn}\Omega$ as $n\to\infty$'', since the latter writing means that $n$ tends to $\infty$ freely, without being constrained to stay in a specific subset of indices $\II$.
\end{definition}

The class of sets we are mainly interested in are the so-called regular sets, which are defined below.

\begin{definition}[\textbf{regular set}]
We say that $\Omega\subset\mathbb R^d$ is a regular set if it is bounded and $\mu_d(\partial\Omega)=0$.
Note that the condition ``$\mu_d(\partial\Omega)=0$'' is equivalent to ``$\chi_\Omega$ is continuous a.e.\ on $\mathbb R^d$''.
Any regular set $\Omega\subset\mathbb R^d$ is measurable and we have $\mu_d(\Omega)=\mu_d(\accentset{\circ}\Omega)=\mu_d(\overline\Omega)<\infty$.
Within the theory of Riemann integration, a regular set is usually referred to as a Peano--Jordan measurable set.
\end{definition}

\begin{remark}\label{rankD-D}
Let $\{\nn=\nn(n)\}_n$ be a sequence of positive $d$-indices, let $\{\Omega_n\}_n$ be a sequence of subsets of $\mathbb R^d$, let $\Omega\subseteq\mathbb R^d$, and suppose that $\Omega_n\xrightarrow{{\rm w.r.t.}\,\nn}\Omega$. Then,
\[ {\rm rank}(D_\nn(\chi_{\Omega_n}-\chi_\Omega))=\biggl|\biggl\{\ii\in\{\bu,\ldots,\nn\}:\,\frac\ii\nn\in\Omega_n\triangle\Omega\biggr\}\biggr|=N_\nn^{\Omega_n\triangle\Omega}=o(N(\nn)). \]
This simple observation will be used several times in the proof of the main results.
\end{remark}

\subsection{Connection between the blockdiag and RIE operators}\label{sec:connect}

Theorem~\ref{rie-lemma} highlights a fundamental connection between the blockdiag and RIE operators.
This connection will play a central role in the proof of the second main result of this paper (Theorem~\ref{AS(GLT)=GLT_MS(GLT)=GLT}).
In particular, it will allow us to transfer the GLT analysis of the BJ/BGS preconditioners to the AS/MS preconditioners.
We remark that a simplified version of Theorem~\ref{rie-lemma} was the main ingredient for the proof of \cite[Theorems~4.2--4.3]{DDM-GLT-1d}.
We also remark that, due to its relevance, Theorem~\ref{rie-lemma} can be considered as a further main result of this paper in addition to Theorems~\ref{blockdiag(GLT)=GLT_blocktril(GLT)=GLT} and~\ref{AS(GLT)=GLT_MS(GLT)=GLT}.
The proof of Theorem~\ref{rie-lemma} requires a basic property of restriction matrices and expansion operators reported in Lemma~\ref{r-lemma-beta}.

\begin{lemma}\label{r-lemma-beta}
Let $d,s,\nu$ be positive integers, let $\nn\in\mathbb N^d$, and let $\{\Omega_1,\ldots,\Omega_\nu\}$ be a partition of $[0,1]^d$. Then, for every $i,j=1,\ldots,\nu$,
\begin{equation}\label{RiRj}
R_{\Omega_i}^{\nn,s}(R_{\Omega_j}^{\nn,s})^T=\begin{cases}I, & \mbox{if}\ i=j,\\
O, & \mbox{if}\ i\ne j.
\end{cases}
\end{equation}
Moreover, for every $i,j,h,k=1,\ldots,\nu$, every $B\in\mathbb C^{N_\nn^{\Omega_i}s\times N_\nn^{\Omega_j}s}$ and every $C\in\mathbb C^{N_\nn^{\Omega_h}s\times N_\nn^{\Omega_k}s}$,
\begin{equation}\label{EijEhk}
E_{\Omega_i,\Omega_j}^{\nn,s,s}(B)E_{\Omega_h,\Omega_k}^{\nn,s,s}(C)=\begin{cases}E_{\Omega_i,\Omega_k}^{\nn,s,s}(BC), & \mbox{if}\ j=h,\\
O, & \mbox{if}\ j\ne h.
\end{cases}
\end{equation}
\end{lemma}
\begin{proof}
Let $1\le i,j\le\nu$. If $i=j$, then \eqref{RiRj} follows immediately from Lemma~\ref{r-lemma}. If $i\ne j$, then \eqref{RiRj} follows from Definition~\ref{R-matrix} and from standard block matrix multiplications. Alternatively, we can prove it by using the properties of tensor products:
\begin{align*}
R_{\Omega_i}^{\nn,s}(R_{\Omega_j}^{\nn,s})^T&=(R_{\Omega_i}^{\nn,1}\otimes I_s)(R_{\Omega_j}^{\nn,1}\otimes I_s)^T=(R_{\Omega_i}^{\nn,1}\otimes I_s)((R_{\Omega_j}^{\nn,1})^T\otimes I_s)=R_{\Omega_i}^{\nn,1}(R_{\Omega_j}^{\nn,1})^T\otimes I_s\\
&=[(\ee_\ii^{(\nn)})^T]_{\ii\in\II_\nn^{\Omega_i}}[\ee_\jj^{(\nn)}]^{\jj\in\II_\nn^{\Omega_j}}\otimes I_s=\begin{bmatrix}0\end{bmatrix}_{\ii\in\II_\nn^{\Omega_i}}^{\jj\in\II_\nn^{\Omega_j}}\otimes I_s=O\otimes I_s=O,
\end{align*}
where the third-to-last equality follows from the fact that $(\ee_\ii^{(\nn)})^T\ee_\jj^{(\nn)}=0$ for every $\ii\in\II_\nn^{\Omega_i}$ and every $\jj\in\II_\nn^{\Omega_j}$, because $\II_\nn^{\Omega_i}$ and $\II_\nn^{\Omega_j}$ are disjoint due to the assumption that $\{\Omega_1,\ldots,\Omega_\nu\}$ is a partition of $[0,1]^d$.

To prove \eqref{EijEhk}, let $1\le i,j,h,k\le\nu$, let $B\in\mathbb C^{N_\nn^{\Omega_i}s\times N_\nn^{\Omega_j}s}$, and let $C\in\mathbb C^{N_\nn^{\Omega_h}s\times N_\nn^{\Omega_k}s}$. By definition of expansion operators, we have
\[ E_{\Omega_i,\Omega_j}^{\nn,s,s}(B)E_{\Omega_h,\Omega_k}^{\nn,s,s}(C)=(R_{\Omega_i}^{\nn,s})^T\,B\,R_{\Omega_j}^{\nn,s}(R_{\Omega_h}^{\nn,s})^T\,C\,R_{\Omega_k}^{\nn,s}. \]
By \eqref{RiRj}, this is equal to $O$ if $j\ne h$ and is equal to $(R_{\Omega_i}^{\nn,s})^T\,BC\,R_{\Omega_k}^{\nn,s}=E_{\Omega_i,\Omega_k}^{\nn,s,s}(BC)$ if $j=h$.
\end{proof}

\begin{theorem}\label{rie-lemma}
Let $d,s,\nu$ be positive integers, let $\nn\in\mathbb N^d$, and let $\{\Omega_1,\ldots,\Omega_\nu\}$ be a partition of $[0,1]^d$. Then, for every $A\in\mathbb C^{N(\nn)s\times N(\nn)s}$ such that $R_{\Omega_1,\Omega_1}^{\nn,s,s}(A),\ldots,R_{\Omega_\nu,\Omega_\nu}^{\nn,s,s}(A)$ are invertible, it holds that $\mathop{\rm blockdiag}^{\nn,s,s}_{\Omega_1,\ldots,\Omega_\nu}(A)$ is invertible,
\begin{equation}\label{scaturigine}
\biggl[\mathop{\rm blockdiag}^{\nn,s,s}_{\Omega_1,\ldots,\Omega_\nu}(A)\biggr]^{-1}=\sum_{i=1}^{\nu}\rie_{\Omega_i}^{\nn,s}(A),
\end{equation}
and, for $i=1,\ldots,\nu$,
\begin{align}
\rie_{\Omega_i}^{\nn,s}(A)&=D_\nn(\chi_{\Omega_i}I_s)\biggl[\mathop{\rm blockdiag}^{\nn,s,s}_{\Omega_1,\ldots,\Omega_\nu}(A)\biggr]^{-1}\label{er1}\\
&=\biggl[\mathop{\rm blockdiag}^{\nn,s,s}_{\Omega_1,\ldots,\Omega_\nu}(A)\biggr]^{-1}D_\nn(\chi_{\Omega_i}I_s)\label{er2}\\
&=D_\nn(\chi_{\Omega_i}I_s)\biggl[\mathop{\rm blockdiag}^{\nn,s,s}_{\Omega_1,\ldots,\Omega_\nu}(A)\biggr]^{-1}D_\nn(\chi_{\Omega_i}I_s)\label{er3}\\
&=Z_{\Omega_i,\Omega_i}^{\nn,s,s}\biggl(\biggl[\mathop{\rm blockdiag}^{\nn,s,s}_{\Omega_1,\ldots,\Omega_\nu}(A)\biggr]^{-1}\biggr).\label{er4}
\end{align}
\end{theorem}
\begin{proof}
Let $A\in\mathbb C^{N(\nn)s\times N(\nn)s}$ be such that $R_{\Omega_1,\Omega_1}^{\nn,s,s}(A),\ldots,R_{\Omega_\nu,\Omega_\nu}^{\nn,s,s}(A)$ are invertible.
In order to prove that $\mathop{\rm blockdiag}^{\nn,s,s}_{\Omega_1,\ldots,\Omega_\nu}(A)$ is invertible and that \eqref{scaturigine} holds, it suffices to show that $\mathop{\rm blockdiag}^{\nn,s,s}_{\Omega_1,\ldots,\Omega_\nu}(A)$ multiplied with the right-hand side of \eqref{scaturigine} yields the identity matrix. This is done by direct computation:
\begin{align*}
\mathop{\rm blockdiag}^{\nn,s,s}_{\Omega_1,\ldots,\Omega_\nu}(A)\sum_{i=1}^{\nu}\rie_{\Omega_i}^{\nn,s}(A)
&=\sum_{i=1}^{\nu}Z_{\Omega_i,\Omega_i}^{\nn,s,s}(A)\sum_{i=1}^{\nu}E_{\Omega_i,\Omega_i}^{\nn,s,s}({R_{\Omega_i,\Omega_i}^{\nn,s,s}(A)}^{-1})\qquad\mbox{\footnotesize(by Definitions~\ref{bJGSp}--\ref{rie})}\\
&=\sum_{i=1}^{\nu}E_{\Omega_i,\Omega_i}^{\nn,s,s}(R_{\Omega_i,\Omega_i}^{\nn,s,s}(A))\sum_{i=1}^{\nu}E_{\Omega_i,\Omega_i}^{\nn,s,s}({R_{\Omega_i,\Omega_i}^{\nn,s,s}(A)}^{-1})\qquad\mbox{\footnotesize(by \eqref{ER=Z})}\\
&=\sum_{i,j=1}^{\nu}E_{\Omega_i,\Omega_i}^{\nn,s,s}(R_{\Omega_i,\Omega_i}^{\nn,s,s}(A))E_{\Omega_j,\Omega_j}^{\nn,s,s}({R_{\Omega_j,\Omega_j}^{\nn,s,s}(A)}^{-1})\\
&=\sum_{i=1}^{\nu}E_{\Omega_i,\Omega_i}^{\nn,s,s}(I)\qquad\mbox{\footnotesize(by Lemma~\ref{r-lemma-beta})}\\
&=\sum_{i=1}^{\nu}(R_{\Omega_i}^{\nn,s})^TR_{\Omega_i}^{\nn,s}\qquad\mbox{\footnotesize(by Definition~\ref{ZRE-operators})}\\
&=\sum_{i=1}^{\nu}D_\nn(\chi_{\Omega_i}I_s)\qquad\mbox{\footnotesize(by Lemma~\ref{r-lemma})}\\
&=D_\nn\Biggl(\sum_{i=1}^\nu\chi_{\Omega_i}I_s\Biggr)\qquad\mbox{\footnotesize(by linearity of the operator $a\mapsto D_\nn(aI_s)$)}\\
&=D_\nn(I_s)\qquad\mbox{\footnotesize($\sum_{i=1}^\nu\chi_{\Omega_i}=1$ identically because $\{\Omega_1,\ldots,\Omega_\nu\}$ is a partition of $[0,1]^d$)}\\
&=I_{N(\nn)s}.
\end{align*}
Now, we fix $1\le i\le\nu$ and we prove \eqref{er1}--\eqref{er4}. Actually, we only prove \eqref{er1}, because the proofs of \eqref{er2}--\eqref{er3} are analogous to that of \eqref{er1}, while \eqref{er4} follows immediately from \eqref{er3} and Definition~\ref{ZRE-operators}. We have
{\allowdisplaybreaks\begin{align*}
D_\nn(\chi_{\Omega_i}I_s)\biggl[\mathop{\rm blockdiag}^{\nn,s,s}_{\Omega_1,\ldots,\Omega_\nu}(A)\biggr]^{-1}&=D_\nn(\chi_{\Omega_i}I_s)\sum_{j=1}^\nu\rie_{\Omega_j}^{\nn,s}(A)\qquad\mbox{\footnotesize(by \eqref{scaturigine})}\\
&=\sum_{j=1}^\nu D_\nn(\chi_{\Omega_i}I_s)E_{\Omega_j,\Omega_j}^{\nn,s,s}({R_{\Omega_j,\Omega_j}^{\nn,s,s}(A)}^{-1})\qquad\mbox{\footnotesize(by Definition~\ref{rie})}\\
&=\sum_{j=1}^\nu (R_{\Omega_i}^{\nn,s})^TR_{\Omega_i}^{\nn,s}(R_{\Omega_j}^{\nn,s})^T{R_{\Omega_j,\Omega_j}^{\nn,s,s}(A)}^{-1}R_{\Omega_j}^{\nn,s}\qquad\mbox{\footnotesize(by Lemma~\ref{r-lemma} and Definition~\ref{ZRE-operators})}\\
&=(R_{\Omega_i}^{\nn,s})^T{R_{\Omega_i,\Omega_i}^{\nn,s,s}(A)}^{-1}R_{\Omega_i}^{\nn,s}\qquad\mbox{\footnotesize(by Lemma~\ref{r-lemma-beta})}\\
&=E_{\Omega_i,\Omega_i}^{\nn,s,s}({R_{\Omega_i,\Omega_i}^{\nn,s,s}(A)}^{-1})\qquad\mbox{\footnotesize(by Definition~\ref{ZRE-operators})}\\
&=\rie_{\Omega_i}^{\nn,s}(A)\qquad\mbox{\footnotesize(by Definition~\ref{rie}).} \qedhere
\end{align*}}%
\end{proof}

\subsection{Connection between the BJ/BGS and AS/MS preconditioners}\label{sec:Tconnect}

Theorem~\ref{S->JGS} shows that, in the case where the sets $\Omega_1,\ldots,\Omega_\nu$ appearing in Definition~\ref{amSp} form a partition of $[0,1]^d$, the AS preconditioner coincides with the BJ preconditioner and the MS preconditioner coincides with the BGS preconditioner.
Although this connection between the BJ/BGS and AS/MS preconditioners is known within the DDM community, it is necessary to prove it starting from the new notations and definitions presented in Section~\ref{bJ/bGS/AS/MS-def-e}. We remark that, due to its relevance, Theorem~\ref{S->JGS} can be considered as a further main result of this paper in addition to Theorems~\ref{blockdiag(GLT)=GLT_blocktril(GLT)=GLT} and~\ref{AS(GLT)=GLT_MS(GLT)=GLT}.

\begin{theorem}\label{S->JGS}
Let $d,s,\nu$ be positive integers, let $\nn\in\mathbb N^d$, and let $\{\Omega_1,\ldots,\Omega_\nu\}$ be a partition of $[0,1]^d$. Then,
\begin{alignat}{3}
P_{\Omega_1,\ldots,\Omega_\nu}^{AS,\nn,s}(A)&=P_{\Omega_1,\ldots,\Omega_\nu}^{BJ,\nn,s,s}(A),&\qquad A&\in D^{AS,\nn,s}_{\Omega_1,\ldots,\Omega_\nu},\label{AS=BJ}\\[5pt]
P_{\Omega_1,\ldots,\Omega_\nu}^{MS,\nn,s}(A)&=P_{\Omega_1,\ldots,\Omega_\nu}^{BGS,\nn,s,s}(A),&\qquad A&\in D^{MS,\nn,s}_{\Omega_1,\ldots,\Omega_\nu}.\label{MS=BGS}
\end{alignat}
\end{theorem}
\begin{proof}[Proof of \eqref{AS=BJ}]
Suppose $A\in D^{AS,\nn,s}_{\Omega_1,\ldots,\Omega_\nu}$. Then, by definition of $D^{AS,\nn,s}_{\Omega_1,\ldots,\Omega_\nu}$, the matrices $R_{\Omega_1,\Omega_1}^{\nn,s,s}(A),\ldots,R_{\Omega_\nu,\Omega_\nu}^{\nn,s,s}(A)$ are invertible.
We can therefore apply Theorem~\ref{rie-lemma} and conclude that $\mathop{\rm blockdiag}^{\nn,s,s}_{\Omega_1,\ldots,\Omega_\nu}(A)$ is invertible with inverse given by \eqref{scaturigine}. Thus,
\begin{align*}
P_{\Omega_1,\ldots,\Omega_\nu}^{BJ,\nn,s,s}&=\mathop{\rm blockdiag}^{\nn,s,s}_{\Omega_1,\ldots,\Omega_\nu}(A)\qquad\mbox{\footnotesize(by Definition~\ref{bJGSp})}\\
&=\Biggl[\biggl[\mathop{\rm blockdiag}^{\nn,s,s}_{\Omega_1,\ldots,\Omega_\nu}(A)\biggr]^{-1}\Biggr]^{-1}=\Biggl[\sum_{i=1}^\nu \rie_{\Omega_i}^{\nn,s}(A)\Biggr]^{-1}\qquad\mbox{\footnotesize(by Theorem~\ref{rie-lemma})}\\
&=P_{\Omega_1,\ldots,\Omega_\nu}^{AS,\nn,s}(A)\qquad\mbox{\footnotesize(by Definition~\ref{amSp}).}
\end{align*}
{\em Proof of \eqref{MS=BGS}.} Suppose $A\in D^{MS,\nn,s}_{\Omega_1,\ldots,\Omega_\nu}$. 
By Definitions~\ref{bJGSp} and~\ref{amSp}, we have
\begin{align*}
P_{\Omega_1,\ldots,\Omega_\nu}^{BGS,\nn,s,s}(A)&=\sum_{1\le j\le i\le\nu}Z_{\Omega_i,\Omega_j}^{\nn,s,s}(A),\\
P_{\Omega_1,\ldots,\Omega_\nu}^{MS,\nn,s}(A)&=A\Biggl[I-\prod_{i=\nu}^1\Bigl(I-\rie_{\Omega_i}^{\nn,s}(A)A\Bigr)\Biggr]^{-1}.
\end{align*}
We want to show that one of the following equivalent equations is satisfied:
\begin{align}
\notag P_{\Omega_1,\ldots,\Omega_\nu}^{BGS,\nn,s,s}(A)=P_{\Omega_1,\ldots,\Omega_\nu}^{MS,\nn,s}(A)&\ \iff\ P_{\Omega_1,\ldots,\Omega_\nu}^{BGS,\nn,s}(A)=A\Biggl[I-\prod_{i=\nu}^1\Bigl(I-\rie_{\Omega_i}^{\nn,s}(A)A\Bigr)\Biggr]^{-1}\\
\notag&\ \iff\ P_{\Omega_1,\ldots,\Omega_\nu}^{BGS,\nn,s}(A)\Biggl[I-\prod_{i=\nu}^1\Bigl(I-\rie_{\Omega_i}^{\nn,s}(A)A\Bigr)\Biggr]=A\\
\notag&\ \iff\ P_{\Omega_1,\ldots,\Omega_\nu}^{BGS,\nn,s,s}(A)\prod_{i=\nu}^1\underbrace{\Bigl(I-\rie_{\Omega_i}^{\nn,s}(A)A\Bigr)}_{F_i}=P_{\Omega_1,\ldots,\Omega_\nu}^{BGS,\nn,s}(A)-A\\
&\ \iff\ P_{\Omega_1,\ldots,\Omega_\nu}^{BGS,\nn,s,s}(A)\prod_{i=\nu}^1F_i=-\sum_{1\le i<j\le\nu}Z_{\Omega_i,\Omega_j}^{\nn,s,s}(A),\label{ultravalid}
\end{align}
where in the last equivalence we used the relation between $A$ and $P_{\Omega_1,\ldots,\Omega_\nu}^{BGS,\nn,s}(A)$ highlighted in Remark~\ref{A-BJ-BGS}.
We prove the last equation \eqref{ultravalid}. To this end, we define
\[ G_k=P_{\Omega_1,\ldots,\Omega_\nu}^{BGS,\nn,s,s}(A)\prod_{i=\nu}^kF_i,\qquad k=\nu,\ldots,1. \]
We show by induction that, for every $k=\nu,\ldots,1$,
\begin{equation}\label{polirat}
G_k=\sum_{1\le j\le i\le k-1}Z_{\Omega_i,\Omega_j}^{\nn,s,s}(A)-\sum_{k\le i<j\le\nu}Z_{\Omega_i,\Omega_j}^{\nn,s,s}(A).
\end{equation}
Once this is done, by selecting $k=1$ in \eqref{polirat}, we obtain
\[ G_1=-\sum_{1\le i<j\le\nu}Z_{\Omega_i,\Omega_j}^{\nn,s,s}(A), \]
which is precisely \eqref{ultravalid}. It only remains to prove \eqref{polirat} by induction on $k=\nu,\ldots,1$.

\medskip

\noindent{\em Case $k=\nu$.} We have
\begin{align*}
G_\nu&=P_{\Omega_1,\ldots,\Omega_\nu}^{BGS,\nn,s,s}(A)F_\nu=P_{\Omega_1,\ldots,\Omega_\nu}^{BGS,\nn,s,s}(A)(I-\rie_{\Omega_\nu}^{\nn,s}(A)A)\qquad\mbox{\footnotesize(by definition of $F_\nu$)}\\
&=P_{\Omega_1,\ldots,\Omega_\nu}^{BGS,\nn,s,s}(A)-P_{\Omega_1,\ldots,\Omega_\nu}^{BGS,\nn,s,s}(A)E_{\Omega_\nu,\Omega_\nu}^{\nn,s,s}({R_{\Omega_\nu,\Omega_\nu}^{\nn,s,s}(A)}^{-1})A\qquad\mbox{\footnotesize(by definition of $\rie_{\Omega_\nu}^{\nn,s}(A)$)}\\
&=P_{\Omega_1,\ldots,\Omega_\nu}^{BGS,\nn,s,s}(A)-\Biggl[\sum_{1\le j\le i\le\nu}Z_{\Omega_i,\Omega_j}^{\nn,s,s}(A)\Biggr]E_{\Omega_\nu,\Omega_\nu}^{\nn,s,s}({R_{\Omega_\nu,\Omega_\nu}^{\nn,s,s}(A)}^{-1})A\qquad\mbox{\footnotesize(by definition of $P_{\Omega_1,\ldots,\Omega_\nu}^{BGS,\nn,s,s}(A)$)}\\
&=P_{\Omega_1,\ldots,\Omega_\nu}^{BGS,\nn,s,s}(A)-\Biggl[\sum_{1\le j\le i\le\nu}E_{\Omega_i,\Omega_j}^{\nn,s,s}(R_{\Omega_i,\Omega_j}^{\nn,s,s}(A))\Biggr]E_{\Omega_\nu,\Omega_\nu}^{\nn,s,s}({R_{\Omega_\nu,\Omega_\nu}^{\nn,s,s}(A)}^{-1})A\qquad\mbox{\footnotesize(by \eqref{ER=Z})}\\
&=P_{\Omega_1,\ldots,\Omega_\nu}^{BGS,\nn,s,s}(A)-\Biggl[\sum_{1\le j\le i\le\nu}\underbrace{E_{\Omega_i,\Omega_j}^{\nn,s,s}(R_{\Omega_i,\Omega_j}^{\nn,s,s}(A))E_{\Omega_\nu,\Omega_\nu}^{\nn,s,s}({R_{\Omega_\nu,\Omega_\nu}^{\nn,s,s}(A)}^{-1})}_{\textup{this is zero if $j\ne\nu$ by Lemma~\ref{r-lemma-beta}}}\Biggr]A\\
&=P_{\Omega_1,\ldots,\Omega_\nu}^{BGS,\nn,s,s}(A)-E_{\Omega_\nu,\Omega_\nu}^{\nn,s,s}(I)A\qquad\mbox{\footnotesize(by Lemma~\ref{r-lemma-beta})}\\
&=P_{\Omega_1,\ldots,\Omega_\nu}^{BGS,\nn,s,s}(A)-E_{\Omega_\nu,\Omega_\nu}^{\nn,s,s}(I)\sum_{i,j=1}^\nu Z_{\Omega_i,\Omega_j}^{\nn,s,s}(A)\qquad\mbox{\footnotesize(by Remark~\ref{A-BJ-BGS})}\\
&=P_{\Omega_1,\ldots,\Omega_\nu}^{BGS,\nn,s,s}(A)-E_{\Omega_\nu,\Omega_\nu}^{\nn,s,s}(I)\sum_{i,j=1}^\nu E_{\Omega_i,\Omega_j}^{\nn,s,s}(R_{\Omega_i,\Omega_j}^{\nn,s,s}(A))\qquad\mbox{\footnotesize(by \eqref{ER=Z})}\\
&=P_{\Omega_1,\ldots,\Omega_\nu}^{BGS,\nn,s,s}(A)-\sum_{i,j=1}^\nu\underbrace{E_{\Omega_\nu,\Omega_\nu}^{\nn,s,s}(I)E_{\Omega_i,\Omega_j}^{\nn,s,s}(R_{\Omega_i,\Omega_j}^{\nn,s,s}(A))}_{\textup{this is zero if $i\ne\nu$ by Lemma~\ref{r-lemma-beta}}}\\
&=P_{\Omega_1,\ldots,\Omega_\nu}^{BGS,\nn,s,s}(A)-\sum_{j=1}^\nu E_{\Omega_\nu,\Omega_j}^{\nn,s,s}(R_{\Omega_\nu,\Omega_j}^{\nn,s,s}(A))\qquad\mbox{\footnotesize(by Lemma~\ref{r-lemma-beta})}\\
&=\sum_{1\le j\le i\le\nu} Z_{\Omega_i,\Omega_j}^{\nn,s,s}(A)-\sum_{j=1}^\nu Z_{\Omega_\nu,\Omega_j}^{\nn,s,s}(A)\qquad\mbox{\footnotesize(by definition of $P_{\Omega_1,\ldots,\Omega_\nu}^{BGS,\nn,s,s}(A)$ and \eqref{ER=Z})}\\
&=\sum_{1\le j\le i\le\nu-1} Z_{\Omega_i,\Omega_j}^{\nn,s,s}(A),
\end{align*}
which is \eqref{polirat} for $k=\nu$.

\medskip

\noindent{\em Case $1\le k\le\nu-1$.} Suppose that  \eqref{polirat} holds for the index $k+1$, i.e.,
\begin{equation}\label{polirat+}
G_{k+1}=\sum_{1\le j\le i\le k}Z_{\Omega_i,\Omega_j}^{\nn,s,s}(A)-\sum_{k+1\le i<j\le\nu}Z_{\Omega_i,\Omega_j}^{\nn,s,s}(A).
\end{equation}
We show that \eqref{polirat} also holds for the index $k$. We have
\begin{align*}
G_k&=G_{k+1}F_k=G_{k+1}(I-\rie_{\Omega_k}^{\nn,s}(A)A)\qquad\mbox{\footnotesize(by definition of $F_k$)}\\
&=G_{k+1}-G_{k+1}E_{\Omega_k,\Omega_k}^{\nn,s,s}({R_{\Omega_k,\Omega_k}^{\nn,s,s}(A)}^{-1})A\qquad\mbox{\footnotesize(by definition of $\rie_{\Omega_k}^{\nn,s}(A)$)}\\
&=G_{k+1}-\Biggl[\sum_{1\le j\le i\le k}Z_{\Omega_i,\Omega_j}^{\nn,s,s}(A)-\sum_{k+1\le i<j\le\nu}Z_{\Omega_i,\Omega_j}^{\nn,s,s}(A)\Biggr]E_{\Omega_k,\Omega_k}^{\nn,s,s}({R_{\Omega_k,\Omega_k}^{\nn,s,s}(A)}^{-1})A\qquad\mbox{\footnotesize(by induction hypothesis \eqref{polirat+})}\\
&=G_{k+1}-\Biggl[\sum_{1\le j\le i\le k}E_{\Omega_i,\Omega_j}^{\nn,s,s}(R_{\Omega_i,\Omega_j}^{\nn,s,s}(A))-\sum_{k+1\le i<j\le\nu}E_{\Omega_i,\Omega_j}^{\nn,s,s}(R_{\Omega_i,\Omega_j}^{\nn,s,s}(A))\Biggr]E_{\Omega_k,\Omega_k}^{\nn,s,s}({R_{\Omega_k,\Omega_k}^{\nn,s,s}(A)}^{-1})A\qquad\mbox{\footnotesize(by \eqref{ER=Z})}\\
&=G_{k+1}-\Biggl[\sum_{1\le j\le i\le k}\underbrace{E_{\Omega_i,\Omega_j}^{\nn,s,s}(R_{\Omega_i,\Omega_j}^{\nn,s,s}(A))E_{\Omega_k,\Omega_k}^{\nn,s,s}({R_{\Omega_k,\Omega_k}^{\nn,s,s}(A)}^{-1})}_{\textup{this is zero if $j\ne k$ by Lemma~\ref{r-lemma-beta}}}\\
&\qquad\qquad\qquad-\sum_{k+1\le i<j\le\nu}\underbrace{E_{\Omega_i,\Omega_j}^{\nn,s,s}(R_{\Omega_i,\Omega_j}^{\nn,s,s}(A))E_{\Omega_k,\Omega_k}^{\nn,s,s}({R_{\Omega_k,\Omega_k}^{\nn,s,s}(A)}^{-1})}_{\textup{this is zero if $j\ne k$ by Lemma~\ref{r-lemma-beta}}}\Biggr]A\\
&=G_{k+1}-E_{\Omega_k,\Omega_k}^{\nn,s,s}(I)A\qquad\mbox{\footnotesize(by Lemma~\ref{r-lemma-beta})}\\
&=G_{k+1}-E_{\Omega_k,\Omega_k}^{\nn,s,s}(I)\sum_{i,j=1}^\nu Z_{\Omega_i,\Omega_j}^{\nn,s,s}(A)\qquad\mbox{\footnotesize(by Remark~\ref{A-BJ-BGS})}\\
&=G_{k+1}-E_{\Omega_k,\Omega_k}^{\nn,s,s}(I)\sum_{i,j=1}^\nu E_{\Omega_i,\Omega_j}^{\nn,s,s}(R_{\Omega_i,\Omega_j}^{\nn,s,s}(A))\qquad\mbox{\footnotesize(by \eqref{ER=Z})}\\
&=G_{k+1}-\sum_{i,j=1}^\nu\underbrace{E_{\Omega_k,\Omega_k}^{\nn,s,s}(I)E_{\Omega_i,\Omega_j}^{\nn,s,s}(R_{\Omega_i,\Omega_j}^{\nn,s,s}(A))}_{\textup{this is zero if $i\ne k$ by Lemma~\ref{r-lemma-beta}}}\\
&=G_{k+1}-\sum_{j=1}^\nu E_{\Omega_k,\Omega_j}^{\nn,s,s}(R_{\Omega_k,\Omega_j}^{\nn,s,s}(A))\qquad\mbox{\footnotesize(by Lemma~\ref{r-lemma-beta})}\\
&=\Biggl[\sum_{1\le j\le i\le k}Z_{\Omega_i,\Omega_j}^{\nn,s,s}(A)-\sum_{k+1\le i<j\le\nu}Z_{\Omega_i,\Omega_j}^{\nn,s,s}(A)\Biggr]-\sum_{j=1}^\nu Z_{\Omega_k,\Omega_j}^{\nn,s,s}(A)\qquad\mbox{\footnotesize(by induction hypothesis \eqref{polirat+} and \eqref{ER=Z})}\\
&=\Biggl[\sum_{1\le j\le i\le k}Z_{\Omega_i,\Omega_j}^{\nn,s,s}(A)-\sum_{k+1\le i<j\le\nu}Z_{\Omega_i,\Omega_j}^{\nn,s,s}(A)\Biggr]-\sum_{j=1}^k Z_{\Omega_k,\Omega_j}^{\nn,s,s}(A)-\sum_{j=k+1}^\nu Z_{\Omega_k,\Omega_j}^{\nn,s,s}(A)\\
&=\sum_{1\le j\le i\le k-1}Z_{\Omega_i,\Omega_j}^{\nn,s,s}(A)-\sum_{k\le i<j\le\nu}Z_{\Omega_i,\Omega_j}^{\nn,s,s}(A),
\end{align*}
which is precisely \eqref{polirat} for the index $k$.
\end{proof}

\section{Main results}\label{sec:main}

In this section, we state and prove the main results of this paper, namely, Theorem~\ref{blockdiag(GLT)=GLT_blocktril(GLT)=GLT} (which presents the GLT analysis of the BJ/BGS preconditioners) and Theorem~\ref{AS(GLT)=GLT_MS(GLT)=GLT} (which presents the GLT analysis of the AS/MS preconditioners).
The proof of Theorem~\ref{AS(GLT)=GLT_MS(GLT)=GLT} heavily rely on the GLT analysis of the RIE operator, which is presented in Theorem~\ref{rie-thm}.

\subsection{GLT analysis of the BJ/BGS preconditioners}\label{sec:BJ-BGS}

Theorem~\ref{blockdiag(GLT)=GLT_blocktril(GLT)=GLT} is the first main result of this paper. It can be seen as the multidimensional generalization of \cite[Theorems~4.1--4.2]{GLH} and \cite[Theorem~4.1]{DDM-GLT-1d}. Theorem~\ref{blockdiag(GLT)=GLT_blocktril(GLT)=GLT} provides the GLT analysis of the blockdiag and blocktril operators. In particular, it shows that any sequence of BJ/BGS preconditioners for a GLT sequence with symbol $\kappa$ is again a GLT sequence with symbol $\kappa$, as long as a suitable convergence assumption is satisfied by the subsets $\Omega_{1,n},\ldots,\Omega_{\nu,n}$ used in the construction of the preconditioners.

\begin{theorem}\label{blockdiag(GLT)=GLT_blocktril(GLT)=GLT}
Let $d,s,t,\nu$ be positive integers, let $\{\nn=\nn(n)\}_n$ be a sequence of positive $d$-indices tending to $\infty$, and, for every~$n$, let $\Omega_{1,n},\ldots,\Omega_{\nu,n}$ be subsets of $[0,1]^d$ such that $\Omega_{i,n}\xrightarrow{{\rm w.r.t.}\,\nn}\Omega_i$ for every $i=1,\ldots,\nu$, where $\Omega_1,\ldots,\Omega_\nu\subseteq[0,1]^d$ are regular.
Let $\kappa:[0,1]^d\times[-\pi,\pi]^d\to\mathbb C^{s\times t}$ be measurable and let $\{A_n\}_n$ be a sequence of matrices with $A_n$ of size $N(\nn)s\times N(\nn)t$ and $\{A_n\}_n\sim_{\rm GLT}\kappa$. Then,
\begin{align}
\biggl\{\mathop{\rm blockdiag}_{\Omega_{1,n},\ldots,\Omega_{\nu,n}}^{\nn,s,t}(A_n)\biggr\}_n&\sim_{\rm GLT}\kappa(\xx,\btheta)\sum_{i=1}^\nu\chi_{\Omega_i}(\xx),\label{clueq}\\
\biggl\{\mathop{\rm blocktril}_{\Omega_{1,n},\ldots,\Omega_{\nu,n}}^{\nn,s,t}(A_n)\biggr\}_n&\sim_{\rm GLT}\kappa(\xx,\btheta)\sum_{1\le j\le i\le\nu}\chi_{\Omega_i\cap\Omega_j}(\xx).\label{clueq'}
\end{align}
In particular, if $\sum_{i=1}^\nu\chi_{\Omega_i}=1$ a.e.\ on $[0,1]^d$, which happens, for instance, if $\{\Omega_1,\ldots,\Omega_\nu\}$ is a partition of $[0,1]^d$, then
\begin{align}
\biggl\{\mathop{\rm blockdiag}_{\Omega_{1,n},\ldots,\Omega_{\nu,n}}^{\nn,s,t}(A_n)\biggr\}_n&\sim_{\rm GLT}\kappa,\label{in-pirtico}\\
\biggl\{\mathop{\rm blocktril}_{\Omega_{1,n},\ldots,\Omega_{\nu,n}}^{\nn,s,t}(A_n)\biggr\}_n&\sim_{\rm GLT}\kappa.\label{in-pirtico'}
\end{align}
\end{theorem}
\begin{proof}
If $\sum_{i=1}^\nu\chi_{\Omega_i}=1$ a.e.\ on $[0,1]^d$, then \eqref{clueq} implies \eqref{in-pirtico}.
Similarly, if $\sum_{i=1}^\nu\chi_{\Omega_i}=1$ a.e.\ on $[0,1]^d$, then $\Omega_i\cap\Omega_j$ has zero measure for every $i,j=1,\ldots,\nu$ with $i\ne j$, hence $\sum_{1\le j\le i\le\nu}\chi_{\Omega_i\cap\Omega_j}=\sum_{i=1}^\nu\chi_{\Omega_i}=1$ a.e.\ on $[0,1]^d$ and \eqref{clueq'} implies \eqref{in-pirtico'}. Thus, we only have to prove \eqref{clueq}--\eqref{clueq'}.

\medskip

\noindent{\em Proof of \eqref{clueq}.}
Let us define, for every $n$,
\begin{align*}
P_n^{BJ}&=\mathop{\rm blockdiag}_{\Omega_{1,n},\ldots,\Omega_{\nu,n}}^{\nn,s,t}(A_n)=\sum_{i=1}^\nu Z_{\Omega_{i,n},\Omega_{i,n}}^{\nn,s,t}(A_n)=\sum_{i=1}^\nu D_\nn(\chi_{\Omega_{i,n}}I_s)A_nD_\nn(\chi_{\Omega_{i,n}}I_t),\\
\tilde P_n^{BJ}&=\sum_{i=1}^\nu D_\nn(\chi_{\Omega_i}I_s)A_nD_\nn(\chi_{\Omega_i}I_t).
\end{align*}
Note that $\tilde P_n^{BJ}$ is obtained from $P_n^{BJ}$ by replacing the sets $\Omega_{1,n},\ldots,\Omega_{\nu,n}$ with their limits $\Omega_1,\ldots,\Omega_\nu$. Since $\Omega_1,\ldots,\Omega_\nu$ are regular by assumption, their characteristic functions $\chi_{\Omega_1},\ldots,\chi_{\Omega_\nu}$ are continuous a.e.\ on $[0,1]^d$ and {\bf GLT2}--{\bf GLT3} imply that
\[ \{\tilde P_n^{BJ}\}_n\sim_{\rm GLT}\sum_{i=1}^\nu\chi_{\Omega_i}(\xx)I_s\,\kappa(\xx,\btheta)\,\chi_{\Omega_i}(\xx)I_t=\sum_{i=1}^\nu\chi_{\Omega_i}(\xx)\kappa(\xx,\btheta)=\kappa^{BJ}(\xx,\btheta). \]
We show that
\begin{equation}\label{P-P=0}
\{P_n^{BJ}-\tilde P_n^{BJ}\}_n\sim_\sigma0.
\end{equation}
Once this is done, the thesis \eqref{clueq} follows from the GLT relation $\{\tilde P_n^{BJ}\}_n\sim_{\rm GLT}\kappa^{BJ}$ and {\bf GLT2}--{\bf GLT3}. To prove \eqref{P-P=0}, let us write
\begin{align*}
&P_n^{BJ}-\tilde P_n^{BJ}=\sum_{i=1}^\nu\Bigl[D_\nn(\chi_{\Omega_{i,n}}I_s)A_nD_\nn(\chi_{\Omega_{i,n}}I_t)-D_\nn(\chi_{\Omega_i}I_s)A_nD_\nn(\chi_{\Omega_i}I_t)\Bigr]\\
&=\sum_{i=1}^\nu\Bigl[(D_\nn(\chi_{\Omega_{i,n}}I_s)-D_\nn(\chi_{\Omega_i}I_s))A_nD_\nn(\chi_{\Omega_{i,n}}I_t)+D_\nn(\chi_{\Omega_i}I_s)A_n(D_\nn(\chi_{\Omega_{i,n}}I_t)-D_\nn(\chi_{\Omega_i}I_t))\Bigr]\\
&=\sum_{i=1}^\nu\Bigl[D_\nn((\chi_{\Omega_{i,n}}-\chi_{\Omega_i})I_s)A_nD_\nn(\chi_{\Omega_{i,n}}I_t)+D_\nn(\chi_{\Omega_i}I_s)A_nD_\nn((\chi_{\Omega_{i,n}}-\chi_{\Omega_i})I_t)\Bigr]\\
&=\sum_{i=1}^\nu\Bigl[(D_\nn(\chi_{\Omega_{i,n}}-\chi_{\Omega_i})\otimes I_s)A_n(D_\nn(\chi_{\Omega_{i,n}})\otimes I_t)+(D_\nn(\chi_{\Omega_i})\otimes I_s)A_n(D_\nn(\chi_{\Omega_{i,n}}-\chi_{\Omega_i})\otimes I_t)\Bigr].
\end{align*}
Since $\Omega_{i,n}\xrightarrow{{\rm w.r.t.}\,\nn}\Omega_i$, Remark~\ref{rankD-D} implies that
\begin{equation}\label{rank-diag}
{\rm rank}(D_\nn(\chi_{\Omega_{i,n}}-\chi_{\Omega_i}))=o(N(\nn)),\qquad i=1,\ldots,\nu.
\end{equation}
Thus,
\begin{align*}
{\rm rank}(P_n^{BJ}-\tilde P_n^{BJ})&\le\sum_{i=1}^\nu\Bigl[s\mathop{\rm rank}(D_\nn(\chi_{\Omega_{i,n}}-\chi_{\Omega_i}))+t\mathop{\rm rank}(D_\nn(\chi_{\Omega_{i,n}}-\chi_{\Omega_i}))\Bigr]\\
&=(s+t)\sum_{i=1}^\nu{\rm rank}(D_\nn(\chi_{\Omega_{i,n}}-\chi_{\Omega_i}))=o(N(\nn)).
\end{align*}
This shows that \eqref{P-P=0} holds, and the proof of \eqref{clueq} is complete.

\medskip

\noindent{\em Proof of \eqref{clueq'}.}
The proof of \eqref{clueq'} follows the same pattern as the proof of \eqref{clueq}. Let us define, for every $n$,
\begin{align*}
P_n^{BGS}&=\mathop{\rm blocktril}_{\Omega_{1,n},\ldots,\Omega_{\nu,n}}^{\nn,s,t}(A_n)=\sum_{1\le j\le i\le\nu}Z_{\Omega_{i,n},\Omega_{j,n}}^{\nn,s,t}(A_n)=\sum_{1\le j\le i\le\nu}D_\nn(\chi_{\Omega_{i,n}}I_s)A_nD_\nn(\chi_{\Omega_{j,n}}I_t),\\
\tilde P_n^{BGS}&=\sum_{1\le j\le i\le\nu}D_\nn(\chi_{\Omega_i}I_s)A_nD_\nn(\chi_{\Omega_j}I_t).
\end{align*}
Note that $\tilde P_n^{BGS}$ is obtained from $P_n^{BGS}$ by replacing the sets $\Omega_{1,n},\ldots,\Omega_{\nu,n}$ with their limits $\Omega_1,\ldots,\Omega_\nu$. Since $\Omega_1,\ldots,\Omega_\nu$ are regular by assumption, their characteristic functions $\chi_{\Omega_1},\ldots,\chi_{\Omega_\nu}$ are continuous a.e.\ on $[0,1]^d$ and {\bf GLT2}--{\bf GLT3} imply that
\[ \{\tilde P_n^{BGS}\}_n\sim_{\rm GLT}\sum_{1\le j\le i\le\nu}\chi_{\Omega_i}(\xx)I_s\,\kappa(\xx,\btheta)\,\chi_{\Omega_j}(\xx)I_t=\sum_{1\le j\le i\le\nu}\chi_{\Omega_i\cap\Omega_j}(\xx)\kappa(\xx,\btheta)=\kappa^{BGS}(\xx,\btheta). \]
We show that
\begin{equation}\label{P-P=0'}
\{P_n^{BGS}-\tilde P_n^{BGS}\}_n\sim_\sigma0.
\end{equation}
Once this is done, the thesis \eqref{clueq'} follows from the GLT relation $\{\tilde P_n^{BGS}\}_n\sim_{\rm GLT}\kappa^{BGS}$ and {\bf GLT2}--{\bf GLT3}. To prove \eqref{P-P=0'}, let us write
\begin{align*}
&P_n^{BGS}-\tilde P_n^{BGS}=\sum_{1\le j\le i\le\nu}\Bigl[D_\nn(\chi_{\Omega_{i,n}}I_s)A_nD_\nn(\chi_{\Omega_{j,n}}I_t)-D_\nn(\chi_{\Omega_i}I_s)A_nD_\nn(\chi_{\Omega_j}I_t)\Bigr]\\
&=\sum_{1\le j\le i\le\nu}\Bigl[(D_\nn(\chi_{\Omega_{i,n}}I_s)-D_\nn(\chi_{\Omega_i}I_s))A_nD_\nn(\chi_{\Omega_{j,n}}I_t)+D_\nn(\chi_{\Omega_i}I_s)A_n(D_\nn(\chi_{\Omega_{j,n}}I_t)-D_\nn(\chi_{\Omega_j}I_t))\Bigr]\\
&=\sum_{1\le j\le i\le\nu}\Bigl[D_\nn((\chi_{\Omega_{i,n}}-\chi_{\Omega_i})I_s)A_nD_\nn(\chi_{\Omega_{j,n}}I_t)+D_\nn(\chi_{\Omega_i}I_s)A_nD_\nn((\chi_{\Omega_{j,n}}-\chi_{\Omega_j})I_t)\Bigr]\\
&=\sum_{1\le j\le i\le\nu}\Bigl[(D_\nn(\chi_{\Omega_{i,n}}-\chi_{\Omega_i})\otimes I_s)A_n(D_\nn(\chi_{\Omega_{j,n}})\otimes I_t)+(D_\nn(\chi_{\Omega_i})\otimes I_s)A_n(D_\nn(\chi_{\Omega_{j,n}}-\chi_{\Omega_j})\otimes I_t)\Bigr].
\end{align*}
Recalling \eqref{rank-diag}, we have
\begin{align*}
{\rm rank}(P_n^{BGS}-\tilde P_n^{BGS})&\le\sum_{1\le j\le i\le\nu}\Bigl[s\mathop{\rm rank}(D_\nn(\chi_{\Omega_{i,n}}-\chi_{\Omega_i}))+t\mathop{\rm rank}(D_\nn(\chi_{\Omega_{j,n}}-\chi_{\Omega_j}))\Bigr]\\
&\le(s\vee t)\sum_{1\le j\le i\le\nu}\Bigl[{\rm rank}(D_\nn(\chi_{\Omega_{i,n}}-\chi_{\Omega_i}))+{\rm rank}(D_\nn(\chi_{\Omega_{j,n}}-\chi_{\Omega_j}))\Bigr]=o(N(\nn)).
\end{align*}
This shows that \eqref{P-P=0'} holds, and the proof of \eqref{clueq'} is complete.
\end{proof}

We present two corollaries of Theorem~\ref{blockdiag(GLT)=GLT_blocktril(GLT)=GLT}.
In what follows, a regular partition of $[0,1]^d$ is a partition of $[0,1]^d$ consisting of regular sets.

\begin{corollary}\label{block-cor}
Let $d,s,t,\nu$ be positive integers, let $\{\nn=\nn(n)\}_n$ be a sequence of positive $d$-indices tending to $\infty$, and, for every~$n$, let $\Omega_{1,n},\ldots,\Omega_{\nu,n}$ be subsets of $[0,1]^d$ with the following property.
\begin{quote}
For every subsequence of indices $n\in\II$ there exist a subsequence of indices $n\in\JJ\subseteq\II$ and a regular partition $\{\Omega_1,\ldots,\Omega_\nu\}$ of $[0,1]^d$, which may depend on $\JJ$, such that $\Omega_{i,n}\xrightarrow{{\rm w.r.t.}\,\nn}\Omega_i$ as $n\to\infty$ in $\JJ$, for every $i=1,\ldots,\nu$.
\end{quote}
Let $\kappa:[0,1]^d\times[-\pi,\pi]^d\to\mathbb C^{s\times t}$ be measurable and let $\{A_n\}_n$ be a sequence of matrices with $A_n$ of size $N(\nn)s\times N(\nn)t$ and $\{A_n\}_n\sim_{\rm GLT}\kappa$. Then,
\begin{align}
\biggl\{\mathop{\rm blockdiag}_{\Omega_{1,n},\ldots,\Omega_{\nu,n}}^{\nn,s,t}(A_n)\biggr\}_n\sim_{\rm GLT}\kappa,\label{clueq-cor}\\
\biggl\{\mathop{\rm blocktril}_{\Omega_{1,n},\ldots,\Omega_{\nu,n}}^{\nn,s,t}(A_n)\biggr\}_n\sim_{\rm GLT}\kappa.\label{clueq'-cor}
\end{align}
\end{corollary}
\begin{proof}
We first prove \eqref{clueq-cor}. Let us define, for every $n$,
\begin{align}
P_n&=\mathop{\rm blockdiag}_{\Omega_{1,n},\ldots,\Omega_{\nu,n}}^{\nn,s,t}(A_n),\label{Pndef}\\[5pt]
Z_n&=P_n-A_n.\notag
\end{align}
By {\bf GLT2}--{\bf GLT3}, the thesis $\{P_n\}_n\sim_{\rm GLT}\kappa$ is equivalent to $\{Z_n\}_n\sim_\sigma0$. Suppose by contradiction that the thesis does not hold. Then, $\{Z_n\}_n$ is not a zero-distributed sequence. This means that there exist $F\in C_c(\mathbb R)$ and a subsequence of indices $n\in\II$ such that
\begin{equation}\label{lim-iuper}
\lim_{\substack{\vphantom{\int}n\to\infty\\n\in\II}}\frac1{N(\nn)s\wedge N(\nn)t}\sum_{i=1}^{N(\nn)s\wedge N(\nn)t}F(\sigma_i(Z_n))=L\ne0.
\end{equation}
By hypothesis, we can extract from the subsequence of indices $n\in\II$ another subsequence of indices $n\in\JJ\subseteq\II$ such that $\Omega_{i,n}\xrightarrow{{\rm w.r.t.}\,\nn}\Omega_i$ as $n\to\infty$ in $\JJ$, for every $i=1,\ldots,\nu$, where $\{\Omega_1,\ldots,\Omega_\nu\}$ is some regular partition of $[0,1]^d$.
By Theorem~\ref{blockdiag(GLT)=GLT_blocktril(GLT)=GLT}, we have $\{P_n\}_{n\in\JJ}\sim_{\rm GLT}\kappa$, which is equivalent to $\{Z_n\}_{n\in\JJ}\sim_\sigma0$ by {\bf GLT2}--{\bf GLT3} and the obvious GLT relation $\{A_n\}_{n\in\JJ}\sim_{\rm GLT}\kappa$.
In particular, we have
\[ \lim_{\substack{\vphantom{\int}n\to\infty\\n\in\JJ}}\frac1{N(\nn)s\wedge N(\nn)t}\sum_{i=1}^{N(\nn)s\wedge N(\nn)t}F(\sigma_i(Z_n))=0, \]
which is a contradiction to \eqref{lim-iuper}. This completes the proof of \eqref{clueq-cor}.

The proof of \eqref{clueq'-cor} is verbatim the same as the proof of \eqref{clueq-cor}, with the only difference that ``blockdiag'' must be replaced with ``blocktril'' in \eqref{Pndef}.
\end{proof}

\begin{corollary}\label{block-cor-cor}
Let $d,s,t,\nu,k$ be positive integers with $1\le k\le d$, let $\{\nn=\nn(n)\}_n$ be a sequence of positive $d$-indices tending to $\infty$, and, for every~$n$, let $(n_{k,1},\ldots,n_{k,\nu})=(n_{k,1}(n),\ldots,n_{k,\nu}(n))$ be a partition of $n_k$ (the $k$th component of $\nn$) and let $\{\Omega_{1,n},\ldots,\Omega_{\nu,n}\}$ be the partition of $[0,1]^d$ defined by
\begin{align}
\Omega_{1,n}&=[0,1]^{k-1}\times\Bigl[0,\frac{n_{k,1}}{n_k}\Bigr]\times[0,1]^{d-k},\label{o1}\\
\Omega_{2,n}&=[0,1]^{k-1}\times\Bigl(\frac{n_{k,1}}{n_k},\frac{n_{k,1}+n_{k,2}}{n_k}\Bigr]\times[0,1]^{d-k},\label{o2}\\
\Omega_{3,n}&=[0,1]^{k-1}\times\Bigl(\frac{n_{k,1}+n_{k,2}}{n_k},\frac{n_{k,1}+n_{k,2}+n_{k,3}}{n_k}\Bigr]\times[0,1]^{d-k},\label{o3}\\
\vdots & \notag\\
\Omega_{\nu,n}&=[0,1]^{k-1}\times\Bigl(\frac{n_{k,1}+\ldots+n_{k,\nu-1}}{n_k},1\Bigr]\times[0,1]^{d-k}.\label{onu}
\end{align}
Let $\kappa:[0,1]^d\times[-\pi,\pi]^d\to\mathbb C^{s\times t}$ be measurable and let $\{A_n\}_n$ be a sequence of matrices with $A_n$ of size $N(\nn)s\times N(\nn)t$ and $\{A_n\}_n\sim_{\rm GLT}\kappa$. Then,
\begin{align*}
\Bigl\{P_{\Omega_{1,n},\ldots,\Omega_{\nu,n}}^{BJ,\nn,s,t}(A_n)\Bigr\}_n&=\biggl\{\mathop{\rm blockdiag}_{\Omega_{1,n},\ldots,\Omega_{\nu,n}}^{\nn,s,t}(A_n)\biggr\}_n\sim_{\rm GLT}\kappa,\\
\Bigl\{P_{\Omega_{1,n},\ldots,\Omega_{\nu,n}}^{BGS,\nn,s,t}(A_n)\Bigr\}_n&=\biggl\{\mathop{\rm blocktril}_{\Omega_{1,n},\ldots,\Omega_{\nu,n}}^{\nn,s,t}(A_n)\biggr\}_n\sim_{\rm GLT}\kappa.
\end{align*}
\end{corollary}
\begin{proof}
It is enough to show that the sets $\Omega_{1,n},\ldots,\Omega_{\nu,n}$ in \eqref{o1}--\eqref{onu} satisfy the property quoted in the statement of Corollary~\ref{block-cor}.
Once this is done, the thesis follows from Corollary~\ref{block-cor}.

Consider a subsequence of indices $n\in\II$.
We extract from $\II$ a subsequence of indices $n\in\II_1\subseteq\II$ such that there exists
\[ \lim_{\substack{\vphantom{\int}n\to\infty\\n\in\II_1}}\frac{n_{k,1}}{n_k}=c_1; \]
then, we extract from $\II_1$ a subsequence of indices $n\in\II_2\subseteq\II_1$ such that there exists
\[ \lim_{\substack{\vphantom{\int}n\to\infty\\n\in\II_2}}\frac{n_{k,2}}{n_k}=c_2; \]
and so on. This extraction procedure ends up with a subsequence of indices $n\in\II_\nu\subseteq\ldots\subseteq\II_1\subseteq\II$ such that there exists
\[ \lim_{\substack{\vphantom{\int}n\to\infty\\n\in\II_\nu}}\frac{n_{k,i}}{n_k}=c_i,\qquad i=1,\ldots,\nu. \]
Set $\JJ=\II_\nu$. It is not difficult to see that $\Omega_{i,n}\xrightarrow{{\rm w.r.t.}\,\nn}\Omega_i$ as $n\to\infty$ in $\JJ$, for every $i=1,\ldots,\nu$, where $\{\Omega_1,\ldots,\Omega_\nu\}$ is the regular partition of $[0,1]^d$ defined by\,\footnote{\,Note that $\{\Omega_1,\ldots,\Omega_\nu\}$ is indeed a partition of $[0,1]^d$, because
$\displaystyle c_1+c_2+\ldots+c_\nu=\lim_{\substack{\vphantom{\int}n\to\infty\\n\in J}}\underbrace{\frac{n_{k,1}+n_{k,2}+\ldots+n_{k,\nu}}{n_k}}_{\textup{this equals $1$ for every $n$}}=1$.}
\begin{align*}
\Omega_1&=[0,1]^{k-1}\times[0,c_1]\times[0,1]^{d-k},\\
\Omega_2&=[0,1]^{k-1}\times(c_1,c_1+c_2]\times[0,1]^{d-k},\\
\Omega_3&=[0,1]^{k-1}\times(c_1+c_2,c_1+c_2+c_3]\times[0,1]^{d-k},\\
\vdots & \\
\Omega_\nu&=[0,1]^{k-1}\times(c_1+\ldots+c_{\nu-1},1]\times[0,1]^{d-k}.
\end{align*}
This shows that the sets $\Omega_{1,n},\ldots,\Omega_{\nu,n}$ in \eqref{o1}--\eqref{onu} satisfy the property quoted in the statement of Corollary~\ref{block-cor}, and the proof is complete.
\end{proof}

\begin{remark}
In view of Example~\ref{recupero}, Theorems~4.1--4.2 of \cite{GLH} are obtained from Corollary~\ref{block-cor-cor} by setting $d=1$ and $\nn=d_n$, where $\{d_n\}_n$ is the positive integer sequence tending to $\infty$ that appears in the statements of \cite[Theorems~4.1--4.2]{GLH}.
\end{remark}

\subsection{GLT analysis of the AS/MS preconditioners}\label{sec:AS-MS}

Theorem~\ref{AS(GLT)=GLT_MS(GLT)=GLT} is the second main result of this paper. It can be seen as the multidimensional generalization of \cite[Theorems~4.2--4.3]{DDM-GLT-1d}.
Theorem~\ref{AS(GLT)=GLT_MS(GLT)=GLT} provides the GLT analysis of the AS/MS preconditioners. In particular, it shows the following.
\begin{itemize}[nolistsep,leftmargin=*]
	\item Any sequence of MS preconditioners for a GLT sequence with symbol $\kappa$ is again a GLT sequence with symbol $\kappa$, as long as a suitable convergence assumption is satisfied by the sets $\Omega_{1,n},\ldots,\Omega_{\nu,n}$ used in the construction of the preconditioners.
	\item Any sequence of AS preconditioners for a GLT sequence with symbol $\kappa$ is again a GLT sequence with symbol $\kappa$, as long as a suitable convergence assumption is satisfied by the sets $\Omega_{1,n},\ldots,\Omega_{\nu,n}$ used in the construction of the preconditioners and the limit sets $\Omega_1,\ldots,\Omega_\nu$ form a partition of $[0,1]^d$; if the latter condition is violated, then the symbol of the considered sequence of AS preconditioners is not $\kappa$ but $\kappa(\xx,\btheta)/\sum_{i=1}^\nu\chi_{\Omega_i}(\xx)$.
\end{itemize}
The proof of Theorem~\ref{AS(GLT)=GLT_MS(GLT)=GLT} requires the GLT analysis of the RIE operator, which is provided in Theorem~\ref{rie-thm}.
Due to its relevance, Theorem~\ref{rie-thm} can be considered as a further main result of this paper in addition to Theorems~\ref{blockdiag(GLT)=GLT_blocktril(GLT)=GLT} and~\ref{AS(GLT)=GLT_MS(GLT)=GLT}.

To prove Theorem~\ref{rie-thm}, we need the following Lemma~\ref{opma} about the RIE operator.
In what follows, if $\Omega\subseteq[0,1]^d$, we denote by $\Omega^c=[0,1]^d\setminus\Omega$ the complement of $\Omega$ in $[0,1]^d$.
Moreover, we denote by $GL_m(\mathbb C)$ the general linear group of $m\times m$ invertible matrices:
\[ GL_m(\mathbb C)=\{A\in\mathbb C^{m\times m}:\,A\textup{ is invertible}\}. \]
We recall that, for every positive integer $m$, $GL_m(\mathbb C)$ is a dense open subset of $\mathbb C^{m\times m}$ and the inversion operator
\[ (\cdot)^{-1}:GL_m(\mathbb C)\to GL_m(\mathbb C),\qquad A\mapsto A^{-1}, \]
is a continuous bijection whose inverse is itself.\,\footnote{\,Here and in what follows, it is understood that, for every positive integer $m$, the topology of $\mathbb C^{m\times m}$ is the standard topology induced by any norm of $\mathbb C^{m\times m}$. All terminology from topology (such as ``dense'', ``open'', ``continuous'', etc.)\ always refers to such topology.}

\begin{lemma}\label{opma}
Let $d,s$ be positive integers, let $\nn\in\mathbb N^d$, and let $\Omega\subseteq[0,1]^d$.
\begin{enumerate}[nolistsep,leftmargin=*]
	\item The domain of the RIE operator $D_\Omega^{\rie,\nn,s}$ is an open subset of $\mathbb C^{N(\nn)s\times N(\nn)s}$.
	\item The RIE operator
	\[ \rie_\Omega^{\nn,s}:D_\Omega^{\rie,\nn,s}\subseteq\mathbb C^{N(\nn)s\times N(\nn)s}\to\mathbb C^{N(\nn)s\times N(\nn)s} \]
	is a continuous function over its domain $D_\Omega^{\rie,\nn,s}$.
	\item For every $A\in D_\Omega^{\rie,\nn,s}$ and every $\epsilon>0$, there exists $\tilde A\in D_\Omega^{\rie,\nn,s}\cap D_{\Omega^c}^{\rie,\nn,s}$ such that
	\[ \|\tilde A-A\|<\epsilon,\qquad\|\rie_\Omega^{\nn,s}(\tilde A)-\rie_\Omega^{\nn,s}(A)\|<\epsilon. \]
\end{enumerate}
\end{lemma}
\begin{proof}
1. By definition,
\begin{align*}
D_\Omega^{\rie,\nn,s}&=\{A\in\mathbb C^{N(\nn)s\times N(\nn)s}:\,R_{\Omega,\Omega}^{\nn,s,s}(A)\textup{ is invertible}\}\\
&=\{A\in\mathbb C^{N(\nn)s\times N(\nn)s}:\,R_{\Omega,\Omega}^{\nn,s,s}(A)\in GL_{N_\nn^\Omega s}(\mathbb C)\}=(R_{\Omega,\Omega}^{\nn,s,s})^{-1}(GL_{N_\nn^\Omega s}(\mathbb C))
\end{align*}
is the pre-image of $GL_{N_\nn^\Omega s}(\mathbb C)$ through the restriction operator $R_{\Omega,\Omega}^{\nn,s,s}:\mathbb C^{N(\nn)s\times N(\nn)s}\to\mathbb C^{N_\nn^\Omega s\times N_\nn^\Omega s}$. Since the latter operator is linear and continuous, and $GL_{N_\nn^\Omega s}(\mathbb C)$ is an open subset of $\mathbb C^{N_\nn^\Omega s\times N_\nn^\Omega s}$, we infer that $D_\Omega^{\rie,\nn,s}$ is an open subset of $\mathbb C^{N(\nn)s\times N(\nn)s}$.

2. By definition, $\rie_\Omega^{\nn,s}$ is the composition of three functions:
\begin{alignat*}{7}
D_\Omega^{\rie,\nn,s}&\xrightarrow{\mbox{\footnotesize$R_{\Omega,\Omega}^{\nn,s,s}$}}GL_{N_\nn^\Omega s}(\mathbb C)&&\xrightarrow{\mbox{\footnotesize$(\cdot)^{-1}$}}GL_{N_\nn^\Omega s}(\mathbb C)&&\xrightarrow{\mbox{\footnotesize$E_{\Omega,\Omega}^{\nn,s,s}$}}\mathbb C^{N(\nn)s\times N(\nn)s}\\
A&\xmapsto{\hphantom{R_{\Omega,\Omega}^{\nn,s,s}}}R_{\Omega,\Omega}^{\nn,s,s}(A)&&\xmapsto{\hphantom{(\cdot)^{-1}}}{R_{\Omega,\Omega}^{\nn,s,s}(A)}^{-1}&&\xmapsto{\hphantom{E_{\Omega,\Omega}^{\nn,s,s}}}E_{\Omega,\Omega}^{\nn,s,s}({R_{\Omega,\Omega}^{\nn,s,s}(A)}^{-1})=\rie_\Omega^{\nn,s}(A)
\end{alignat*}
The first function (the restriction operator) is linear and continuous, the second function (the inversion operator) is continuous, and the third function (the expansion operator) is linear and continuous. Thus, the RIE operator $\rie_\Omega^{\nn,s}$ is continuous as a composition of three continuous functions.

3. Fix $A\in D_\Omega^{\rie,\nn,s}$ and $\epsilon>0$.
Since $\rie_\Omega^{\nn,s}$ is continuous on $D_\Omega^{\rie,\nn,s}$ by item~2, there exists $\delta=\delta_{\epsilon,A}>0$ such that $\|\rie_\Omega^{\nn,s}(B)-\rie_\Omega^{\nn,s}(A)\|<\epsilon$ for every $B\in D_\Omega^{\rie,\nn,s}$ with $\|B-A\|<\delta$.
Since $D_\Omega^{\rie,\nn,s}$ is open by item~1, by decreasing $\delta$ (if necessary), we can assume that the open ball
\[ \mathcal B(A,\delta)=\{B\in\mathbb C^{N(\nn)s\times N(\nn)s}:\,\|B-A\|<\delta\} \]
is contained in $D_\Omega^{\rie,\nn,s}$. Of course, we can also assume that $\delta\le\epsilon$. Thus, every matrix $B\in\mathcal B(A,\delta)$ belongs to $D_\Omega^{\rie,\nn,s}$ and satisfies the conditions
\[ \|B-A\|<\epsilon,\qquad\|\rie_\Omega^{\nn,s}(B)-\rie_\Omega^{\nn,s}(A)\|<\epsilon. \]
We show that $\mathcal B(A,\delta)\cap D_{\Omega^c}^{\rie,\nn,s}\ne\emptyset$. Once this is done, the thesis is proved, because any $\tilde A\in\mathcal B(A,\delta)\cap D_{\Omega^c}^{\rie,\nn,s}$ satisfies all the required conditions. 
An elegant proof of the fact that $\mathcal B(A,\delta)\cap D_{\Omega^c}^{\rie,\nn,s}\ne\emptyset$ can be obtained through the open mapping theorem \cite[p.~99]{Rudinone}, as follows.
The restriction operator $R_{\Omega^c,\Omega^c}^{\nn,s,s}:\mathbb C^{N(\nn)s\times N(\nn)s}\to\mathbb C^{N_\nn^{\Omega^c}s\times N_\nn^{\Omega^c}s}$ is linear and continuous, and it is also surjective by \eqref{RE=I}. Hence, by the open mapping theorem, $R_{\Omega^c,\Omega^c}^{\nn,s,s}$ is an open map, i.e., it maps open sets to open sets. In particular, it maps the open ball $\mathcal B(A,\delta)$ to an open neighborhood $\mathcal U$ of $R_{\Omega^c,\Omega^c}^{\nn,s,s}(A)$. Pick any invertible matrix $C\in\mathcal U$ (which exists because $GL_{N_\nn^{\Omega^c}s}(\mathbb C)$ is dense in $\mathbb C^{N_\nn^{\Omega^c}s\times N_\nn^{\Omega^c}s}$). If $B_C$ is a matrix in $\mathcal B(A,\delta)$ such that $R_{\Omega^c,\Omega^c}^{\nn,s,s}(B_C)=C$, then $R_{\Omega^c,\Omega^c}^{\nn,s,s}(B_C)$ is invertible, i.e., $B_C\in D_{\Omega^c}^{\rie,\nn,s}$. This shows that $\mathcal B(A,\delta)\cap D_{\Omega^c}^{\rie,\nn,s}\ne\emptyset$.
\end{proof}

\begin{theorem}\label{rie-thm}
Let $d,s$ be positive integers, let $\{\nn=\nn(n)\}_n$ be a sequence of positive $d$-indices tending to $\infty$, and, for every~$n$, let $\Omega_n$ be a subset of $[0,1]^d$ such that $\Omega_n\xrightarrow{{\rm w.r.t.}\,\nn}\Omega$, where $\Omega\subseteq[0,1]^d$ is regular.
Let $\kappa:[0,1]^d\times[-\pi,\pi]^d\to\mathbb C^{s\times s}$ be measurable and invertible a.e.\ on $[0,1]^d\times[-\pi,\pi]^d$, and let $\{A_n\}_n$ be a sequence of square matrices with $A_n$ of size $N(\nn)s$ and $\{A_n\}_n\sim_{\rm GLT}\kappa$. Suppose that $A_n\in D_{\Omega_n}^{\rie,\nn,s}$ for every $n$. Then,
\begin{equation}\label{rie-glt}
\{\rie_{\Omega_n}^{\nn,s}(A_n)\}_n\sim_{\rm GLT}\chi_\Omega(\xx)(\kappa(\xx,\btheta))^{-1}.
\end{equation}
\end{theorem}
\begin{proof}
For every $n$, let $\tilde A_n\in D_{\Omega_n}^{\rie,\nn,s}\cap D_{\Omega_n^c}^{\rie,\nn,s}$ be such that
\[ \|\tilde A_n-A_n\|<\frac1n,\qquad\|\rie_{\Omega_n}^{\nn,s}(\tilde A_n)-\rie_{\Omega_n}^{\nn,s}(A_n)\|<\frac1n. \]
Note that such a sequence $\{\tilde A_n\}_n$ exists by Lemma~\ref{opma}.
Note also that $\{\rie_{\Omega_n}^{\nn,s}(\tilde A_n)-\rie_{\Omega_n}^{\nn,s}(A_n)\}_n\sim_\sigma0$ and so, by {\bf GLT2}--{\bf GLT3}, the GLT relation \eqref{rie-glt} is equivalent to
\begin{equation}\label{rie-glt-tilde}
\{\rie_{\Omega_n}^{\nn,s}(\tilde A_n)\}_n\sim_{\rm GLT}\chi_\Omega(\xx)(\kappa(\xx,\btheta))^{-1}.
\end{equation}
We prove \eqref{rie-glt-tilde}. For every $n$, keeping in mind that the matrices $R_{\Omega_n,\Omega_n}^{\nn,s,s}(\tilde A_n)$, $R_{\Omega_n^c,\Omega_n^c}^{\nn,s,s}(\tilde A_n)$ are invertible (because $\tilde A_n\in D_{\Omega_n}^{\rie,\nn,s}\cap D_{\Omega_n^c}^{\rie,\nn,s}$), we have
\begin{align*}
\rie_{\Omega_n}^{\nn,s}(\tilde A_n)&=D_\nn(\chi_{\Omega_n}I_s)\biggl[\mathop{\rm blockdiag}^{\nn,s,s}_{\Omega_n,\Omega_n^c}(\tilde A_n)\biggr]^{-1}\qquad\mbox{\footnotesize(by Theorem~\ref{rie-lemma} applied to $\tilde A_n$ with the partition $\{\Omega_n,\Omega_n^c\}$)}\\
&=D_\nn(\chi_\Omega I_s)\biggl[\mathop{\rm blockdiag}^{\nn,s,s}_{\Omega_n,\Omega_n^c}(\tilde A_n)\biggr]^{-1}+(D_\nn(\chi_{\Omega_n}I_s)-D_\nn(\chi_\Omega I_s))\biggl[\mathop{\rm blockdiag}^{\nn,s,s}_{\Omega_n,\Omega_n^c}(\tilde A_n)\biggr]^{-1}\\
&=D_\nn(\chi_\Omega I_s)\biggl[\mathop{\rm blockdiag}^{\nn,s,s}_{\Omega_n,\Omega_n^c}(\tilde A_n)\biggr]^{-1}+\underbrace{(D_\nn(\chi_{\Omega_n}-\chi_\Omega)\otimes I_s)\biggl[\mathop{\rm blockdiag}^{\nn,s,s}_{\Omega_n,\Omega_n^c}(\tilde A_n)\biggr]^{-1}}_{R_n}.
\end{align*}
Since $\Omega_n\xrightarrow{{\rm w.r.t.}\,\nn}\Omega$, Remark~\ref{rankD-D} implies that
\[ {\rm rank}(D_\nn(\chi_{\Omega_n}-\chi_\Omega))=o(N(\nn)). \]
Hence,
\[ {\rm rank}(R_n)\le s\mathop{\rm rank}(D_\nn(\chi_{\Omega_n}-\chi_\Omega))=o(N(\nn)) \]
and $\{R_n\}_n\sim_\sigma0$. We now apply Theorem~\ref{blockdiag(GLT)=GLT_blocktril(GLT)=GLT} to the GLT sequence $\{\tilde A_n\}_n\sim_{\rm GLT}\kappa$ using the partition $\{\Omega_n,\Omega_n^c\}$, which satisfies $\Omega_n\xrightarrow{{\rm w.r.t.}\,\nn}\Omega$ and $\Omega_n^c\xrightarrow{{\rm w.r.t.}\,\nn}\Omega^c$, where $\{\Omega,\Omega^c\}$ is a regular partition of $[0,1]^d$, because $\Omega$ is regular by assumption and the complement of a regular set is regular as well. In this way, we obtain
\[ \biggl\{\mathop{\rm blockdiag}^{\nn,s,s}_{\Omega_n,\Omega_n^c}(\tilde A_n)\biggr\}_n\sim_{\rm GLT}\kappa. \]
Thus, by {\bf GLT2}--{\bf GLT3} in combination with the assumptions that $\Omega$ is regular (i.e., $\chi_\Omega$ is continuous a.e.\ on $[0,1]^d$) and $\kappa$ is invertible a.e.\ on $[0,1]\times[-\pi,\pi]^d$, we get
\begin{align*}
\{\rie_{\Omega_n}^{\nn,s}(\tilde A_n)\}_n&=\biggl\{D_\nn(\chi_\Omega I_s)\biggl[\mathop{\rm blockdiag}^{\nn,s,s}_{\Omega_n,\Omega_n^c}(\tilde A_n)\biggr]^{-1}+R_n\biggr\}_n\sim_{\rm GLT}\chi_\Omega(\xx)(\kappa(\xx,\btheta))^{-1},
\end{align*}
which completes the proof of \eqref{rie-glt-tilde}.
\end{proof}

In what follows, a regular cover of $[0,1]^d$ is a cover of $[0,1]^d$ consisting of regular sets.

\begin{theorem}\label{AS(GLT)=GLT_MS(GLT)=GLT}
Let $d,s,\nu$ be positive integers, let $\{\nn=\nn(n)\}_n$ be a sequence of positive $d$-indices tending to $\infty$, and, for every~$n$, let $\{\Omega_{1,n},\ldots,\Omega_{\nu,n}\}$ be a cover of $[0,1]^d$ such that $\Omega_{i,n}\xrightarrow{{\rm w.r.t.}\,\nn}\Omega_i$ for every $i=1,\ldots,\nu$, where $\{\Omega_1,\ldots,\Omega_\nu\}$ is a regular cover of $[0,1]^d$. Let $\kappa:[0,1]^d\times[-\pi,\pi]^d\to\mathbb C^{s\times s}$ be measurable and invertible a.e.\ on $[0,1]^d\times[-\pi,\pi]^d$, and let $\{A_n\}_n$ be a sequence of square matrices with $A_n$ of size $N(\nn)s$ and $\{A_n\}_n\sim_{\rm GLT}\kappa$. Then, the following properties hold.
\begin{enumerate}[nolistsep,leftmargin=*]
	\item Suppose that $A_n\in D^{AS,\nn,s}_{\Omega_{1,n},\ldots,\Omega_{\nu,n}}$ for every $n$. Then,
	\begin{equation}\label{clueq-ams}
	\Bigl\{P_{\Omega_{1,n},\ldots,\Omega_{\nu,n}}^{AS,\nn,s}(A_n)\Bigr\}_n\sim_{\rm GLT}\frac{\kappa(\xx,\btheta)}{\sum_{i=1}^\nu\chi_{\Omega_i}(\xx)}.
	\end{equation}
	In particular, if $\sum_{i=1}^\nu\chi_{\Omega_i}=1$ a.e.\ on $[0,1]^d$, which happens, for instance, if $\{\Omega_1,\ldots,\Omega_\nu\}$ is a partition of $[0,1]^d$, then
	\begin{equation}\label{in-pirtico-ams}
	\Bigl\{P_{\Omega_{1,n},\ldots,\Omega_{\nu,n}}^{AS,\nn,s}(A_n)\Bigr\}_n\sim_{\rm GLT}\kappa.
	\end{equation}
	\item Suppose that $A_n\in D^{MS,\nn,s}_{\Omega_{1,n},\ldots,\Omega_{\nu,n}}$ for every $n$. Then,
	\begin{equation}\label{clueq'-ams}
	\Bigl\{P_{\Omega_{1,n},\ldots,\Omega_{\nu,n}}^{MS,\nn,s}(A_n)\Bigr\}_n\sim_{\rm GLT}\kappa.
	\end{equation}
\end{enumerate}
\end{theorem}
\begin{proof}
If $\sum_{i=1}^\nu\chi_{\Omega_i}=1$ a.e.\ on $[0,1]^d$, then \eqref{clueq-ams} implies \eqref{in-pirtico-ams}.
Thus, we only have to prove \eqref{clueq-ams} and \eqref{clueq'-ams}.

\medskip

\noindent{\em Proof of \eqref{clueq-ams}.} By Definition~\ref{amSp},
\[ P^{AS,\nn,s}_{\Omega_{1,n},\ldots,\Omega_{\nu,n}}(A_n)=\Biggl[\sum_{i=1}^\nu\rie_{\Omega_{i,n}}^{\nn,s}(A_n)\Biggr]^{-1}. \]
By Theorem~\ref{rie-thm} and {\bf GLT3},
\[ \Biggl\{\sum_{i=1}^\nu\rie_{\Omega_{i,n}}^{\nn,s}(A_n)\Biggr\}_n\sim_{\rm GLT}\sum_{i=1}^\nu\chi_{\Omega_i}(\xx)(\kappa(\xx,\btheta))^{-1}. \]
The thesis \eqref{clueq-ams} now follows from {\bf GLT3} and the assumption that $\{\Omega_1,\ldots,\Omega_\nu\}$ is a cover of $[0,1]^d$, which implies that $\sum_{i=1}^\nu\chi_{\Omega_i}(\xx)\ne0$ for every $\xx\in[0,1]^d$.

\medskip

\noindent{\em Proof of \eqref{clueq'-ams}.} By Definition~\ref{amSp},
\[ P^{MS,\nn,s}_{\Omega_{1,n},\ldots,\Omega_{\nu,n}}(A_n)=A_n\Biggl[I-\prod_{i=\nu}^1\Bigl(I-\rie_{\Omega_{i,n}}^{\nn,s}(A_n)A_n\Bigr)\Biggr]^{-1}. \]
By Theorem~\ref{rie-thm} and {\bf GLT2}--{\bf GLT3},\,\footnote{\,Note that the identity matrix $I=I_{N(\nn)s}$ satisfies $\{I_{N(\nn)s}\}_n\sim_{\rm GLT}I_s$ by {\bf GLT2} since $I_{N(\nn)s}=D_\nn(I_s)=T_\nn(I_s)$.}
\begin{align*}
\Biggl\{I-\prod_{i=\nu}^1\Bigl(I-\rie_{\Omega_{i,n}}^{\nn,s}(A_n)A_n\Bigr)\Biggr\}_n&\sim_{\rm GLT}I_s-\prod_{i=\nu}^1\Bigl(I_s-\chi_{\Omega_i}(\xx)(\kappa(\xx,\btheta))^{-1}\kappa(\xx,\btheta)\Bigr)\\
&\hphantom{{}\sim_{\rm GLT}{}}=I_s-\prod_{i=\nu}^1(I_s-\chi_{\Omega_i}(\xx)I_s)=\biggl[1-\prod_{i=\nu}^1(1-\chi_{\Omega_i}(\xx))\biggr]I_s=I_s,
\end{align*}
where the last equality follows from the fact that $\{\Omega_1,\ldots,\Omega_\nu\}$ is a cover of $[0,1]^d$ by assumption, and so, for every $\xx\in[0,1]^d$, we have $1-\chi_{\Omega_i}(\xx)=0$ for at least one index $i\in\{1,\ldots,\nu\}$ and $\prod_{i=\nu}^1(1-\chi_{\Omega_i}(\xx))=0$.
The thesis \eqref{clueq'-ams} now follows from {\bf GLT3}.
\end{proof}

\begin{remark}\label{few_mesh_sizes=small-o_of_matrix-size}
When applying Theorem~\ref{AS(GLT)=GLT_MS(GLT)=GLT} in the context of DDMs, the limit sets $\Omega_1,\ldots,\Omega_\nu$ often form a partition of $[0,1]^d$, because the overlaps in the sets $\Omega_{1,n},\ldots,\Omega_{\nu,n}$, if any, asymptotically vanish as $n\to\infty$; 
see Remark~\ref{overlap-cases}.
Thus, the GLT relations \eqref{in-pirtico-ams}--\eqref{clueq'-ams} hold. There are important cases, however, where the limit sets $\Omega_1,\ldots,\Omega_\nu$ overlap; see again Remark~\ref{overlap-cases}.
In that cases, while the sequence of MS preconditioners $\{P_{\Omega_{1,n},\ldots,\Omega_{\nu,n}}^{MS,\nn,s}(A_n)\}_n$ continue to satisfy the GLT relation \eqref{clueq'-ams}, the sequence of AS preconditioners $\{P_{\Omega_{1,n},\ldots,\Omega_{\nu,n}}^{AS,\nn,s}(A_n)\}_n$ satisfies the GLT relation \eqref{clueq-ams} and, in particular, it does not share the same symbol as the original sequence $\{A_n\}_n$.
\end{remark}

Two corollaries of Theorem~\ref{AS(GLT)=GLT_MS(GLT)=GLT}, completely analogous to Corollaries~\ref{block-cor}--\ref{block-cor-cor}, are reported below.

\begin{corollary}\label{block-cor'}
Let $d,s,\nu$ be positive integers, let $\{\nn=\nn(n)\}_n$ be a sequence of positive $d$-indices tending to $\infty$, and, for every~$n$, let $\{\Omega_{1,n},\ldots,\Omega_{\nu,n}\}$ be a cover of $[0,1]^d$ with the following property.
\begin{quote}
For every subsequence of indices $n\in\II$ there exist a subsequence of indices $n\in\JJ\subseteq\II$ and a regular partition $\{\Omega_1,\ldots,\Omega_\nu\}$ of $[0,1]^d$, which may depend on $\JJ$, such that $\Omega_{i,n}\xrightarrow{{\rm w.r.t.}\,\nn}\Omega_i$ as $n\to\infty$ in $\JJ$, for every $i=1,\ldots,\nu$.
\end{quote}
Let $\kappa:[0,1]^d\times[-\pi,\pi]^d\to\mathbb C^{s\times s}$ be measurable and invertible a.e.\ on $[0,1]^d\times[-\pi,\pi]^d$, and let $\{A_n\}_n$ be a sequence of square matrices with $A_n$ of size $N(\nn)s$ and $\{A_n\}_n\sim_{\rm GLT}\kappa$. Then, the following implications hold:
\begin{alignat}{3}
&A_n\in D^{AS,\nn,s}_{\Omega_{1,n},\ldots,\Omega_{\nu,n}}\mbox{ for every }n&&\qquad\implies\qquad\Bigl\{P_{\Omega_{1,n},\ldots,\Omega_{\nu,n}}^{AS,\nn,s}(A_n)\Bigr\}_n\sim_{\rm GLT}\kappa,\label{clueq-cor'}\\
&A_n\in D^{MS,\nn,s}_{\Omega_{1,n},\ldots,\Omega_{\nu,n}}\mbox{ for every }n&&\qquad\implies\qquad\Bigl\{P_{\Omega_{1,n},\ldots,\Omega_{\nu,n}}^{MS,\nn,s}(A_n)\Bigr\}_n\sim_{\rm GLT}\kappa.\label{clueq'-cor'}
\end{alignat}
\end{corollary}
\begin{proof}
The proof of \eqref{clueq-cor'} (respectively, \eqref{clueq'-cor'}) is verbatim the same as the proof of \eqref{clueq-cor} in Corollary~\ref{block-cor}, with the only difference that ``$\mathop{\rm blockdiag}_{\Omega_{1,n},\ldots,\Omega_{\nu,n}}^{\nn,s,t}(A_n)$'' must be replaced in \eqref{Pndef} with ``$P_{\Omega_{1,n},\ldots,\Omega_{\nu,n}}^{AS,\nn,s}(A_n)$'' (respectively, ``$P_{\Omega_{1,n},\ldots,\Omega_{\nu,n}}^{MS,\nn,s}(A_n)$'') and ``Theorem~\ref{blockdiag(GLT)=GLT_blocktril(GLT)=GLT}'' must be replaced with ``Theorem~\ref{AS(GLT)=GLT_MS(GLT)=GLT}''.
\end{proof}

\begin{corollary}\label{block-cor-cor'}
Let $d,s,\nu,k$ be positive integers with $1\le k\le d$, let $\{\nn=\nn(n)\}_n$ be a sequence of positive $d$-indices tending to $\infty$, and, for every~$n$, let $(n_{k,1},\ldots,n_{k,\nu})=(n_{k,1}(n),\ldots,n_{k,\nu}(n))$ be a partition of $n_k$ (the $k$th component of $\nn$) and let $\{\Omega_{1,n},\ldots,\Omega_{\nu,n}\}$ be the partition of $[0,1]^d$ defined by \eqref{o1}--\eqref{onu}.
Let $\kappa:[0,1]^d\times[-\pi,\pi]^d\to\mathbb C^{s\times s}$ be measurable 
and let $\{A_n\}_n$ be a sequence of square matrices with $A_n$ of size $N(\nn)s$ and $\{A_n\}_n\sim_{\rm GLT}\kappa$. Then, the following implications hold:
\begin{alignat*}{3}
&A_n\in D^{AS,\nn,s}_{\Omega_{1,n},\ldots,\Omega_{\nu,n}}\mbox{ for every }n&&\qquad\implies\qquad\Bigl\{P_{\Omega_{1,n},\ldots,\Omega_{\nu,n}}^{AS,\nn,s}(A_n)\Bigr\}_n\sim_{\rm GLT}\kappa,\\
&A_n\in D^{MS,\nn,s}_{\Omega_{1,n},\ldots,\Omega_{\nu,n}}\mbox{ for every }n&&\qquad\implies\qquad\Bigl\{P_{\Omega_{1,n},\ldots,\Omega_{\nu,n}}^{MS,\nn,s}(A_n)\Bigr\}_n\sim_{\rm GLT}\kappa.
\end{alignat*}
\end{corollary}
\begin{proof}
We give two proofs of this corollary. The first one requires the additional assumption that $\kappa$ is invertible a.e.\ on $[0,1]^d\times[-\pi,\pi]^d$ in order to apply Corollary~\ref{block-cor'}. The second one does not require additional assumptions.

\medskip

\noindent{\em First proof.} It is verbatim the same as the proof of Corollary~\ref{block-cor-cor} with the only difference that each occurrence of ``Corollary~\ref{block-cor}'' must be replaced with ``Corollary~\ref{block-cor'}''.

\medskip

\noindent{\em Second proof.} Since $\{\Omega_{1,n},\ldots,\Omega_{\nu,n}\}$ is a partition of $[0,1]^d$ for every $n$, Theorem~\ref{S->JGS} implies that, for every $n$,
\begin{alignat*}{3}
P_{\Omega_{1,n},\ldots,\Omega_{\nu,n}}^{AS,\nn,s}(A)&=P_{\Omega_{1,n},\ldots,\Omega_{\nu,n}}^{BJ,\nn,s,s}(A),&\qquad A&\in D^{AS,\nn,s}_{\Omega_{1,n},\ldots,\Omega_{\nu,n}},\\[5pt]
P_{\Omega_{1,n},\ldots,\Omega_{\nu,n}}^{MS,\nn,s}(A)&=P_{\Omega_{1,n},\ldots,\Omega_{\nu,n}}^{BGS,\nn,s,s}(A),&\qquad A&\in D^{MS,\nn,s}_{\Omega_{1,n},\ldots,\Omega_{\nu,n}}.
\end{alignat*}
Thus, the thesis follows immediately from Corollary~\ref{block-cor-cor}.
\end{proof}

\section{Applications to GLT preconditioning}\label{sec:app_discussion}

The idea behind the so-called GLT preconditioning can be described as follows.
Suppose we are given a matrix $A_n$ belonging to a GLT sequence $\{A_n\}_n\sim_{\rm GLT}\kappa$.
For every $n$, choose an invertible matrix $P_n$ of the same size as $A_n$ such that $\{P_n\}_n\sim_{\rm GLT}\xi$, where $\xi$ is assumed to be invertible a.e.
We are interested in analyzing/predicting the performance of $P_n$ as a (Krylov) preconditioner for a linear system with coefficient matrix $A_n$.

By {\bf GLT3}, the sequence of preconditioned matrices $\{P_n^{-1}A_n\}_n$ satisfies
\begin{equation}\label{pica-GLT}
\{P_n^{-1}A_n\}_n\sim_{\rm GLT}\xi^{-1}\kappa.
\end{equation}
By {\bf GLT1}, the GLT relation \eqref{pica-GLT} implies the singular value distribution
\begin{equation}\label{pica-sigma}
\{P_n^{-1}A_n\}_n\sim_\sigma\xi^{-1}\kappa.
\end{equation}
The GLT relation \eqref{pica-GLT} often implies also the spectral distribution
\begin{equation}\label{pica-lambda}
\{P_n^{-1}A_n\}_n\sim_\lambda\xi^{-1}\kappa.
\end{equation}
This happens, for instance, if the matrices $A_n$ are Hermitian and the preconditioners $P_n$ are Hermitian positive definite; see Theorem~\ref{GLT-preconditioning} below.

The spectral distribution \eqref{pica-lambda} implies the cluster of the eigenvalues of $P_n^{-1}A_n$ at the range of $\xi^{-1}\kappa$, which is defined as the union of the ranges of the eigenvalues of $\xi^{-1}\kappa$; see \cite[Sections~2.3.2 and~2.4.2]{bgd} for details.
In particular, if the range of $\xi^{-1}\kappa$ is small and far from $0$, then the eigenvalues of $P_n^{-1}A_n$ are clustered at a small set far from $0$. In this situation, we can predict a good performance of $P_n$ as a preconditioner for $A_n$ on the basis of the convergence properties of Krylov methods such as the conjugate gradient (CG) and the generalized minimal residual (GMRES) methods; see \cite[Section~6.11]{Saad}.

\begin{remark}\label{P-efficiency}
Given a GLT sequence $\{A_n\}_n\sim_{\rm GLT}\kappa$, in this paper we have considered for $\{A_n\}_n$ four different sequences of preconditioners $\{P_n\}_n$, namely, the sequences of BJ/BGS/AS/MS preconditioners. Under the (mild) assumptions of our main Theorems~\ref{blockdiag(GLT)=GLT_blocktril(GLT)=GLT} and~\ref{AS(GLT)=GLT_MS(GLT)=GLT}, we have proved the following.
\begin{itemize}[nolistsep,leftmargin=*]
	\item If $\{P_n\}_n$ denotes either the sequence of BJ preconditioners or the sequence of BGS preconditioners or the sequence of MS preconditioners for $\{A_n\}_n$, then we have $\{P_n\}_n\sim_{\rm GLT}\xi$, where $\xi=\kappa$. In other words, $\{P_n\}_n$ is a GLT sequence with symbol $\kappa$ exactly as $\{A_n\}_n$.
	This is the most ``lucky'' case, in which \eqref{pica-GLT}--\eqref{pica-lambda} become
	\begin{equation*}
	\{P_n^{-1}A_n\}_n\sim_{\rm GLT}I,\qquad\{P_n^{-1}A_n\}_n\sim_\sigma I,\qquad\{P_n^{-1}A_n\}_n\sim_\lambda I.
	\end{equation*}
	The range of the identity symbol $I$ is the same as the range of the identically~$1$ function and reduces to the singleton $\{1\}$. Hence, the spectral distribution $\{P_n^{-1}A_n\}_n\sim_\lambda I$ implies the cluster of the eigenvalues of $P_n^{-1}A_n$ at $1$. Thus, we can predict that $P_n$ is an efficient preconditioner for $A_n$. In conclusion, the BJ/BGS/MS preconditioners are expected to be efficient when applied to a matrix $A_n$ belonging to a GLT sequence.
	\item If $\{P_n\}_n$ denotes the sequence of AS preconditioners for $\{A_n\}_n$, then we have $\{P_n\}_n\sim_{\rm GLT}\xi$, where $\xi(\xx,\btheta)=\kappa(\xx,\btheta)/\sum_{i=1}^\nu\chi_{\Omega_i}(\xx)$ and $\Omega_1,\ldots,\Omega_\nu$ are the limit sets of those used in the construction of $P_n$ (see Theorem~\ref{AS(GLT)=GLT_MS(GLT)=GLT}). In other words, $\{P_n\}_n$ is a GLT sequence with a symbol $\xi$ that is approximately equal to $\kappa$, though not exactly equal to $\kappa$.
	In this case, \eqref{pica-GLT}--\eqref{pica-lambda} become
	\begin{equation*}
	\{P_n^{-1}A_n\}_n\sim_{\rm GLT}\sum_{i=1}^\nu\chi_{\Omega_i}(\xx)I,\qquad\{P_n^{-1}A_n\}_n\sim_\sigma\sum_{i=1}^\nu\chi_{\Omega_i}(\xx)I,\qquad\{P_n^{-1}A_n\}_n\sim_\lambda\sum_{i=1}^\nu\chi_{\Omega_i}(\xx)I.
	\end{equation*}
	The range of the symbol $\sum_{i=1}^\nu\chi_{\Omega_i}(\xx)I$ is the same as the range of the function $\sum_{i=1}^\nu\chi_{\Omega_i}(\xx)$ and reduces to a subset of $\{1,\ldots,\nu\}$.\,\footnote{\,It reduces to $\{1\}$ in the case where $\Omega_1,\ldots,\Omega_\nu$ form a partition of $[0,1]$.}
	Hence, the spectral distribution $\{P_n^{-1}A_n\}_n\sim_\lambda\chi_{\Omega_i}(\xx)I$ implies the cluster of the eigenvalues of $P_n^{-1}A_n$ at a subset of $\{1,\ldots,\nu\}$. Since the latter is a small set far from $0$, we can predict that $P_n$ is an efficient preconditioner for $A_n$. In conclusion, the AS preconditioner is expected to be efficient when applied to a matrix $A_n$ belonging to a GLT sequence.
\end{itemize}
Example~\ref{exa:num} provides a numerical confirmation of the above expectations.
\end{remark}

\begin{theorem}\label{GLT-preconditioning}
Let $d,s$ be positive integers and let $\{\nn=\nn(n)\}_n$ be a sequence of positive $d$-indices tending to $\infty$.
Let $\{A_n\}_n$ be a sequence of Hermitian matrices and let $\{P_n\}_n$ be a sequence of Hermitian positive definite matrices, with $A_n$ and $P_n$ of size $N(\nn)s$.
Suppose that $\{A_n\}_n\sim_{\rm GLT}\kappa$ and $\{P_n\}_n\sim_{\rm GLT}\xi$, where $\kappa,\xi:[0,1]^d\times[-\pi,\pi]^d\to\mathbb C^{s\times s}$ are measurable and $\xi$ is invertible a.e.~on $[0,1]^d\times[-\pi,\pi]^d$. Then, the sequence of preconditioned matrices $P_n^{-1}A_n$ satisfies
\begin{equation*}
\{P_n^{-1}A_n\}_n\sim_{\rm GLT}\xi^{-1}\kappa,\qquad\{P_n^{-1}A_n\}_n\sim_\sigma\xi^{-1}\kappa,\qquad\{P_n^{-1}A_n\}_n\sim_\lambda\xi^{-1}\kappa.
\end{equation*}
\end{theorem}
\begin{proof}
The proof in the case $d=1$ was given in \cite[Section~3]{GLH}. Up to straightforward adaptations, the proof in the case of an arbitrary $d\ge1$ is the same as the proof in the case $d=1$.
\end{proof}

\begin{example}\label{exa:num}
In this example, we illustrate the effectiveness of the BJ/BGS/AS/MS preconditioners for GLT sequences, as predicted by Remark~\ref{P-efficiency}.
The example is organized as follows.

\begin{table}
\footnotesize
\centering
\caption{Example~\ref{exa:num} --- Number of PGMRES iterations with BJ preconditioner $P^{BJ}_{n,p,\alpha,\beta,\nu}$ for solving the linear system \eqref{sistema_normalizzato} up to a precision of $10^{-6}$.}
\label{BJ-table}
\begin{tabular}{rrrrrr}
\toprule
 & & \multicolumn{4}{c}{PGMRES iterations with BJ preconditioner}\\
\cline{3-6}
$n$ & $N_{n,p}$ & $\hphantom{100}\nu=2$ & $\hphantom{100}\nu=4$ & $\hphantom{100}\nu=8$ & $\hphantom{100}\nu=16^{\vphantom{\int^0}}$\\
\midrule
9 & 20 & 2 & 6 & 13 & 13\\
19 & 40 & 2 & 6 & 12 & 21\\
39 & 80 & 2 & 6 & 12 & 20\\
79 & 160 & 2 & 6 & 12 & 22\\
159 & 320 & 2 & 6 & 13 & 23\\
319 & 640 & 2 & 6 & 13 & 24\\
639 & 1280 & 2 & 6 & 13 & 25\\
\bottomrule
\end{tabular}
\vspace{20pt}
\centering
\caption{Example~\ref{exa:num} --- Number of PGMRES iterations with BGS preconditioner $P^{BGS}_{n,p,\alpha,\beta,\nu}$ for solving the linear system \eqref{sistema_normalizzato} up to a precision of $10^{-6}$.}
\label{BGS-table}
\begin{tabular}{rrrrrr}
\toprule
 & & \multicolumn{4}{c}{PGMRES iterations with BGS preconditioner}\\
\cline{3-6}
$n$ & $N_{n,p}$ & $\hphantom{100}\nu=2$ & $\hphantom{100}\nu=4$ & $\hphantom{100}\nu=8$ & $\hphantom{100}\nu=16^{\vphantom{\int^0}}$\\
\midrule
9 & 20 & 2 & 5 & 9 & 10\\
19 & 40 & 2 & 6 & 10 & 16\\
39 & 80 & 2 & 6 & 11 & 19\\
79 & 160 & 2 & 6 & 11 & 21\\
159 & 320 & 2 & 6 & 12 & 23\\
319 & 640 & 2 & 6 & 12 & 24\\
639 & 1280 & 2 & 6 & 12 & 25\\
\bottomrule
\end{tabular}
\vspace{20pt}
\centering
\caption{Example~\ref{exa:num} --- Number of PGMRES iterations with AS preconditioner $P^{AS}_{n,p,\alpha,\beta,v}$ for solving the linear system \eqref{sistema_normalizzato} up to a precision of $10^{-6}$.}
\label{AS-table}
\begin{tabular}{rrrrrr}
\toprule
 & & \multicolumn{4}{c}{PGMRES iterations with AS preconditioner}\\
\cline{3-6}
$n$ & $N_{n,p}$ & $\hphantom{100}v=1$ & $\hphantom{100}v=2$ & $\hphantom{100}v=3$ & $\hphantom{100}v=4^{\vphantom{\int^0}}$\\
\midrule
9 & 20 & 10 & 9 & 9 & 8\\
19 & 40 & 9 & 9 & 8 & 10\\
39 & 80 & 9 & 9 & 8 & 8\\
79 & 160 & 10 & 9 & 8 & 8\\
159 & 320 & 10 & 9 & 9 & 8\\
319 & 640 & 10 & 9 & 9 & 8\\
639 & 1280 & 10 & 9 & 9 & 9\\
\bottomrule
\end{tabular}
\vspace{20pt}
\centering
\caption{Example~\ref{exa:num} --- Number of PGMRES iterations with MS preconditioner $P^{MS}_{n,p,\alpha,\beta,v}$ for solving the linear system \eqref{sistema_normalizzato} up to a precision of $10^{-6}$.}
\label{MS-table}
\begin{tabular}{rrrrrr}
\toprule
 & & \multicolumn{4}{c}{PGMRES iterations with MS preconditioner}\\
\cline{3-6}
$n$ & $N_{n,p}$ & $\hphantom{100}v=1$ & $\hphantom{100}v=2$ & $\hphantom{100}v=3$ & $\hphantom{100}v=4^{\vphantom{\int^0}}$\\
\midrule
9 & 20 & 5 & 4 & 4 & 3\\
19 & 40 & 5 & 5 & 4 & 4\\
39 & 80 & 5 & 5 & 5 & 4\\
79 & 160 & 5 & 5 & 5 & 4\\
159 & 320 & 5 & 5 & 5 & 4\\
319 & 640 & 5 & 5 & 5 & 5\\
639 & 1280 & 5 & 5 & 5 & 5\\
\bottomrule
\end{tabular}
\end{table}

\begin{enumerate}[nolistsep,leftmargin=20pt]
	\item[(A)] We consider the $1$-level GLT sequence $\{n^{-1}B_{n,p,\alpha,\beta}\}_n\sim_{\rm GLT}f_p(\theta)$ of Example~\ref{exa6} for fixed values of $p$, $\alpha$, $\beta$, with $n$ varying in the set of indices $\JJ_{\alpha,\beta}=\{n\in\mathbb N:n\ge\max(1-\alpha,\beta)/(\beta-\alpha)\}$; see Example~\ref{exa6} for the reason behind asking that $n\in\JJ_{\alpha,\beta}$. Recall that $B_{n,p,\alpha,\beta}$ is given by \eqref{mat_rhs} and the size of $B_{n,p,\alpha,\beta}$ is $N_{n,p}=2(n+p-2)$.
	\item[(B)] For every integer $\nu\ge2$, every $p$ and every $n\in\JJ_{\alpha,\beta}$ such that $N_{n,p}\ge\nu$, we define $(N_{n,p,1},\ldots,N_{n,p,\nu})$ as the partition of $N_{n,p}$ given by
	\begin{equation}\label{partitionNnp}
	N_{n,p,i}=\lfloor N_{n,p}/\nu\rfloor,\qquad i=1,\ldots,\nu-1,\qquad N_{n,p,\nu}=N_{n,p}-(\nu-1)\lfloor N_{n,p}/\nu\rfloor.
	\end{equation}
	We consider the BJ/BGS preconditioners for $n^{-1}B_{n,p,\alpha,\beta}$ defined as in Example~\ref{recupero}, equations~\eqref{block-J}--\eqref{block-GS}, with $s=t=1$, $n$ replaced by $N_{n,p}$, and partition of $N_{n,p}$ given by $(N_{n,p,1},\ldots,N_{n,p,\nu})$; we denote them by
	\begin{align*}
	P^{BJ}_{n,p,\alpha,\beta,\nu}:=P^{BJ,N_{n,p},1,1}_{N_{n,p,1},\ldots,N_{n,p,\nu}}(n^{-1}B_{n,p,\alpha,\beta}),\\[5pt]
	P^{BGS}_{n,p,\alpha,\beta,\nu}:=P^{BGS,N_{n,p},1,1}_{N_{n,p,1},\ldots,N_{n,p,\nu}}(n^{-1}B_{n,p,\alpha,\beta}).
	\end{align*}
	Both preconditioners will turn out to be invertible just like $n^{-1}B_{n,p,\alpha,\beta}$.
	\item[(C)] For every integer $v\ge1$, every $p$ and every $n\in\JJ_{\alpha,\beta}$ such that $N_{n,p}\ge4(v+1)$, let $(N_{n,p,1},N_{n,p,2},N_{n,p,3},N_{n,p,4})$ be the partition of $N_{n,p}$ given by \eqref{partitionNnp} for $\nu=4$ and note that $N_{n,p,i}\ge v+1$ for all $i=1,2,3,4$. Define the following cover $\{\Omega_{1,n,p,v},\Omega_{2,n,p,v},\Omega_{3,n,p,v},\Omega_{4,n,p,v}\}$ of $[0,1]$:
	\begin{align*}
	\Omega_{1,n,p,v}&=\left[0,\frac{N_{n,p,1}+v}{N_{n,p}}\right],\\[3pt]
	\Omega_{2,n,p,v}&=\left[\frac{N_{n,p,1}-v}{N_{n,p}},\frac{N_{n,p,1}+N_{n,p,2}+v}{N_{n,p}}\right],\\[3pt]
	\Omega_{3,n,p,v}&=\left[\frac{N_{n,p,1}+N_{n,p,2}-v}{N_{n,p}},\frac{N_{n,p,1}+N_{n,p,2}+N_{n,p,3}+v}{N_{n,p}}\right],\\[3pt]
	\Omega_{4,n,p,v}&=\left[\frac{N_{n,p,1}+N_{n,p,2}+N_{n,p,3}-v}{N_{n,p}},1\right].
	\end{align*}
	By construction, the sets of indices $i\in\{1,\ldots,N_{n,p}\}$ associated with $\Omega_{1,n,p,v},\Omega_{2,n,p,v},\Omega_{3,n,p,v},\Omega_{4,n,p,v}$ are given by
	\begin{align*}
	\II_{N_{n,p}}^{\Omega_{1,n,p,v}}&=\{1,\ldots,N_{n,p,1}+v\},\\[3pt]
	\II_{N_{n,p}}^{\Omega_{2,n,p,v}}&=\{N_{n,p,1}-v,\ldots,N_{n,p,1}+N_{n,p,2}+v\},\\[3pt]
	\II_{N_{n,p}}^{\Omega_{3,n,p,v}}&=\{N_{n,p,1}+N_{n,p,2}-v,\ldots,N_{n,p,1}+N_{n,p,2}+N_{n,p,3}+v\},\\[3pt]
	\II_{N_{n,p}}^{\Omega_{4,n,p,v}}&=\{N_{n,p,1}+N_{n,p,2}+N_{n,p,3}-v,\ldots,N_{n,p}\}.
	\end{align*}
	Note that $v$ determines the size of the overlaps between the sets of the cover $\{\Omega_{1,n,p,v},\Omega_{2,n,p,v},\Omega_{3,n,p,v},\Omega_{4,n,p,v}\}$.
	Note also that the overlaps in the sets $\Omega_{1,n,p,v},\Omega_{2,n,p,v},\Omega_{3,n,p,v},\Omega_{4,n,p,v}$ vanish 
	as $n\to\infty$, and the limiting sets $\Omega_{1,p,v}=[0,\frac14]$, $\Omega_{2,p,v}=(\frac14,\frac12]$, $\Omega_{3,p,v}=(\frac12,\frac34]$, $\Omega_{4,p,v}=(\frac34,1]$ form a partition of $[0,1]$.
	We consider the AS/MS preconditioners for $n^{-1}B_{n,p,\alpha,\beta}$ defined as in Definition~\ref{amSp} with $d=s=1$, $\nu=4$, $\nn=N_{n,p}$, and $\{\Omega_1,\ldots,\Omega_\nu\}=\{\Omega_{1,n,p,v},\Omega_{2,n,p,v},\Omega_{3,n,p,v},\Omega_{4,n,p,v}\}$; we denote them by
	\begin{align*}
	P^{AS}_{n,p,\alpha,\beta,v}:=P^{AS,N_{n,p},1}_{\Omega_{1,n,p,v},\Omega_{2,n,p,v},\Omega_{3,n,p,v},\Omega_{4,n,p,v}}(n^{-1}B_{n,p,\alpha,\beta}),\\[5pt]
	P^{MS}_{n,p,\alpha,\beta,v}:=P^{MS,N_{n,p},1}_{\Omega_{1,n,p,v},\Omega_{2,n,p,v},\Omega_{3,n,p,v},\Omega_{4,n,p,v}}(n^{-1}B_{n,p,\alpha,\beta}).
	\end{align*}
	Both preconditioners will turn out to be invertible just like $n^{-1}B_{n,p,\alpha,\beta}$.
	\item[(D)] We report in Tables~\ref{BJ-table}--\ref{MS-table}, for increasing values of $\nu$, $v$, $n$, the number of iterations needed by the PGMRES method with the above BJ/BGS/AS/MS preconditioners for solving, up to a precision of $10^{-6}$, the normalized $N_{n,p}\times N_{n,p}$ linear system \eqref{the_sys}, i.e.,
	\begin{equation}\label{sistema_normalizzato}
	n^{-1}B_{n,p,\alpha,\beta}\,\boldsymbol u=n^{-1}\boldsymbol f_{n,p,\alpha,\beta},
	\end{equation}
	in the case where $p=3$, $\alpha=\frac25$, $\beta=\frac35$, $f(x)=1$.
	The PGMRES method is started with the zero vector $\mathbf0$ and applied without restarting.
	It is clear from Tables~\ref{BJ-table}--\ref{MS-table} that, for each fixed $\nu$ and $v$, the number of PGMRES iterations is bounded from above by a constant independent of $n$.
	This shows an optimal PGMRES convergence rate and, consequently, the efficiency of the BJ/BGS/AS/MS preconditioners for $n^{-1}B_{n,p,\alpha,\beta}$.
	\item[(E)] In Figures~\ref{BJ-figure}--\ref{MS-figure}, we plot the real parts of the eigenvalues of the preconditioned matrices
	\begin{alignat*}{3}
	&(P^{\hspace{1pt}\mbox{\scriptsize\ding{105}}}_{n,p,\alpha,\beta,\nu})^{-1}(n^{-1}B_{n,p,\alpha,\beta}),&\qquad\mbox{\ding{105}}&\in\{BJ,BGS\},\\[3pt]
	&(P^{\hspace{1pt}\mbox{\scriptsize\ding{105}}}_{n,p,\alpha,\beta,v})^{-1}(n^{-1}B_{n,p,\alpha,\beta}),&\qquad\mbox{\ding{105}}&\in\{AS,MS\},
	\end{alignat*}
	in the case where $p\in\{2,3,4\}$, $n=500$, $\alpha=\frac25$, $\beta=\frac35$, $\nu=4$, $v=2$. The real parts of the eigenvalues are sorted in ascending order and positioned at $i/(N_{n,p}-1)$ for $i=0,\ldots,N_{n,p}-1$. The imaginary parts of the eigenvalues are not shown because they are all very close to zero, with their maximal absolute value being smaller than $3.25\cdot10^{-5}$. We see from Figures~\ref{BJ-figure}--\ref{MS-figure} that, up to a small number of outliers, the eigenvalues of the preconditioned matrices are clustered at $1$, as expected based on Remark~\ref{P-efficiency}.

\end{enumerate}
\begin{figure}
\centering
\begin{subfigure}{0.325\textwidth}
\includegraphics[width=\textwidth]{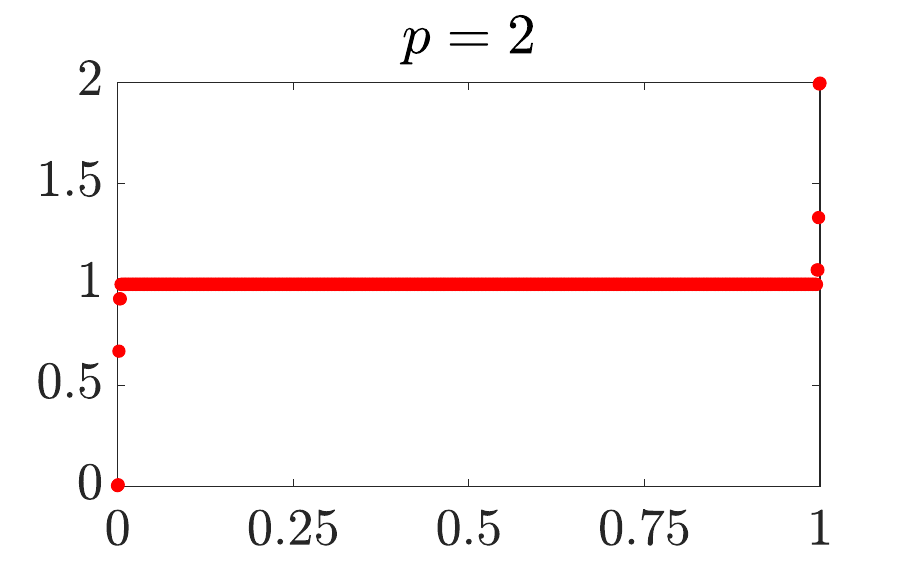}
\end{subfigure}
\begin{subfigure}{0.325\textwidth}
\includegraphics[width=\textwidth]{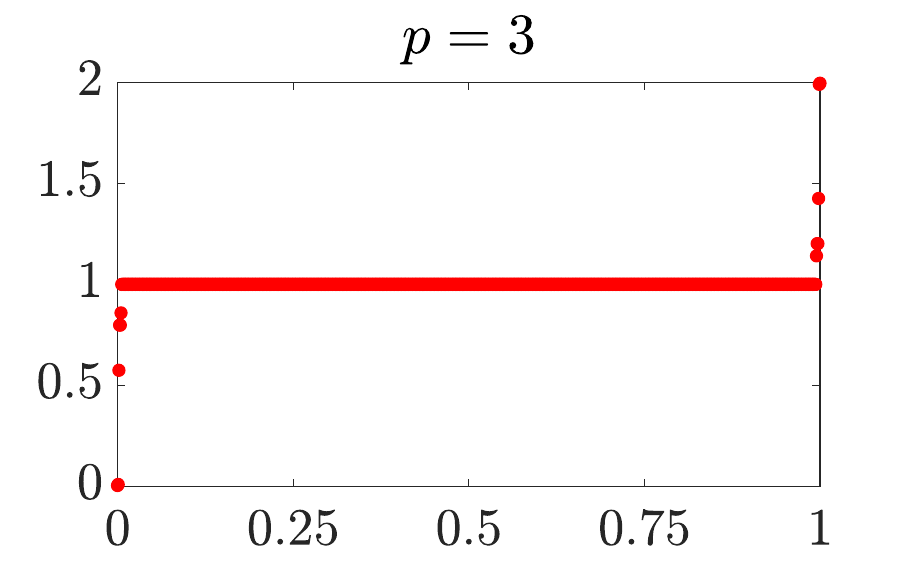}
\end{subfigure}
\begin{subfigure}{0.325\textwidth}
\includegraphics[width=\textwidth]{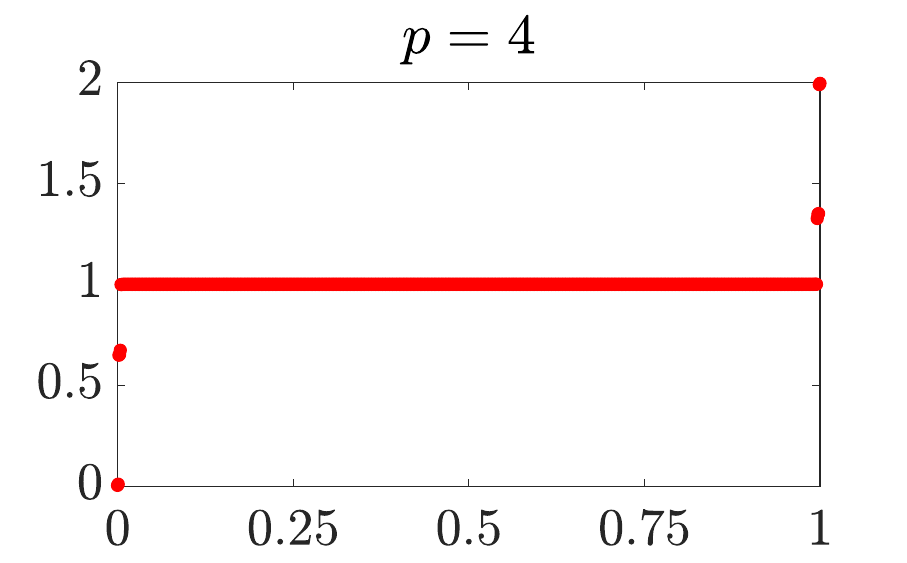}
\end{subfigure}
\caption{Real parts of the eigenvalues of the preconditioned matrix $(P^{BJ}_{n,p,\alpha,\beta,\nu})^{-1}(n^{-1}B_{n,p,\alpha,\beta})$ for $p\in\{2,3,4\}$, $n=500$, $\alpha=\frac25$, $\beta=\frac35$, $\nu=4$.}
\label{BJ-figure}
\vspace{20pt}
\begin{subfigure}{0.325\textwidth}
\includegraphics[width=\textwidth]{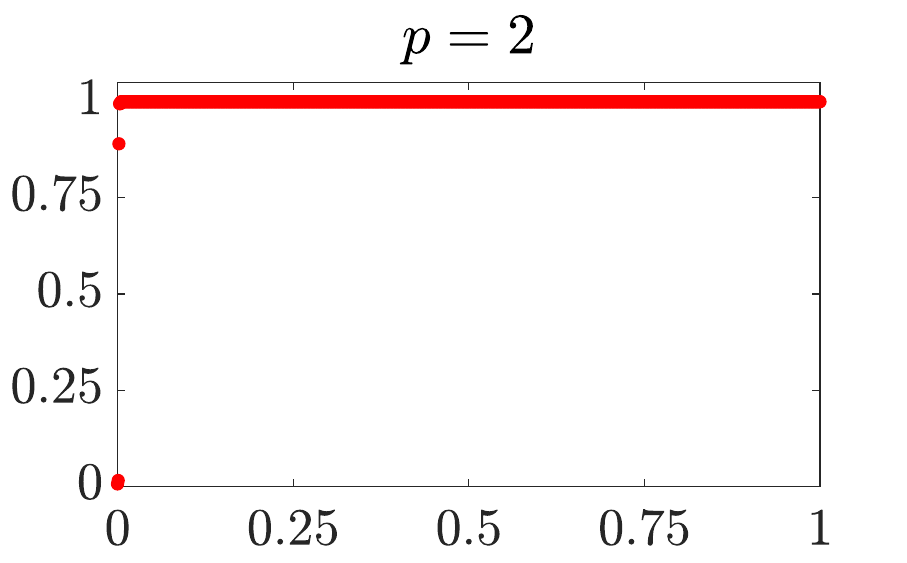}
\end{subfigure}
\begin{subfigure}{0.325\textwidth}
\includegraphics[width=\textwidth]{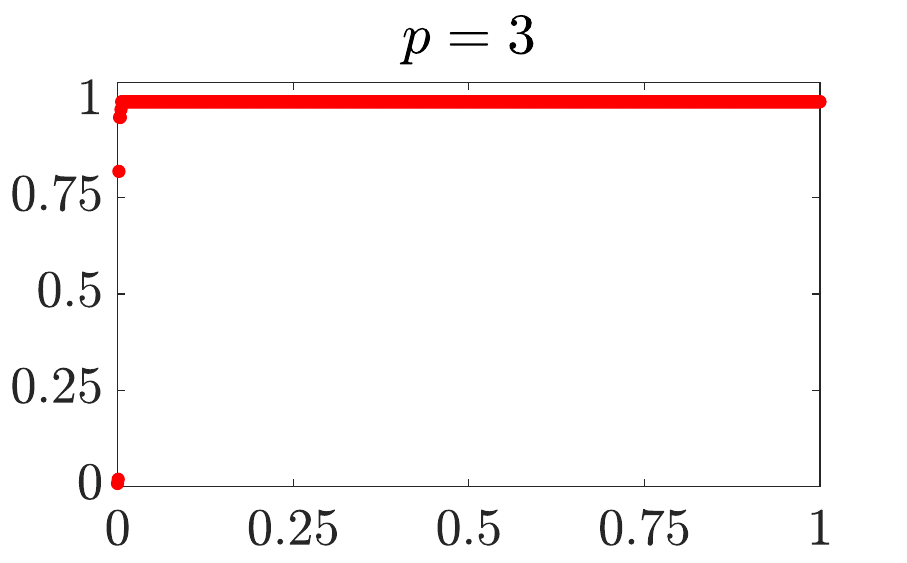}
\end{subfigure}
\begin{subfigure}{0.325\textwidth}
\includegraphics[width=\textwidth]{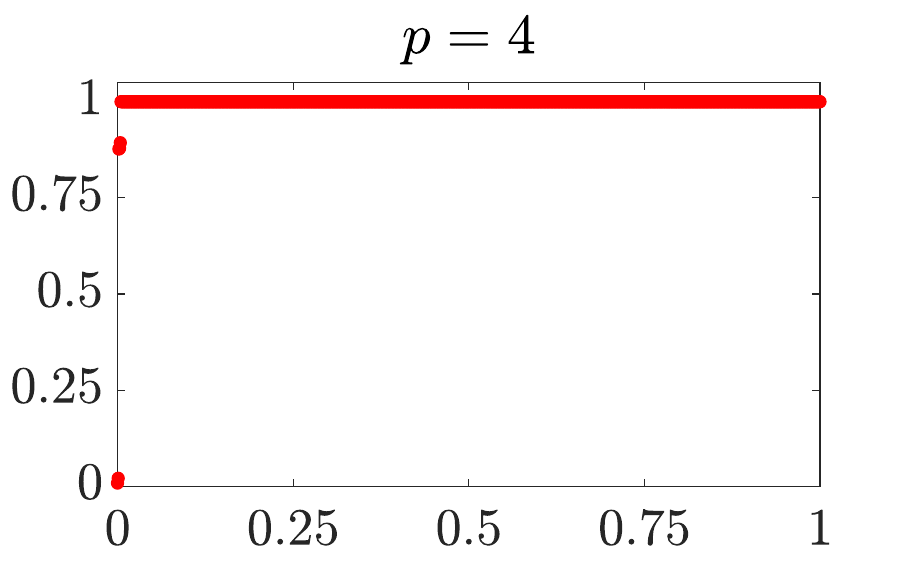}
\end{subfigure}
\caption{Real parts of the eigenvalues of the preconditioned matrix $(P^{BGS}_{n,p,\alpha,\beta,\nu})^{-1}(n^{-1}B_{n,p,\alpha,\beta})$ for $p\in\{2,3,4\}$, $n=500$, $\alpha=\frac25$, $\beta=\frac35$, $\nu=4$.}
\label{BGS-figure}
\vspace{20pt}
\begin{subfigure}{0.325\textwidth}
\includegraphics[width=\textwidth]{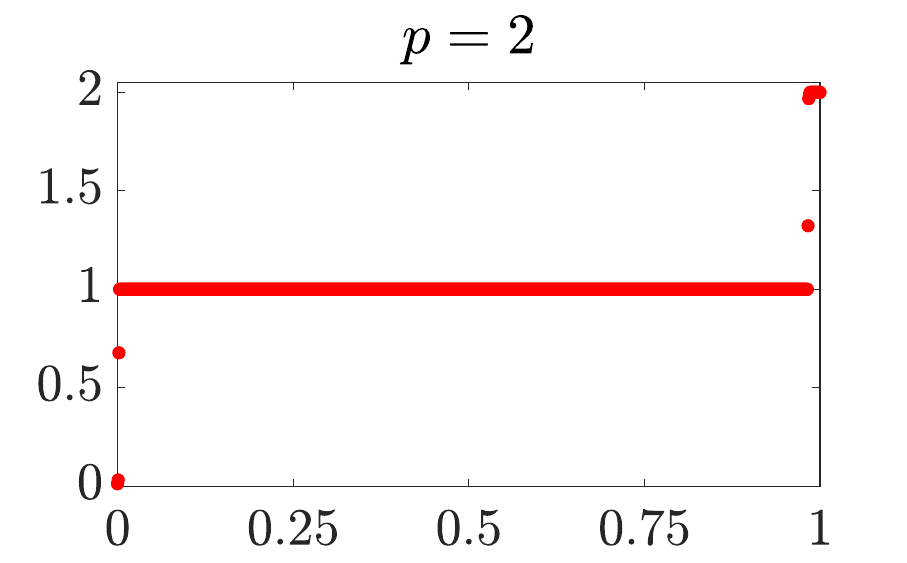}
\end{subfigure}
\begin{subfigure}{0.325\textwidth}
\includegraphics[width=\textwidth]{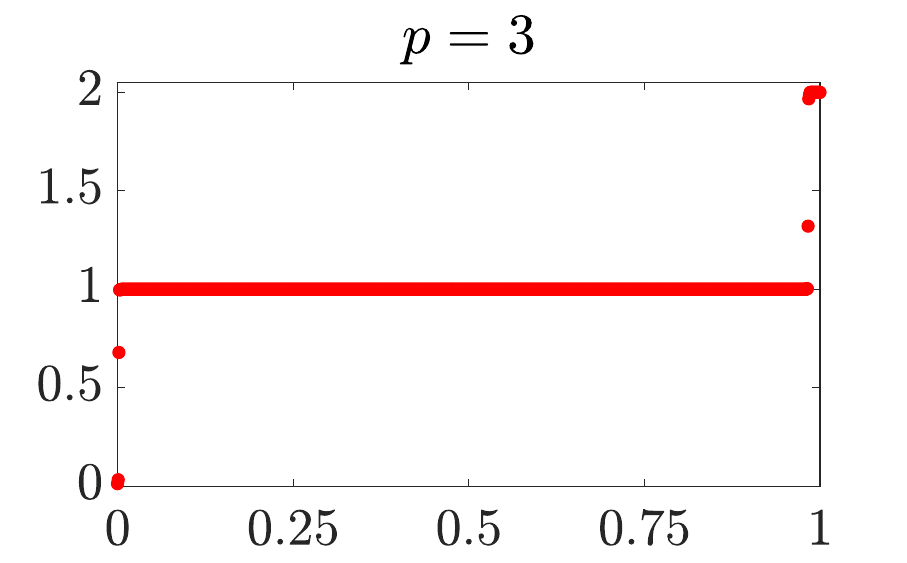}
\end{subfigure}
\begin{subfigure}{0.325\textwidth}
\includegraphics[width=\textwidth]{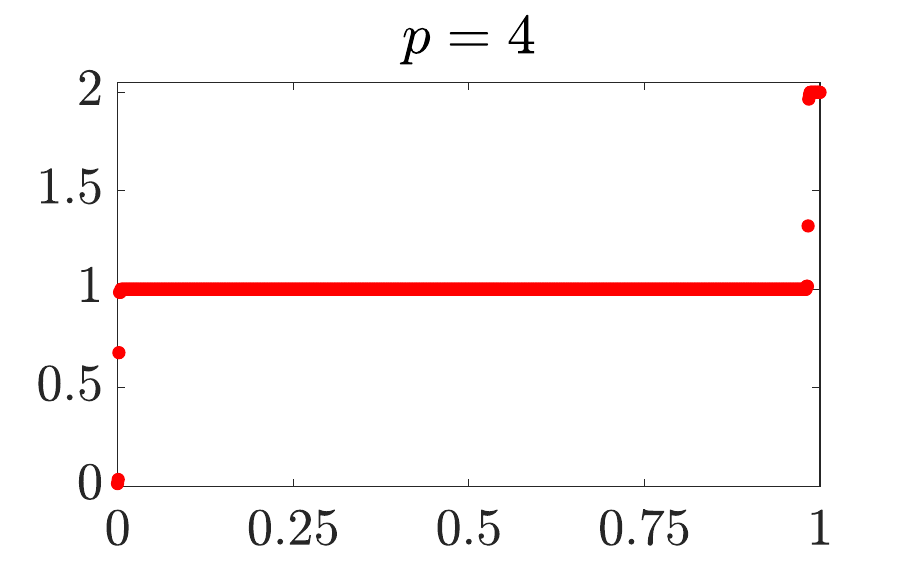}
\end{subfigure}
\caption{Real parts of the eigenvalues of the preconditioned matrix $(P^{AS}_{n,p,\alpha,\beta,v})^{-1}(n^{-1}B_{n,p,\alpha,\beta})$ for $p\in\{2,3,4\}$, $n=500$, $\alpha=\frac25$, $\beta=\frac35$, $v=2$.}
\label{AS-figure}
\vspace{20pt}
\begin{subfigure}{0.325\textwidth}
\includegraphics[width=\textwidth]{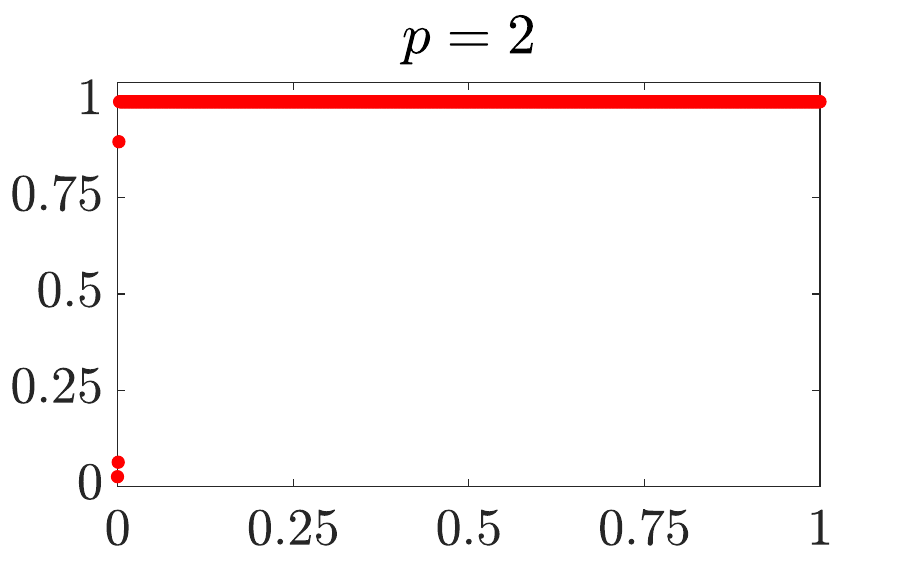}
\end{subfigure}
\begin{subfigure}{0.325\textwidth}
\includegraphics[width=\textwidth]{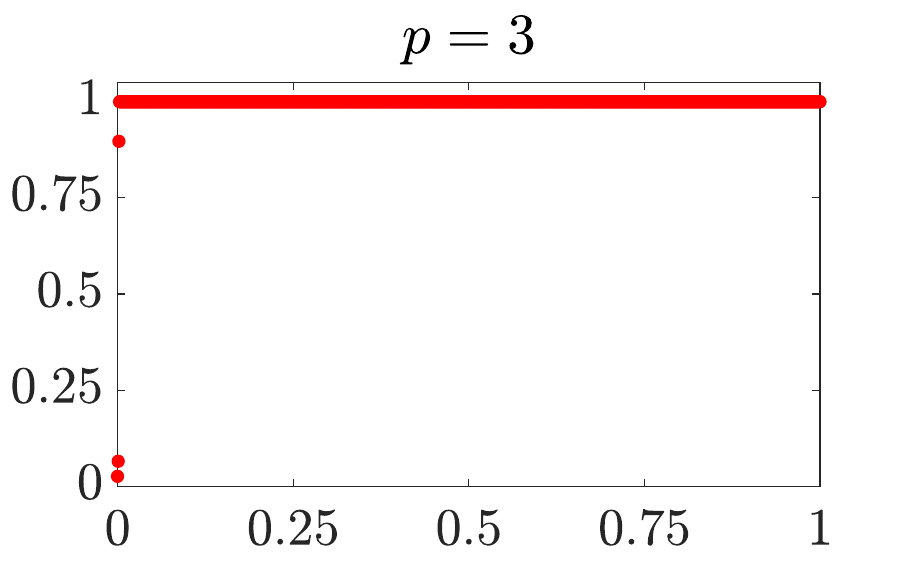}
\end{subfigure}
\begin{subfigure}{0.325\textwidth}
\includegraphics[width=\textwidth]{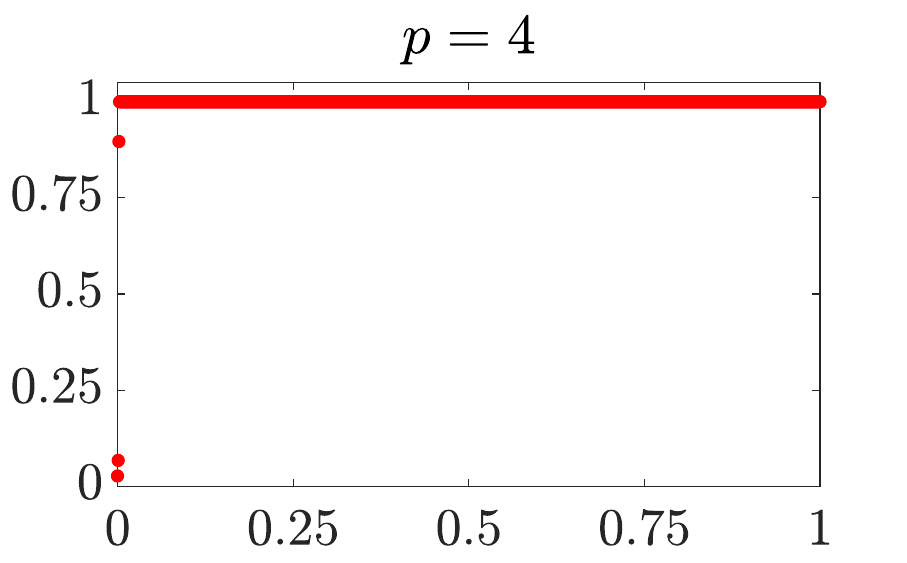}
\end{subfigure}
\caption{Real parts of the eigenvalues of the preconditioned matrix $(P^{MS}_{n,p,\alpha,\beta,v})^{-1}(n^{-1}B_{n,p,\alpha,\beta})$ for $p\in\{2,3,4\}$, $n=500$, $\alpha=\frac25$, $\beta=\frac35$, $v=2$.}
\label{MS-figure}
\end{figure}

\end{example}

\section{Conclusions and perspectives}\label{sec:c}
We have shown through examples that the matrices arising from a DDM discretization often give rise to a GLT sequence, thus confirming for DDMs what is also observed for other classical discretization methods such as finite differences, finite elements, isogeometric analysis, etc.
Subsequently, we have formally defined, for an arbitrary multilevel block matrix, the BJ/BGS/AS/MS preconditioners commonly employed in the context of DDMs.
Our definitions, as well as the associated notations, were inspired by the theory of GLT sequences and are proposed as alternatives to those commonly used by the DDM community.
In our main results (Theorems~\ref{blockdiag(GLT)=GLT_blocktril(GLT)=GLT} and~\ref{AS(GLT)=GLT_MS(GLT)=GLT}), we have shown that any sequence of BJ/BGS/MS preconditioners associated with a GLT sequence $\{A_n\}_n\sim_{\rm GLT}\kappa$ is again a GLT sequence with the same symbol $\kappa$; and any sequence of AS preconditioners associated with a GLT sequence $\{A_n\}_n\sim_{\rm GLT}\kappa$ is a GLT sequence with symbol $\kappa^{AS}\approx\kappa$ such that $\kappa^{AS}=\kappa$ if and only if the overlaps in the subdomains used for the construction of the considered AS preconditioners vanish as $n\to\infty$.
Based on Theorems~\ref{blockdiag(GLT)=GLT_blocktril(GLT)=GLT} and~\ref{AS(GLT)=GLT_MS(GLT)=GLT}, we have predicted in Remark~\ref{P-efficiency} that the BJ/BGS/AS/MS preconditioners for matrices $A_n$ belonging to a GLT sequence are efficient. A numerical validation of this prediction in the context of isogeometric DDMs has been presented in Example~\ref{exa:num}.

We conclude this paper by suggesting some possible future lines of research.
\begin{enumerate}[nolistsep,leftmargin=*]

\item The definition and GLT analysis of the BJ/BGS/AS/MS preconditioners presented in this paper apply to multilevel block matrices only.
This represents a non-negligible limitation.
By definition, the size of a multilevel block matrix is constrained to be of the form $N(\nn)s\times N(\nn)t$ for some multi-index $\nn$ and some positive integers $s,t$.
As highlighted in Remark~\ref{mbm-rectangular}, multilevel block matrices usually arise from the discretization of multidimensional differential problems defined over rectangular domains. When the domain is not rectangular, the multilevel block structure is normally lost as well as the ``classical'' GLT structure. For non-rectangular domains, the resulting discretization matrices often show a ``reduced'' GLT structure 
\cite{rg}. An interesting subject for future investigation consists in extending the definition and GLT analysis of the BJ/BGS/AS/MS preconditioners to such matrices. 
This extension 
is supposed to heavily rely on the theory of reduced GLT sequences, which was originally introduced in \cite[pp.~398--399]{glt_1} and \cite[Section~3.1.4]{glt_2}, was systematically developed in \cite{rg}, and was actually an important source of inspiration for this paper too.

\item This paper focused on the GLT analysis of the BJ/BGS/AS/MS preconditioners only.
Two further widely used preconditioners in the context of DDMs are the restricted additive Schwarz (RAS) and restricted multiplicative Schwarz (RMS) preconditioners \cite{Cai-pioneering,EG,FS2001,NS2002}.
More recently, optimized additive Schwarz (OAS) and optimized multiplicative Schwarz (OMS) methods (and preconditioners) have been proposed to the DDM community \cite{Gan06,GLS12}.
Generalizing the GLT analysis of this paper to the RAS/RMS/OAS/OMS preconditioners is a natural direction of future research.

\item A setting closely related to the DDM framework considered in this paper 
is that of domain truncation, where an unbounded domain is replaced by a bounded computational domain via artificial boundary conditions, such as absorbing boundary conditions (ABCs) or perfectly matched layers (PMLs).
Indeed, when a domain is decomposed into subdomains solved independently, each subdomain problem can be viewed as posed on a truncated domain whose artificial boundary is the interface with its neighbors. Techniques such as ABCs and PMLs 
naturally motivate transmission conditions for Schwarz methods and, as shown in~\cite{gander_jakabcin_outrata}, after eliminating the artificial-region unknowns, they yield a Schur complement whose form depends on the chosen boundary conditions.
Performing a GLT analysis of the preconditioners arising from ABC/PML-based domain truncation 
is proposed as a topic for future research.
\end{enumerate}

\section*{Acknowledgements}

{\footnotesize
The first and third authors are members of the research group GNCS (Gruppo Nazionale per il Calcolo Scientifico) of INdAM (Istituto Nazionale di Alta Matematica).
This work was supported
by the INdAM-GNCS project DETECTED (aDaptive and lEarning-assisTed mEthods for numeriCal inTEgration and Data, CUP E53C25002010001),
by the MUR excellence department project MatMod@TOV awarded to the Department of Mathematics of the University of Rome Tor Vergata (CUP E83C23000330006),
by the Department of Mathematics of the University of Rome Tor Vergata through the project METRO (Methods and modEls for arTificial neuRal netwOrks, CUP E83C25000630005),
and by the PRIN-PNRR project MATHPROCULT (MATHematical tools for predictive maintenance and PROtection of CULTural heritage, Code P20228HZWR, CUP J53D23003780006).
}


\begin{thebibliography}{99}

\footnotesize

\bibitem{blocking}
{\sc Adriani A., Furci I., Garoni C., Serra-Capizzano S.}
{\em Spectral and singular value distribution of sequences of block matrices with rectangular Toeplitz blocks. Part I: Asymptotically rational block size ratios.}
J. Numer. Math. (in press) https://doi.org/10.1515/jnma-2025-0091.

\bibitem{rg}
{\sc Barbarino G.}
{\em A systematic approach to reduced GLT.}
BIT Numer. Math. 62 (2022) 681--743.

\bibitem{bgR}
{\sc Barbarino G., Garoni C., Mazza M., Serra-Capizzano S.}
{\em Rectangular GLT sequences.}
Electron. Trans. Numer. Anal. 55 (2022) 585--617.

\bibitem{bg}
{\sc Barbarino G., Garoni C., Serra-Capizzano S.}
{\em Block generalized locally Toeplitz sequences: theory and applications in the unidimensional case.}
Electron. Trans. Numer. Anal. 53 (2020) 28--112.

\bibitem{bgd}
{\sc Barbarino G., Garoni C., Serra-Capizzano S.}
{\em Block generalized locally Toeplitz sequences: theory and applications in the multidimensional case.}
Electron. Trans. Numer. Anal. 53 (2020) 113--216.


\bibitem{Ber94}
{\sc B\'erenger J.-P.}
{\em A perfectly matched layer for the absorption of electromagnetic waves.}
J. Comput. Phys. 114 (1994) 185--200.

\bibitem{Bini}
{\sc Bini D. A., Capovani M., Menchi O.}
{\em Metodi Numerici per l'Algebra Lineare.}
Zanichelli, Bologna (1988).

\bibitem{Brezis}
{\sc Brezis H.}
{\em Functional Analysis, Sobolev Spaces and Partial Differential Equations.}
Springer, New York (2011).

\bibitem{Cai-pioneering}
{\sc Cai X.-C., Sarkis M.}
{\em A restricted additive Schwarz preconditioner for general sparse linear systems.}
SIAM J. Sci. Comput. 21 (1999) 792--797.

\bibitem{CM94}
{\sc Chan T. F., Mathew T. P.}
{\em Domain decomposition algorithms.}
Acta Numer. 3 (1994) 61--143.

\bibitem{DDM-book-New}
{\sc Dolean V., Jolivet P., Nataf F.}
{\em An Introduction to Domain Decomposition Methods: Algorithms, Theory, and Parallel Implementation.}
SIAM, Philadelphia (2015).


\bibitem{EG}
{\sc Efstathiou E., Gander M. J.}
{\em Why restricted additive Schwarz converges faster than additive Schwarz.}
BIT Numer. Math. 43 (2003) 945--959.

\bibitem{EM77}
{\sc Engquist B., Majda A.}
{\em Absorbing boundary conditions for numerical simulation of waves.}
Proc. Natl. Acad. Sci. U.S.A. 74 (1977) 1765--1766.



\bibitem{FS2001}
{\sc Frommer A., Szyld D. B.}
{\em An algebraic convergence theory for restricted additive Schwarz methods using weighted max norms.}
SIAM J. Numer. Anal. 39 (2001) 463--479.

\bibitem{Gan06}
{\sc Gander M. J.}
{\em Optimized Schwarz methods.}
SIAM J. Numer. Anal. 44 (2006) 699--731.

\bibitem{Gan08}
{\sc Gander M. J.}
{\em Schwarz methods over the course of time.}
Electron. Trans. Numer. Anal. 31 (2008) 228--255.

\bibitem{gander_jakabcin_outrata}
{\sc Gander M. J., Jakabcin L., Outrata M.}
{\em Domain truncation, absorbing boundary conditions, Schur complements, and Pad\'e approximation.}
Electron. Trans. Numer. Anal. 59 (2024) 319--341.

\bibitem{GLS12}
{\sc Gander M. J., Loisel S., Szyld D. B.}
{\em An optimal block iterative method and preconditioner for banded matrices with applications to PDEs on irregular domains.}
SIAM J. Matrix Anal. Appl. 33 (2012) 653--680.


\bibitem{GLT-intro}
{\sc Garoni C.}
{\em Introduction to the theory of generalized locally Toeplitz sequences and its applications.}
Math. Notes 119 (2026) 832--873.

\bibitem{GLTbookI}
{\sc Garoni C., Serra-Capizzano S.}
{\em Generalized Locally Toeplitz Sequences: Theory and Applications (Volume I).}
Springer, Cham (2017).

\bibitem{GLTbookII}
{\sc Garoni C., Serra-Capizzano S.}
{\em Generalized Locally Toeplitz Sequences: Theory and Applications (Volume II).}
Springer, Cham (2018).

\bibitem{aip}
{\sc Garoni C., Serra-Capizzano S.}
{\em Multilevel generalized locally Toeplitz sequences: an overview and an example of application.}
AIP Conf. Proc. 2116 (2019) 020003.

\bibitem{GLH}
{\sc Garoni C., Serra-Capizzano S.}
{\em Block Jacobi/Gauss--Seidel preconditioning for GLT sequences, and GLH sequences.}
arXiv:2606.01888.

\bibitem{GV}
{\sc Golub G. H., Van Loan C. F.}
{\em Matrix Computations.}
4th Edition, The Johns Hopkins University Press, Baltimore (2013).


\bibitem{Lio88I}
{\sc Lions P.-L.}
{\em On the Schwarz alternating method. I.}
In ``First International Symposium on Domain Decomposition Methods for Partial Differential Equations'', 
SIAM, Philadelphia (1988) 1--42.

\bibitem{Lio88II}
{\sc Lions P.-L.}
{\em On the Schwarz alternating method. II: Stochastic interpretation and order properties.}
In ``Domain Decomposition Methods'', 
SIAM, Philadelphia (1989) 47--70.

\bibitem{Lio88III}
{\sc Lions P.-L.}
{\em On the Schwarz alternating method. III: A variant for nonoverlapping subdomains.}
In ``Third International Symposium on Domain Decomposition Methods for Partial Differential Equations'', 
SIAM, Philadelphia (1990) 202--223.

\bibitem{LMS}
{\sc Lyche T., Manni C., Speleers S.}
{\em Foundations of spline theory: B-splines, spline approximation, and hierarchical refinement.}
In ``Splines and PDEs: From Approximation Theory to Numerical Linear Algebra'', 
Lect. Notes Math.~2219, Springer, Cham (2018) 1--76.


\bibitem{NS2002}
{\sc Nabben R., Szyld D. B.}
{\em Convergence theory of restricted multiplicative Schwarz methods.}
SIAM J. Numer. Anal. 40 (2003) 2318--2336.



\bibitem{DDM-GLT-1d}
{\sc Rifqui A., Ratnani A., Serra-Capizzano S.}
{\em Block Schwarz methods and preconditioning strategies using generalized locally Toeplitz tools -- Part I: Analysis of the preconditioners and numerical validation.}
Submitted.

\bibitem{Rudinone}
{\sc Rudin W.}
{\em Real and Complex Analysis.}
3rd Edition, McGraw-Hill, Singapore (1987).

\bibitem{Saad}
{\sc Saad Y.}
{\em Iterative Methods for Sparse Linear Systems.}
2nd Edition, SIAM, Philadelphia (2003).

\bibitem{Sch1870}
{\sc Schwarz H. A.}
{\em Ueber einen Grenz\"ubergang durch alternirendes Verfahren}
(1870).

\bibitem{glt_1}
{\sc Serra-Capizzano S.}
{\em Generalized locally Toeplitz sequences: spectral analysis and applications to discretized partial differential equations.}
Linear Algebra Appl. 366 (2003) 371--402.

\bibitem{glt_2}
{\sc Serra-Capizzano S.}
{\em The GLT class as a generalized Fourier analysis and applications.}
Linear Algebra Appl. 419 (2006) 180--233.




\bibitem{DDM-book}
{\sc Toselli A., Widlund O.}
{\em Domain Decomposition Methods -- Algorithms and Theory.}
Springer, Berlin, Heidelberg (2005).

\bibitem{Ty96}
{\sc Tyrtyshnikov E. E.}
{\em A unifying approach to some old and new theorems on distribution and clustering.}
Linear Algebra Appl. 232 (1996) 1--43.

\end{thebibliography}
\end{document}